\documentclass[11pt]{amsart}

\usepackage[utf8]{inputenc}
\usepackage[T1]{fontenc}

\usepackage[bookmarks,bookmarksnumbered,pdfstartview=FitBH,backref=page]{hyperref}
\hypersetup{
colorlinks=true,
citecolor= blue,
linktoc= all,
linkcolor= blue,
urlcolor= green,
}

\usepackage{amsfonts, amssymb, amsmath, amsthm, enumerate,appendix,todo}

\usepackage{tikz}
\usetikzlibrary{decorations.pathreplacing,decorations.markings}

\usepackage[margin=1.5in]{geometry}

\theoremstyle{definition}
\newtheorem{thmintro}{Theorem}

\newtheorem{definition}{Definition}[section]
\newtheorem{remark}[definition]{Remark}
\newtheorem{example}[definition]{Example}
\theoremstyle{plain}
\newtheorem{theorem}[definition]{Theorem}
\newtheorem{thm}[definition]{Theorem}
\newtheorem{conj}[definition]{Conjecture}
\newtheorem{proposition}[definition]{Proposition}
\newtheorem*{proposition*}{Proposition}

\newtheorem{lemma}[definition]{Lemma}
\newtheorem{corollary}[definition]{Corollary}
\newtheorem{question}[definition]{Question}

\newtheorem*{claim}{Claim}

\newcommand{\actson}{\curvearrowright}

\newcommand{\from}{\colon}

\newcommand{\mc}[1]{\mathcal{#1}}

\newcommand{\ra}{\rightarrow}

\newcommand{\suchthat}{\colon}

\newcommand{\FF}{\mathbf{F}}

\newcommand{\subrel}{\subseteq}

\newcommand{\inters}{\cap}

\newcommand{\disjointunion}{\sqcup}
\newcommand{\bigunion}{\bigcup}
\newcommand{\restrict}{\restriction}

\DeclareMathOperator{\Isom}{Isom}
\DeclareMathOperator{\SL}{SL}
\DeclareMathOperator{\PSL}{PSL}

\DeclareMathOperator{\dom}{dom}

\DeclareMathOperator{\Aut}{Aut}

\newcommand{\boundary}{\partial}

\usepackage{bm}
\newcommand{\R}{\bm{\mathrm{R}}}
\newcommand{\N}{\bm{\mathrm{N}}}
\newcommand{\Z}{\bm{\mathrm{Z}}}
\newcommand{\F}{\bm{\mathrm{F}}}

\newcommand{\HH}{\bm{\mathrm{H}}}
\newcommand{\C}{\bm{\mathrm{C}}}

\renewcommand{\emptyset}{\varnothing}

\newcommand{\ol}[1]{\overline{#1}}
\newcommand{\wh}[1]{\widehat{#1}}

\newcommand{\Ggraph}{\mathcal{G}}
\newcommand{\Hgraph}{\mathcal{H}}
\newcommand{\Kgraph}{\mathcal{K}}
\newcommand{\Tgraph}{\mathcal{T}}

\usepackage[mathscr]{eucal}
\newcommand{\BSTreeable}{\bm{\mathscr{BST}}}
\newcommand{\MSTreeable}{\bm{\mathscr{MST}}}
\newcommand{\STreeable}{\bm{\mathscr{ST}}}
\newcommand{\Treeable}{\bm{\mathscr{T}}}
\newcommand{\ClassTreeable}{\bm{\mathscr{C}}}

\usepackage{csquotes}
\MakeOuterQuote{"}

\newcommand\Hmath{\mathbf{H}}
\newcommand{\define}[1]{\textbf{#1}}
\newcommand{\pmp}{\text{p.m.p.}}

\newcommand\RR{\mathscr{R}}

\newcommand\GGG{\mathtt{G}}
\newcommand\VVV{\mathtt{V}}
\newcommand\EEE{\mathtt{E}}

\newcommand\GG{\mathcal{G}}

\newcommand\TT{\mathcal{T}}
\newcommand\TTT{\mathtt{T}}
\newcommand\CC{\mathcal{C}}
\newcommand{\Dual}{D}

\tikzset{every picture/.style={line width=0.11mm}}
\newcommand{\oPerpStar}{\begin{tikzpicture}[scale=0.134]
\draw[thick,line width=0.7pt](-0.6,-0.2)--(0.6,-0.2) ;
    \draw (0,0.1) node[scale=0.6] {\rm{*}};
  \draw[line width=0.7pt] (0,0) circle [radius=1];
\end{tikzpicture}}

\newcommand{\oPerpstar}{\mathbin{\raisebox{-1pt}{\oPerpStar}}}

\newcommand{\mysetminusD}{\hbox{\tikz{\draw[line width=0.6pt,line cap=round] (3pt,0) -- (0,6pt);}}}
\newcommand{\mysetminusT}{\mysetminusD}
\newcommand{\mysetminusS}{\hbox{\tikz{\draw[line width=0.45pt,line cap=round] (2pt,0) -- (0,4pt);}}}
\newcommand{\mysetminusSS}{\hbox{\tikz{\draw[line width=0.4pt,line cap=round] (1.5pt,0) -- (0,3pt);}}}
\newcommand{\mysetminus}{\mathbin{\mathchoice{\mysetminusD}{\mysetminusT}{\mysetminusS}{\mysetminusSS}}}%according to the maths size

\begin{document}

\title{One-ended spanning subforests and treeability of groups}

\author[Conley \and Gaboriau \and Marks \and Tucker-Drob]{Clinton T.~Conley \and Damien Gaboriau \and Andrew S.~Marks \and Robin D.~Tucker-Drob}
\date{March 12, 2026}
\thanks{}
\begin{abstract}
  We show that several new classes of groups are measure strongly treeable. In
  particular, finitely generated groups admitting planar Cayley graphs,
  elementarily free groups, and $\Isom(\HH^2)$ and all its closed subgroups. This provides the first examples of one-ended nonamenable groups which are measure strongly treeable. In
  higher dimensions, we also prove a dichotomy that the
  fundamental group of a closed aspherical $3$-manifold is either amenable or
  has strong ergodic dimension $2$. Our main technical tool is a method
  for finding measurable treeings of Borel planar graphs by constructing
  one-ended spanning subforests in their planar dual.
  Our techniques for constructing one-ended spanning subforests also give a
  complete classification of the locally finite \pmp{} graphs which
  admit Borel a.e.\ one-ended spanning subforests.
\end{abstract}

\maketitle

\tableofcontents

\noindent
\textbf{MSC:}
37A20 (03E15, 22F10, 05C10)
%37A20  Orbit equivalence, cocycles, ergodic equivalence relations\\
%03E15  Descriptive set theory
%22F10  Measurable group actions
%05C10 Planar graphs

\noindent
\textbf{Keywords:}
{ergodic and Borel equivalence relations, probability measure preserving and Borel actions of groups, locally compact groups, trees, treeability, planar graphs, elementarily free groups, cost of groups, ergodic dimension, measure equivalence}

\addtocounter{section}{-1}
\section{Introduction}

This article is a contribution to the study of measured and Borel equivalence relations, in terms of their graphed structures, with applications in the measured group theory of countable and locally compact groups.

Dramatic progress has been realized in the study of discrete groups in relation with topological and geometric ideas over the course of the 20th century, from the early works of Klein, Poincar\'e, Dehn, Nielsen, Reidemeister and Schreier for instance, to Bass-Serre theory and Thurston's Geometrization program as well as hyperbolic groups and the emergence of geometric group theory as a distinct area of mathematics under the impulse of Gromov.
In his monograph \cite{Gro93}, Gromov outlined his program of understanding countable discrete groups up to quasi-isometry (e.g., cocompact lattices in the same locally compact second countable group $G$).
In the same text Gromov also introduced the parallel notion of measure equivalence (\define{ME}) between countable discrete groups \cite[0.5.E]{Gro93}, the most emblematic example being lattices in $G$.
Two groups are ME if they admit commuting, free, measure-preserving actions on a nonzero Lebesgue measure space with finite measure fundamental domains.
This concept is strongly connected with orbit equivalence (\define{OE}) in ergodic theory (\cite{Furman1999}, \cite[Th. 2.3]{Gab02b}; see \cite{Ga05,Furman2011} for surveys on ME and OE).

The history of orbit equivalence itself can be traced back to the work of Dye \cite{Dye59,Dye63} stemming from the group-measure-space von Neumann algebra of Murray and von Neumann \cite{MvN36}. The abstract and basic objects connecting this turn out to be the standard measure-class preserving equivalence relations, as axiomatized by Feldman-Moore \cite{FM77}.
A major milestone is the elucidation of the connections between five properties (see \cite{Connes-1976,Connes-Krieger-1977,OW80, CFW81}):
{\em the following are equivalent: (1)  hyperfiniteness of the group-measure-space von Neumann algebra, (2)  hyperfiniteness of the equivalence relation, (3) amenability of the equivalence relation, (4) orbit equivalence with a $\Z$-action, and (5) when the action is assumed to be probability measure preserving (\define{\pmp}) and free, amenability of the acting group.}
As a consequence, the measure equivalence class of $\Z$ consists exactly in all infinite amenable groups. Thus such a useful geometric invariant as the growth becomes apparently irrelevant in measured group theory insofar as amenability is concerned, although our Theorem~\ref{thm:quad} leads us to reconsider this observation.

Much of progress in orbit equivalence has been realized since the 80's following a suggestion of A.~Connes at a conference in Santa Barbara in 1978 (see  \cite{Ada90}) of studying equivalence relations $\RR$ with an additional piece of data: a measurably-varying simplicial complex structure on each equivalence class (aka a \define{complexing} \cite{Abert-Gab}). The $1$-dimensional complexings are known as \define{graphings}. Their acyclic version (\define{treeings}) were originally studied by S. Adams \cite{Ada90,Ada88}. Both are constitutive of the theory of \define{cost} \cite{Le95,Gab00a}, since this is defined in terms of graphings, and treeings allow it to be computed since treeings of $\RR$ attain the cost of $\RR$ \cite[Th\'eor\`eme 1]{Gab00a}. Graphings and treeings have also played a crucial role in the theory of structurings on countable Borel equivalence relations \cite{JKL02}.

Amenability, seen from the perspective of orbit equivalence, can be rephrased as the capability of embellishing almost every orbit (equivalence class) with a measurably-varying oriented line structure \cite{Dye59,OW80,CFW81}.
Alternatively, it is easy to equip the classes of any hyperfinite equivalence relations with a \define{one-ended} tree structure. As a kind of converse, in the {\pmp} context (which will be our context in the introduction through Theorem \ref{thm:lastpmp}) any treeing of an amenable equivalence relation is (class-wise) at most two-ended (\cite{Ada90}).

Beyond amenability, the simplest groups from the measured theoretic point of view are the \define{treeable} ones: those admitting a free {\pmp} action whose orbit equivalence relation can be equipped with a treeing\footnote{For precise definitions of the various notions of treeability, see Appendices~\ref{sect: Treeability for locally compact groups} and \ref{sec:permanence}}.
This is an extremely rich and still mysterious class of groups (see the survey part of \cite{Ga05}).
By a theorem of Hjorth \cite{Hjo-cost-att}, this is precisely the class of groups $\Gamma$ that are ME with a free group $\FF_n$. This family splits into four ME-classes: $n=0$ when $\Gamma$ is finite, $n=1$ when $\Gamma$ is infinite amenable, and $n=2$ or $n=\infty$ according to their cost belonging to $(1,\infty)$ or $\{\infty\}$ \cite{Gab00a}.

The first substantial example of a treeable group apart from free products of amenable groups is the fundamental group $\pi_1(\Sigma)$ of a closed hyperbolic surface $\Sigma$.
Indeed, both $\pi _1 (\Sigma )$ and $\FF_2$ share the property of being isomorphic to lattices in $G=\mathrm{SL}(2, \R)$.
It follows that $\pi_1(\Sigma)$ admits at least one treeable free action, namely the natural action by multiplication $\pi_1(\Sigma)\actson G/\FF _2$ (with Haar measure).
It is a longstanding question of  \cite[Question VI.2]{Gab00a} whether treeable groups are \define{strongly treeable}\footnote{Observe that unfortunately for consistency of terminology, in \cite{Gab00a} the terms ``arborable'' and ``anti-arborable''
are used instead of the current better terms ``strongly treeable'' and ``non-treeable'' respectively, that we will adopt here.} ($\STreeable$), i.e., whether \textbf{all} their free {\pmp} actions are treeable. This question has been open for twenty years, even for $\pi_1(\Sigma)$ (see \cite[Question p. 176]{Gab02b}), and we solve it in this case. This is our first main result:
\begin{thmintro}
\label{th-intro surface + planar Cayley graphs}
Surface groups are strongly treeable.
More generally, finitely generated groups admitting a planar Cayley graph are strongly treeable.
\end{thmintro}

A more general statement can be found in Theorem \ref{thm:planargroup}.

The introduction of the notion of "measurable free factor" in \cite{Ga05} led to the production of some new examples of treeable groups, such as branched surface groups. These are examples from a family that we will discuss now. The \define{elementarily free groups} are those groups with the same first-order theory as the free group $\FF_2$.
We shall use their description (when finitely generated) as fundamental groups of certain tower spaces (according to the results of \cite{SeVI} and \cite{KhM1},
made utterly complete in \cite{Guirardel-Levitt-Sklinos-2020} -- see \S\ref{sec:elemfree} for more details).
A careful analysis of their virtual structure allowed \cite{BTW-07} to apply results from \cite{Ga05} to a finite index subgroup in order to achieve their treeability. The question of their strong treeability has remained open since then; we resolve it:
\begin{thmintro}
\label{th-intro lement free groups}
Finitely generated elementarily free groups are strongly treeable.
\end{thmintro}
 Strong treeability has a number of consequences which do not follow from treeability. In particular, if $\Gamma$ is a strongly treeable group, then by \cite[Prop. VI.21]{Gab00a}
$\Gamma$ satisfies the fixed price conjecture \cite[Question I.8]{Gab00a}. The groups having a planar Cayley graph appearing in Theorem~\ref{th-intro surface + planar Cayley graphs}, such as the cocompact Fuchsian triangle groups, give the first new examples of groups of fixed price greater than $1$ since \cite{Gab00a}.
(The cocompact Fuchsian triangle groups admit finite index subgroups which are surfaces, but both strong treeability and fixed price are not known to pass to finite index super-groups.)

The arguments developed for the above theorem gave us as a by-product the following interesting claim (Corollary~\ref{cor: a one-rel amalg product}).
Let $r\geq 3$ and $\Gamma_1, \Gamma_2, \cdots, \Gamma_r$ be countable groups and let $\gamma_i\in \Gamma_i$ be an infinite order element for each $i=1, 2, \cdots, r$.
If the $\Gamma_i$ are all treeable or strongly treeable, then the same holds not only for their free product, but also for its quotient by the normal subgroup generated by the product of the $\gamma_i$:
$\left(\Gamma_1*\Gamma_2*\cdots *\Gamma_r\right)/\langle\langle  \prod_{i=1}^r \gamma_i\rangle\rangle$.

It is worth mentioning that treeability has also had an impact in the theory of von Neumann algebras. Popa's discovery \cite{Pop06a} of the first II$_1$ factor with trivial fundamental group (namely the group von Neumann algebra $L(\mathrm{SL}(2,\Z)\ltimes \Z^2)$)
 used his rigidity-deformation theory to establish uniqueness of the HT Cartan subalgebra, thereby reducing the problem to the study of the orbit equivalence relation of the treeable action of $\mathrm{SL}(2,\Z)$ on the $2$-torus. Later on, Popa and Vaes extended drastically the class of groups whose free \pmp{} actions lead to uniqueness of the Cartan subalgebra \cite{Popa-Vaes-2014-unique-Cartan-Fn, Popa-Vaes-2014-unique-Cartan-hyperbolic-gp}, and thus the study of the group-measure space von Neumann algebra of these actions boils down to that of the action up to OE. This class contains free groups, non-elementary hyperbolic groups, their direct products and all groups that are ME with these groups.
 The class of countable groups satisfying 2-cohomology vanishing for cocycle actions on II$_1$ factors is speculated to coincide with the class of  treeable groups
 \cite[Remarks 4.5]{Popa2018-vanishing}.

 However, it is far from the case that all countable groups are treeable. The first examples of non-treeable groups are the infinite Kazhdan property (T) groups \cite{AS90} and the non-amenable cost $1$ groups \cite[Th. 4]{Gab00a}, and more generally all non-amenable groups with $\beta_1^{(2)}=0$ \cite[Proposition 6.10]{Gab02}.
The random graph formulation of treeability is the existence of an invariant probability measure supported on the set of spanning trees on the group; in \cite{PP00}, Pemantle and Peres prove that a non-amenable direct product of infinite groups can have no such probability measure. The non treeability of non-amenable direct products also follows from the theory of cost \cite{Gab00a}.

 The study of non-treeable groups must involve higher dimensional geometric objects.
Recall that the \define{geometric dimension} of a countable group $\Gamma$ is the smallest dimension of a contractible complex on which $\Gamma$ acts freely. Analogously,  the second author introduced in \cite[D\'ef. 3.18]{Gab02} the \define{geometric dimension} of a \pmp{} standard equivalence relation $\RR$ on a probability measure space $(X,\mu)$ as the smallest dimension of a measurable bundle of contractible simplicial complexes over $X$ on which $\RR$ acts \define{smoothly}, i.e., for which there exists a Borel fundamental domain for the action of $\RR$.
The \define{ergodic dimension of a group $\Gamma$} \cite[D\'ef. 6.4]{Gab02} is the smallest geometric dimension among all of its free \pmp{} actions  (see \cite{Gab-erg-dim} for more on this notion).
Being infinite treeable is thus a synonym of having ergodic dimension~$1$.
The ergodic dimension of a group is an ME-invariant, it is bounded above by its virtual geometric dimension, and it is non-increasing when taking subgroups.

We say that $\Gamma$ has \define{strong ergodic dimension $d$} if all its free \pmp{} actions
have geometric dimension $d$. Honesty and humility force us to admit our ignorance: not a single group is known with two free \pmp{} actions
having different geometric dimensions.

By Ornstein-Weiss \cite{OW80}, all infinite amenable groups have strong ergodic dimension $1$.
Recall that non vanishing of $\ell^2$-Betti numbers produces lower bounds for the ergodic dimension (see \cite[Prop. 5.8, Cor. 3.17]{Gab02}). It follows for instance that (with $p_i\geq 2$)  $\Gamma=\FF_{p_1}\times \FF_{p_2}\times \cdots \times \FF_{p_d}$ or $\Gamma=(\FF_{p_1}\times \FF_{p_2}\times \cdots \times \FF_{p_d})*  \FF_{k}$ have strong ergodic dimension $d$  \cite[pp. 126-127]{Gab02} while $\Gamma\times \Z$ has strong ergodic dimension $d+1$  and $\mathrm{Out}(\FF_n)$ has strong ergodic dimension $2n-3$
\cite[Theorem 1.6, Theorem 1.1]{Gab-top-dim}.

It is not hard to check that the $3$-dimensional manifolds with one of the eight geometric structures of Thurston have ergodic dimension at most $2$.
For instance the fundamental group of a closed hyperbolic $3$-dimensional manifold is a cocompact lattice in $\mathrm{SO}(3,1)$. It is ME with non-compact lattices in $\mathrm{SO}(3,1)$, which have geometric dimension $2$ and zero first $\ell^2$-Betti number, and thus they have ergodic dimension at most $2$.
We prove a strong dichotomy theorem for the ergodic dimension of fundamental
groups of aspherical manifolds of dimension $3$.
\begin{thmintro}[Theorem~\ref{erg dim 3-dim mfld}]
\label{th intro: closed 3-dim dichotomy}
 Suppose $\Gamma$ is the fundamental group of a closed (i.e., compact without boundary) aspherical (possibly non-orientable) manifold of
  dimension $3$. Then either
  \begin{enumerate}
  \item $\Gamma$ is amenable, or
  \item $\Gamma$ has strong ergodic dimension $2$.
  \end{enumerate}
\end{thmintro}

If one removes the assumption of asphericity, the Kneser-Milnor theorem  \cite{Kneser-1929,Milnor-1962} decomposes the fundamental group of an orientable closed $3$-dimensional
manifold as a free product of amenable groups and groups to which Theorem~\ref{th intro: closed 3-dim dichotomy} applies. Since the geometric dimension of a free p.m.p.\ action of a free product of infinite groups is the maximum of the geometric dimensions of the actions of the factors, 
it follows that:
\begin{thmintro}
  If $M$ is an orientable closed $3$-dimensional manifold, then $\Gamma=\pi_1(M)$ has strong ergodic dimension $\in\{0,1,2\}$.
\end{thmintro}
Considering the orientation covering, one deduces that the ergodic dimension is at most $2$ if $M^3$ is non-orientable.
However, strongness also holds in this case, but this shall be treated elsewhere \cite{Gab-erg-dim}. This touches the delicate open question whether strong ergodic dimension is an invariant of commensurability.
If one allows $M$ to have boundary components, then we lose a priori the strongness but we still obtain that every \pmp{} free action of $\Gamma=\pi_1(M)$ has geometric dimension at most $2$.

It is worth noting that all our results are proved without appealing to the recent progress in Thurston's geometrization theorem.
In higher dimensions, we also obtain a non-trivial bound :
\begin{thmintro}[Theorem~\ref{erg dim d-mfld}]
  Suppose $\Gamma = \pi_1(M)$ is the fundamental group of a compact aspherical manifold
  $M$ (possibly with boundary) of
  dimension at least $2$. Then all free \pmp{} actions of $\Gamma$ have
  ergodic dimension at most $\dim(M) - 1$.
\end{thmintro}

The theory of measure equivalence has been extended beyond countable discrete groups to include all unimodular locally compact second countable (\define{lcsc}) groups (see the nice survey \cite{Furman2011} and see \cite{KKR-2017} for basic invariance properties). The investigation about their treeability began with a result of Hjorth \cite[Theorem 0.5]{Hj08b} stating that the products $G_1\times G_2$ of infinite lcsc groups are non treeable unless both are amenable.
He observed that amenable groups are strongly treeable and asks which other lcsc groups satisfy this property  \cite[p. 387]{Hj08b}.
We produce the first progress in this study since then:
\begin{thmintro}[See Corollary~\ref{cor:subgroup}]
\label{th-intro Isom H2, PSL 2 are str. treeable}
$\Isom(\HH ^2)$, $\PSL_2(\R )$, and $\SL_2(\R )$ are all strongly treeable, as are all of their closed subgroups.
\end{thmintro}

While the notion of treeability extends in the natural way to orbit equivalence relations of actions $G\actson (X,\mu)$ of lcsc groups, an equivalent way of conceiving a treeing in this context is by introducing a cross section $B\subseteq X$ (see section~\ref{sect: Treeability for locally compact groups}) to which the restriction $\RR_{\restriction B}$ of the orbit equivalence relation $\RR_{G\actson (X,\mu)}$
has countable classes and is treeable.

This result gives as a by-product the first examples of non trivial fixed price for connected lcsc groups (Definition \ref{def:lcsccost}).  In contrast, fixed price $1$ for the direct product of some lcsc groups with the integers is obtained in \cite{AM-21}.  Once a Haar measure is prescribed on $G$, the quantity $\frac{\mathrm{cost}(\RR_{\restrict B})-1}{\mathrm{covolume}(B)}$ does not depend on the cross section $B$, since the restrictions are pairwise stably orbit equivalent (Proposition \ref{prop:covol}).  A remarkable consequence of Theorem \ref{th-intro Isom H2, PSL 2 are str. treeable} is that this quantity is also independent of the free \pmp{} action $G\actson (X,\mu)$, for each of these groups:
\begin{thmintro}[See Corollary \ref{cor:subgroup} and Remark \ref{rem:lcsccost}]
The groups $G=\Isom(\HH ^2)$, $\PSL_2(\R )$, and $\SL_2(\R )$, and their closed subgroups have fixed price.
\end{thmintro}

A central fascination in our study is, given a graphing, the hunt for a subgraphing all of whose connected components are acyclic and have exactly one end (hereafter named a \define{one-ended spanning subforest}, since several connected components are usually necessary for covering a single class). Besides their intrinsic interest, one-ended spanning subforests prove to be extremely useful in our applications, e.g., Theorems \ref{thm:planar}, \ref{elemfreetreeable}, \ref{erg dim d-mfld}, and \ref{erg dim 3-dim mfld}.

Much of the technical work in the paper consists of finding new techniques for constructing Borel a.e.\ one-ended spanning subforests of locally finite Borel graphs.
In particular, in the case of locally finite \pmp{} graphs, we give a
complete characterization of what graphs admit Borel a.e.\ one-ended
spanning subforests.
\begin{thmintro}[Theorem~\ref{thm:pmp}]
\label{th: pmp: 1-ended subforest iff now 2-ended - intro}
  Suppose that ${\Ggraph}$ is a measure preserving aperiodic locally finite
  Borel graph on a standard probability space $(X,\mu)$. Then ${\Ggraph}$ has a Borel
  $\mu$-a.e.\ one-ended spanning subforest $\Tgraph\subseteq \Ggraph$ if and only if ${\Ggraph}$ is $\mu$-nowhere
  two-ended.
\end{thmintro}

Here, \define{$\mu$-a.e.\ one-ended spanning subforest} means that the set of vertices of the one-ended trees of $\Tgraph$ has full $\mu$-measure in $X$;
while \define{$\mu$-nowhere two-ended} means that the set of vertices of the two-ended connected components of $\Ggraph$ has measure zero in $X$.

The search for subtrees or subforests has attracted enormous attention in another but related mathematical field: the theory of random graphs and percolation (already alluded to in our introduction to non-treeable groups).
Thus, for example Pemantle \cite{Pemantle-1991-WSF} introduced the spanning forest FUSF for $\Z^d$, obtained as the
 limiting measure of the uniform spanning tree on large finite pieces of the lattice. He proved that it is connected if and only if $d\leq 4$. The use of various subforests such as the (wired and free) minimal spanning forests (WMSF, FMSF) or the (wired and free) uniform spanning (WUSF, FUSF) forests are of crucial significance in the theory of percolation on graphs \cite{BLPS01,LPS06}.
The WMSF is an instance of a random one-ended spanning subforest.
The authors of \cite{LPS06} have shown the equivalence of WMSF$\not=$FMSF with the famous conjecture of Benjamini-Schramm \cite{Benjamini-Schramm-96} whether $p_c\not=p_u$. The mean valency of the FUSF equals two plus twice the first $\ell^2$-Betti number \cite[Corollary 4.12]{Lyons-2009-random-complexes}.
If we knew that adding a random graph of arbitrarily small mean valency could make the FUSF forest connected, then it would solve the cost vs first $\ell^2$-Betti number question of  \cite[p. 129]{Gab02}. See also \cite{GL} and \cite{Ti19} for further connections between treeability and percolation.

Theorem \ref{th: pmp: 1-ended subforest iff now 2-ended - intro} relies on Elek and Kaimanovich's characterization of when a locally finite \pmp{} graph $\Ggraph$ is  \define{$\mu$-hyperfinite} (i.e., when there exists a $\mu$-conull subset $X_0$ of $X$ such that the connectedness equivalence relation $\RR_{\Ggraph_{\restriction{X_0}}}$ of the induced subgraph $\Ggraph_{\restriction{X_0}}$ is hyperfinite). On the contrary, we say $\Ggraph$ is \define{$\mu$-nowhere hyperfinite} if there does not exist a positive measure subset $A$ of $X$ such that $(\RR_{\Ggraph})_{\restriction A}$ is hyperfinite.

We also use Theorem~\ref{th: pmp: 1-ended subforest iff now 2-ended - intro} to give an interesting dual statement to
the well-known part (1) of the following theorem that a graph is
$\mu$-hyperfinite if and only if it has complete sections of arbitrarily large measure whose induced subgraphs are finite.
\begin{thmintro}[Theorem~\ref{thm:1percent}]\label{thm:lastpmp}
Let $\Ggraph$ be a locally finite p.m.p\ graph on a standard
probability space $(X,\mu )$.
\begin{enumerate}
  \item $\Ggraph$ is $\mu$-hyperfinite if and only if for
every $\epsilon > 0$, there exists a Borel complete section $A \subseteq X$ for
$\RR_\Ggraph$ with $\mu(A) > 1 - \epsilon$ so that $\Ggraph_{\restriction A}$ has
finite connected components.
  \item $\Ggraph$ is $\mu$-nowhere hyperfinite if and only if for every
  $\epsilon >0$ there exists a Borel complete section $A\subseteq X$ for
  $\RR_{\Ggraph}$ with $\mu (A)<\epsilon$ such that $\Ggraph_{\restriction A}$ is $\mu _{\restriction A}$-nowhere hyperfinite.
\end{enumerate}
\end{thmintro}

By significantly relaxing the measure preserving hypothesis, we arrived at the
study of Borel graphings and their behavior with respect to various Borel
probability measures that are not necessarily measure preserving. Although the
measure $\mu$ is not a priori related to the Borel graph $\Ggraph$ under consideration, some properties may hold up to discarding a set of $\mu$-measure $0$. In this non \pmp{} context, we obtain one-ended spanning subforests out of Borel planar graphs (see Definition~\ref{def:BorelPlanar}).
\begin{thmintro}[See Corollary~\ref{cor:planarforest}]
\label{th-intro: universal one-ended subforest}
Let ${\Ggraph}$ be a locally finite Borel graph on $X$ whose connected components are planar.
Let $\mu$ be a Borel probability measure on $X$. If
${\Ggraph}$ is $\mu$-nowhere two-ended, then $\Ggraph$ has a Borel a.e.\ one-ended spanning subforest.
\end{thmintro}
Without the "$\mu$-nowhere two-ended" assumption, one still gets the existence of a spanning subforest on a $\mu$-conull set (see Theorem \ref{thm:planar}).

We say that a Borel equivalence relation $\RR$ on $X$ is \define{measure
treeable} if for each Borel probability measure $\mu$ on $X$, there is a
$\mu$-conull subset $X_0$ of $X$ such that the restriction of $\RR$ to $X_0$ is
Borel treeable (Definition \ref{def:BorelTreeable}). Observe that when the
classes of $\RR$ are countable then each such $\mu$ is dominated by a
quasi-invariant measure $\mu ' = \sum_{i=1}^{\infty} 2^{-i} g_{i*}\mu$, where
$\{g_i\}_{i=1}^{\infty}$ is an enumeration of some countable group $G$ that
generates $\RR$ (see \cite{FM77}). Note that $\mu$ and $\mu '$ have the same
$\RR$-invariant null sets. See also Proposition \ref{prop:qinv}. Hence, we may assume without loss of generality
that our measures are quasi-invariant throughout the paper. For example, a
countable Borel equivalence relation
$\RR$ is measure treeable if and only if it is $\mu$-treeable for
every $\RR$-quasi-invariant measure $\mu$.

Perhaps surprisingly, by enlarging our perspective to the more general possibly non-\pmp{} setting, we obtain consequences in the purely \pmp setting whose proof uses non-\pmp{} techniques in an apparently essential way. Specifically, the more general setting provides access to the induced action construction even for subgroups of infinite index.
It is an easy observation (Lemma \ref{lem:ABC}) that: {\em If $A$, $B$, and $C$ are countable groups with $A\leq B\leq C$ and if $A$ is a measure strong free factor of $C$ (see Definition in Section~\ref{sect:Meas free fact}), then $A$ is a measure strong free factor of $B$}. Combining this with our techniques, we obtain in \S\ref{sec:cyclicmff} new measure strong free factors in free products of groups. This generalizes results of Alonso~\cite{Alonso-2014} (see Corollary \ref{cor:freegroupmsff}).

\medskip
An lcsc group $G$ is called \define{measure strongly treeable} (${\MSTreeable}$) if all orbit equivalence relations generated by all free Borel actions of $G$ are measure treeable (Definition~\ref{def:various treeab, groups}).

In this context, using Theorem~\ref{th-intro: universal one-ended subforest},
our Theorems \ref{th-intro surface + planar Cayley graphs},
\ref{th-intro lement free groups} and
\ref{th-intro Isom H2, PSL 2 are str. treeable} take indeed a much stronger non \pmp{} form (see
 Corollary~\ref{cor:subgroup}, Corollary~\ref{thm:planargroup},
and Theorem~\ref{elemfreetreeable}):
\begin{thmintro}
  The following groups are measure strongly treeable
  \begin{enumerate}
    \item $\Isom(\Hmath^2)$, $\PSL_2(\R )$, and $\SL_2(\R )$, and all of their closed subgroups.
    \item Finitely generated groups admitting a planar Cayley graph.
    \item Finitely generated elementarily free groups.
  \end{enumerate}
\end{thmintro}
Our proof follows an idea from \cite{BLPS01} (also used in \cite{Ga05}) of finding an a.e.\ one-ended spanning subforest
in the planar dual (see \S\ref{sec:planar}). Suppose $\Ggraph$ is a locally finite graph admitting an accumulation-point
free planar embedding into $\R^2$. Then it is easy to see that subtreeings of $\Ggraph$
are in one-to-one correspondence with one-ended subforests in the planar
dual of $\Ggraph$. We use this correspondence to show the measure
treeability of the above groups $G$ by converting the problem of treeing an
action of $G$ into finding an a.e.\ one-ended spanning subforest of the planar dual
of graphings of the action which are Borel planar.

\bigskip
\noindent
{\bf Acknowledgments:}
The authors would like to thank
Agelos Georgakopoulos for helpful conversations about planar Cayley graphs and
Vincent Guirardel for useful discussions about elementarily free groups.
We are grateful to Gilbert Levitt for pointing out the reference \cite{Guirardel-Levitt-Sklinos-2020}, for explaining its content, and for discussing with us the proof of Theorem~\ref{th: elem free gps as large towers}, after a preliminary version of our article has circulated.
We also thank the referee for their detailed and helpful comments.

CC was supported by NSF grants DMS-1855579 and DMS-2154160. DG was supported by the ANR project
GAMME (ANR-14-CE25-0004) and by the LABEX MILYON (ANR-10-LABX-0070) of
Université de Lyon, within the program "Investissements d'Avenir"
(ANR-11-IDEX-0007) operated by the French National Research Agency (ANR). AM was
supported by NSF grant DMS-2054182. RTD was supported by NSF grants
DMS-1855825 and DMS-2246684.

\section{Elek's refinement of Kaimanovich's Theorem for measured graphs}
If ${\Ggraph}$ is a locally finite Borel graph on a standard
probability space $(X,\mu)$, then the \define{(vertex) isoperimetric
constant of ${\Ggraph}$} is the
infimum of $\mu(\boundary _{\Ggraph} A) / \mu(A)$ over all Borel subsets $A \subseteq X$ of
positive measure such that ${\Ggraph}_{\restriction A}$ has finite
connected components. Here, $\boundary _{\Ggraph} A$ denotes the set of vertices in $X\mysetminus A$ which are ${\Ggraph}$-adjacent to a vertex in $A$.  If the measure $\mu$ is $\RR_{\Ggraph}$-quasi-invariant then the isoperimetric constant of ${\Ggraph}$ can be equivalently phrased in terms of the associated Radon-Nikodym cocycle (see \cite[\S 8]{KM04}).

In \cite{Kai97}, Kaimanovich established the equivalence between $\mu$-hyperfiniteness of a measured equivalence relation $\RR$ and vanishing of the isoperimetric constant of all bounded graph structures on $\RR$. In \cite{El12}, Elek sharpened Kaimanovich's Theorem by establishing the following characterization of hyperfiniteness for a fixed measured graph ${\Ggraph}$.

\begin{thm}[Elek \cite{El12}]\label{thm:Elek}
  Let ${\Ggraph}$ be a locally finite Borel graph on a standard probability
  space $(X,\mu)$. Then ${\Ggraph}$ is $\mu$-hyperfinite if and only if for every positive measure Borel subset $X_0\subseteq X$, the isoperimetric constant of ${\Ggraph}_{\restriction X_0}$ is $0$.
\end{thm}

While the theorem in \cite{El12} is stated for measure preserving bounded degree graphs, it can easily be extended to all locally finite graphs which are not necessarily measure preserving. For the convenience of the reader we indicate the proof.

\begin{proof}[Proof of Theorem \ref{thm:Elek}]
Suppose first that ${\Ggraph}$ is $\mu$-hyperfinite. Let $X_0\subseteq X$ be a Borel set of positive measure and let ${\Hgraph} = {\Ggraph}_{\restriction X_0}$. Then ${\Hgraph}$ is $\mu$-hyperfinite, so after ignoring a null set we can find finite Borel subequivalence relations $\RR_0\subrel  \RR_1\subrel  \cdots$ with $\RR_{{\Hgraph}} = \bigcup _n \RR_n$. Since ${\Hgraph}$ is locally finite,
given $\epsilon >0$, we may find $n$ large enough so that $\mu (A_n)>\mu (X_0)(1-\epsilon )$, where $A_n= \{ x\in X_0 \, : \, {\Hgraph}_x\subseteq [x]_{\RR_n} \}$ and ${\Hgraph}_x$ denotes the set of ${\Hgraph}$-neighbors of $x$. Then ${\Hgraph}_{\restriction A_n} \subseteq \RR_n$, so ${\Hgraph}_{\restriction A_n}$ has finite connected components. In addition, $\mu (\boundary _{\Hgraph} A_n)/\mu (A_n) < \epsilon /(1-\epsilon )$, so as $\epsilon >0$ was arbitrary this shows the isoperimetric constant of ${\Hgraph}$ is $0$.

Assume now that for every positive measure Borel subset $X_0\subseteq X$ the isoperimetric constant of ${\Ggraph}_{\restriction X_0}$ is $0$. To show that ${\Ggraph}$ is $\mu$-hyperfinite it suffices to show that for any $\epsilon >0$ there exists a Borel set $Y\subseteq X$ with $\mu (Y)\geq 1-\epsilon$ such that ${\Ggraph}_{\restriction Y}$ has finite connected components (since then we can find a sequence of such sets $Y_n$, $n\in \N$, with $\mu (Y_n)\geq 1-2^{-n}$, so by Borel-Cantelli $\RR_{\Ggraph}=\liminf _n \RR_{\Ggraph_{\restriction Y_n}}$ is $\mu$-hyperfinite). Given $\epsilon >0$, by Zorn's Lemma we can find a maximal collection $\mc{A}$ of pairwise disjoint nonnull Borel subsets of $X$ subject to
\begin{enumerate}
\item[(i)] ${\Ggraph}_{\restriction \bigcup \mc{A}}$ has finite connected components;
\item[(ii)] $\mu (\boundary _{\Ggraph} (\bigcup \mc{A})) \leq \epsilon \mu (\bigcup \mc{A} )$;
\item[(iii)] If $A,B\in \mc{A}$ are distinct, then no vertex in $A$ is adjacent to a vertex in $B$.
\end{enumerate}
Let $Y= \bigcup \mc{A}$. We now claim that the set $X_0 = X \mysetminus (Y \cup \boundary _{\Ggraph} Y)$ is null. Otherwise, by hypothesis we may find a Borel set $A_0\subseteq X_0$ of positive measure such that ${\Ggraph}_{\restriction A_0}$ has finite connected components and $\mu (\boundary _{{\Ggraph}_{\restriction X_0}} A_0 ) <\epsilon \mu (A_0)$. But then the collection $\mc{A}_0 = \mc{A}\cup \{ A_0 \}$ satisfies (i)-(iii) in place of $\mc{A}$, contradicting maximality of $\mc{A}$. Thus, $\mu (Y) = 1- \mu (\boundary _{\Ggraph} Y ) \geq 1- \epsilon \mu (Y) \geq 1-\epsilon$, and ${\Ggraph}_{\restriction Y}$ has finite connected components, which finishes the proof.
\end{proof}

\begin{remark} 
A Borel graph $\Ggraph$ is hyperfinite if there is an increasing sequence of
Borel graphs $(\Ggraph_n)_{n \in \N}$ with finite connected components so that
$\bigcup_n \Ggraph_n = \Ggraph$. If $\Ggraph$ is a locally finite Borel graph on a standard
Borel space $X$ whose connectedness relation is $\RR_\Ggraph$, then $\RR_\Ggraph$ is
hyperfinite if and only if $\Ggraph$ is hyperfinite. More generally, if $Y
\subseteq X$ is an $\RR_\Ggraph$-invariant Borel set, then $\Ggraph_{\restriction Y}$ is
hyperfinite if and only if $(\RR_\Ggraph)_{\restriction Y}$ is hyperfinite. So if $\mu$
is a $\Ggraph$-quasi-invariant measure, then since every $\mu$-null set is contained in
a Borel invariant $\mu$-null set, $\Ggraph$ is $\mu$-hyperfinite if and only if
$\RR_\Ggraph$ is $\mu$-hyperfinite.

In contrast, there exist a locally finite graph $\Ggraph$ on $X$ and a
non-$\Ggraph$-quasi-invariant Borel probability measure $\mu$ on $X$ so that $\Ggraph$ is
$\mu$-hyperfinite, but $\RR_\Ggraph$ is not
$\mu$-hyperfinite. One such example arises when $\Ggraph$ is any ergodic
quasi-invariant locally finite Borel graph on a standard probability space
$(X,\nu)$ so that $\RR_{\Ggraph}$ is
not $\nu$-hyperfinite. Then by \cite[Proposition 4.2]{KST99} we can find a Borel maximal
independent set $A \subseteq X$ for $\Ggraph$. Since $A$ intersects every
connected component of $\Ggraph$ and $\nu$ is $\Ggraph$-quasi-invariant, $A$ has positive
measure. Let $\mu$ be the measure on $X$ defined by
$\mu(Y) = \frac{\nu(Y \inters A)}{\nu(A)}$. Then $\Ggraph$ is
$\mu$-hyperfinite, since $A$ is $\mu$-conull and $\Ggraph_{\restriction A}$ has finite connected
components. However $\RR_\Ggraph$ is not $\mu$-hyperfinite, since for any Borel
set $Y \subseteq X$, $(\RR_\Ggraph)_{\restriction
Y}$ is hyperfinite if and only
if $(\RR_\Ggraph)_{\restriction [Y]_{\RR_\Ggraph}}$ is hyperfinite, and if $Y$ is $\mu$-conull,
then $[Y]_{\RR_\Ggraph}$ is $\nu$-conull since $\Ggraph$ is $\nu$-ergodic. So
$\RR_\Ggraph$ is not
$\mu$-hyperfinite since $\Ggraph$ is not $\nu$-hyperfinite.
\end{remark}

\begin{remark} We have used the vertex isoperimetric constant, whereas
\cite{El12} uses the edge isoperimetric constant. The relationship is as follows. Let ${\Ggraph}$ be a graph on $(X,\mu )$ and let $M_r$ be the
Borel $\sigma$-finite measures on ${\Ggraph}$ given by $M_r(D)= \int _X
|D^x| \, d\mu$. The \define{edge isoperimetric constant} of ${\Ggraph}$ is the infimum of $M_r (\boundary _{\Ggraph}^{e} A) / \mu(A)$ over all Borel subsets $A \subseteq X$ of positive measure such that ${\Ggraph}_{\restriction A}$ has finite connected components. Here, $\boundary _{\Ggraph} ^e A$ is the set of all edges of ${\Ggraph}$ having one endpoint in $A$ and one in $X\mysetminus A$. (Note that since $\boundary _{\Ggraph} ^e A$ is symmetric, one obtains the same definition if in place of $M_r$ one uses the measure $M_l(D)=\int _X |D_x|\, d\mu$.) It is then easy to see that if $\mu$ is ${\Ggraph}$-quasi-invariant, and if ${\Ggraph}$ is bounded (meaning that ${\Ggraph}$ is bounded degree and the Radon-Nikodym cocycle $\rho : \RR_{\Ggraph} \ra \R ^+$ associated to $\mu$ is essentially bounded on ${\Ggraph}$) then for any positive measure Borel subset $X_0\subseteq X$, the edge isoperimetric constant of ${\Ggraph}_{\restriction X_0}$ vanishes if and only if the vertex isoperimetric constant of ${\Ggraph}_{\restriction X_0}$ vanishes.
\end{remark}

\section{One-ended spanning subforests}

In this section, we characterize exactly when a locally finite probability
measure preserving Borel graph has a one-ended spanning subforest.

\begin{thm}[For \pmp{} graphings] \label{thm:pmp}
  Suppose that ${\Ggraph}$ is a measure preserving aperiodic
  locally finite
  Borel graph on a standard probability space $(X,\mu )$. Then ${\Ggraph}$ has a Borel
  a.e.\ one-ended spanning subforest if and only if ${\Ggraph}$ is $\mu$-nowhere
  two-ended.
\end{thm}

We further conjecture the following strengthening of this theorem for
graphs which are not necessarily measure preserving. In this more general setting, the correct generalization of ($\mu$-a.e.) aperiodicity is $\mu$-nowhere smoothness of $\Ggraph$; we say that $\Ggraph$ is \define{$\mu$-nowhere smooth} if there is no positive measure Borel subset of $X$ which meets each $\Ggraph$-component in at most one point.

\begin{conj}[Non necessarily \pmp{} graphings] \label{conj:pmp}
Suppose that $\Ggraph$ is a $\mu$-nowhere smooth locally finite Borel
graph on a standard probability space $(X,\mu)$. Then $\Ggraph$ has a Borel a.e.\
one-ended spanning subforest if and only if $\Ggraph$ is $\mu$-nowhere two-ended.
\end{conj}

We know that the forward direction of the above conjecture is true by
Lemma~\ref{lem:obstruct} below. The reverse direction is known to be true
in the case when $\Ggraph$ is hyperfinite by Lemma~\ref{lem:hyperfinite}, and when
$\Ggraph$ is acyclic by the following theorem of \cite{CMT-D}.
\begin{thm}[Non necessarily \pmp{} treeings {\cite[Theorem 1.5]{CMT-D}}]\label{thm:acyclic}
  Suppose that $\Ggraph$ is an acyclic, aperiodic locally finite Borel graph on a standard
  probability space $(X,\mu)$. If ${\Ggraph}$ is $\mu$-nowhere two-ended, then
  $\Ggraph$ has a Borel a.e.\  one-ended spanning subforest.
\end{thm}

We will begin by proving the forward direction of
Theorem~\ref{thm:pmp} (and also Conjecture~\ref{conj:pmp}). An easy
argument shows that the graph associated to a
free measure preserving action of $\Z$ cannot have a Borel a.e.\ one-ended spanning
subforest; such a subforest must come from removing a single
edge from each connected component of the graph. This set of edges would witness the
fact that the graph $\Ggraph$ is smooth, contradicting our assumption that the
action of $\Z$ is measure-preserving. Our argument is a simple
generalization of this idea.

\begin{lemma}\label{lem:obstruct}
  Suppose that ${\Ggraph}$ is a $\mu$-nowhere smooth locally finite Borel graph on a
  standard probability space $(X,\mu)$. If there
  is a set of positive measure on which ${\Ggraph}$ is two-ended, then ${\Ggraph}$ does not
  admit a Borel a.e.\ one-ended spanning subforest.
\end{lemma}
\begin{proof}
  By restricting to and renormalizing a Borel ${\Ggraph}$-invariant subset of positive
  measure, we may assume
  that ${\Ggraph}$ is everywhere two-ended and has a Borel a.e.\ one-ended spanning
  subforest ${\Tgraph}$. We will now show ${\Ggraph}$ is smooth. Let $Y$ be the set of
  connected $C \in
  [\RR_{\Ggraph}]^{<\N}$ such that removing $C$ from ${\Ggraph}$ disconnects its connected
  component into exactly two infinite pieces. Recall that $ [\RR_{\Ggraph}]^{<\N}$ is the Borel set of finite subsets of $X$ made of $\RR_{\Ggraph}$-equivalent points.
  By taking a countable
  coloring of the intersection graph on $Y$ (see {\cite[Lemma 7.3]{KM04}} and
  {\cite[Proposition 2]{CM16}}), we may find a Borel set $Z
  \subseteq Y$ which meets every connected component of ${\Ggraph}$ and so that
  distinct $C,D \in Z$ are pairwise disjoint and if $C$ and $D$ are in
  the same connected component, then $|C| = |D|$. By discarding a smooth
  set, we may assume $Z$ meets each connected component of ${\Ggraph}$ infinitely
  many times.

  Let ${\Hgraph}$ be the graph on $Z$ where $C \mathrel{{\Hgraph}} C'$ if $C$
  and $C'$ are in the same connected component of ${\Ggraph}$ and there is no $D
  \in Z$ such that removing $D$ from ${\Ggraph}$ places $C$ and $C'$ in different
  connected components. 
  It is then easy to see that every element of $Z$ has ${\Hgraph}$-degree at most $2$ (for full details see {\cite[Sublemma 5.7]{Mil09}}).  
  As ${\Ggraph}$ is two-ended, each connected component of ${\Ggraph}$ contains at most one element of $Z$ of ${\Hgraph}$-degree $1$; we may thus assume that ${\Hgraph}$ is 2-regular by discarding another smooth set.

  Now let $Z' \subseteq Z$ be the set of $C \in Z$ such that there exists a $\Hgraph$-neighbor $D$ of $C$ and a component $F$ of $\Tgraph$, such that $C$ meets $F$, but $D$ does not meet $F$.

  It is easy to see that $Z'$ meets each connected component
  of ${\Ggraph}$ and is finite (else ${\Tgraph}$ is not a one-ended spanning subforest),
  but then ${\Ggraph}$ is smooth.
\end{proof}

Our proof of the reverse direction of Theorem \ref{thm:pmp} splits into two
cases based on Theorem \ref{thm:Elek}. In particular, it will suffice to
prove Theorem~\ref{thm:pmp} for $\mu$-hyperfinite graphs, and graphs having
positive isoperimetric constant.

\subsection{Measure preserving graphs with at least quadratic growth}\label{sec:growth}

We begin with a lemma giving a sufficient condition for a graph to possess
a one-ended spanning subforest. (In fact, this condition can be shown to be
equivalent to the existence of such a subforest)

Let $f$ be a partial function from a set $X$ into itself, and let $y\in X$. The \define{back-orbit} of $y$ under $f$ is the set of all $x\in \dom(f)$ for which there is some $n\geq 0$ with $f^n(x)=y$.

\begin{lemma}\label{abstract_clinton}
  Suppose that ${\Ggraph}$ is a locally finite Borel graph on a standard
  probability space $(X,\mu)$, and there are partial Borel functions $f_0,
  f_1, \ldots \subseteq {\Ggraph}$ such that
  \begin{enumerate}
    \item $\sum_{i} \mu(\dom(f_i)) < \infty$
    \item $\bigunion \dom(f_i) = X$
    \item Every $f_i$ is aperiodic and has finite back-orbits.
    \item For every $i$ and $x \in \dom(f_i)$ there is a $j \geq i$ such
    that $f_i(x) \in \dom(f_j)$.
  \end{enumerate}
  Then ${\Ggraph}$ has a Borel a.e.\ one-ended spanning subforest.
\end{lemma}
\begin{proof}
  As usual, we may assume $\mu$ is ${\Ggraph}$-quasi-invariant.
   By (1), (2) and a
  Borel-Cantelli argument there is a ${\Ggraph}$-invariant conull Borel set $A$
  such that for every $x \in A$ there are only finitely many $i$ such that
  $x \in \dom(f_i)$. Define $n(x)$ for $x \in A$ to be the
  largest $i$ such that $x \in \dom(f_i)$. Now define $f \from A \to A$ by
  $f(x) = f_{n(x)}(x)$. We claim that $f$ generates
  a one-ended spanning subforest of ${\Ggraph}_{\restriction A}$.
  To see this, note $f$ is aperiodic since each $f_i$ is aperiodic by (3),
  and by (4) the value of $n(x)$ is non-decreasing along orbits of $f$. We also
  see that $f$ has finite back-orbits by induction since the $f_i$ do.
\end{proof}

Now we apply this lemma to show that any measure preserving graph of at least quadratic
growth has a Borel a.e.\ one-ended spanning subforest.

\begin{theorem}\label{thm:quad}
  Suppose that ${\Ggraph}$ is a measure preserving locally finite Borel graph on a standard
  probability space $(X,\mu)$ of at least quadratic growth, so there is a $c >
  0$ such that for every $x\in X$, $|B_r(x)| \geq c r^2$ where $B_r(x)$ is the
  ball of radius $r$ around $x$ in ${\Ggraph}$. Then ${\Ggraph}$ has a Borel a.e.\ one-ended spanning subforest.
\end{theorem}

\begin{figure}[h]
\def\svgwidth{\columnwidth}
\resizebox{67mm}{!}{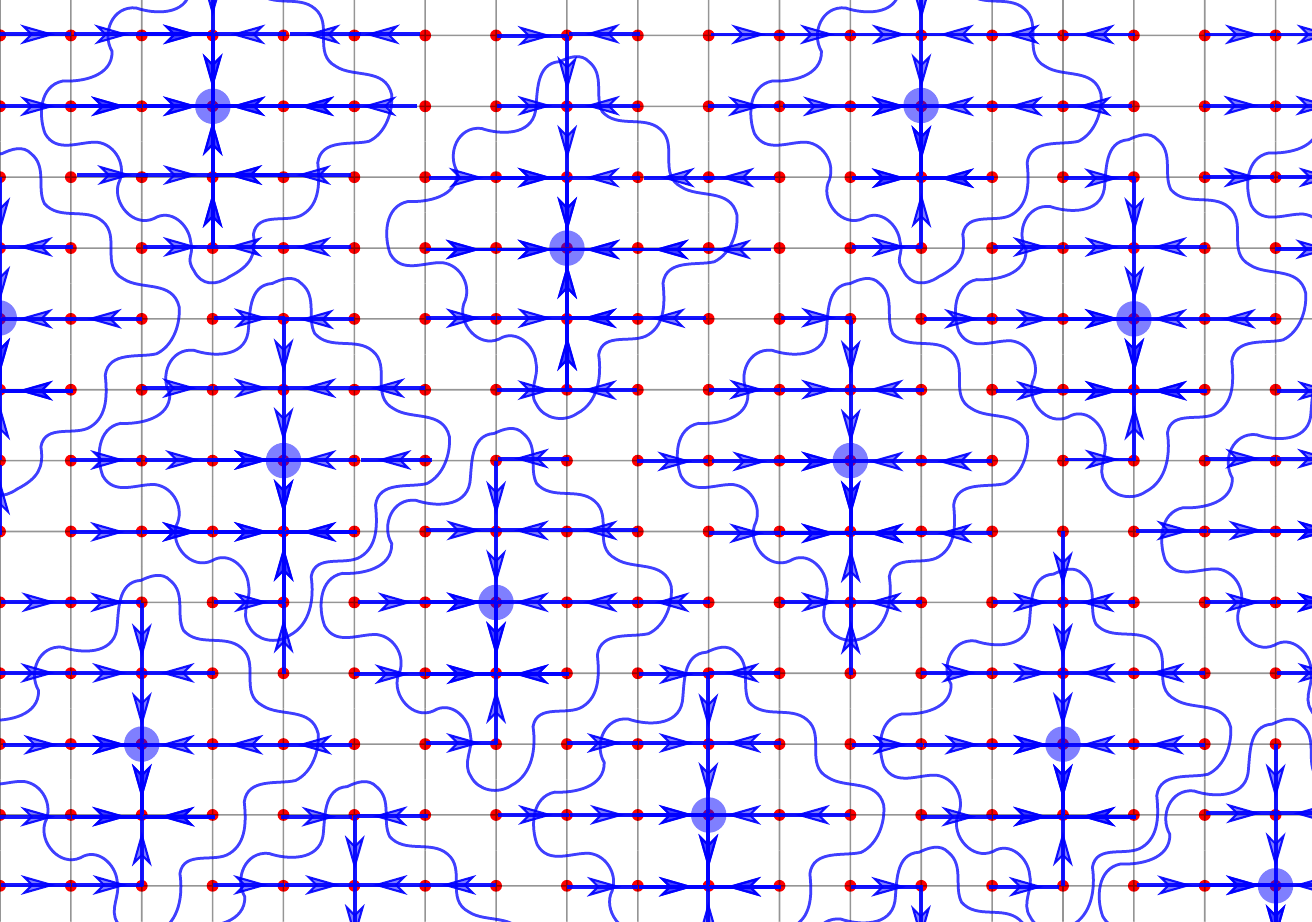}
\hskip10pt
\resizebox{67mm}{!}{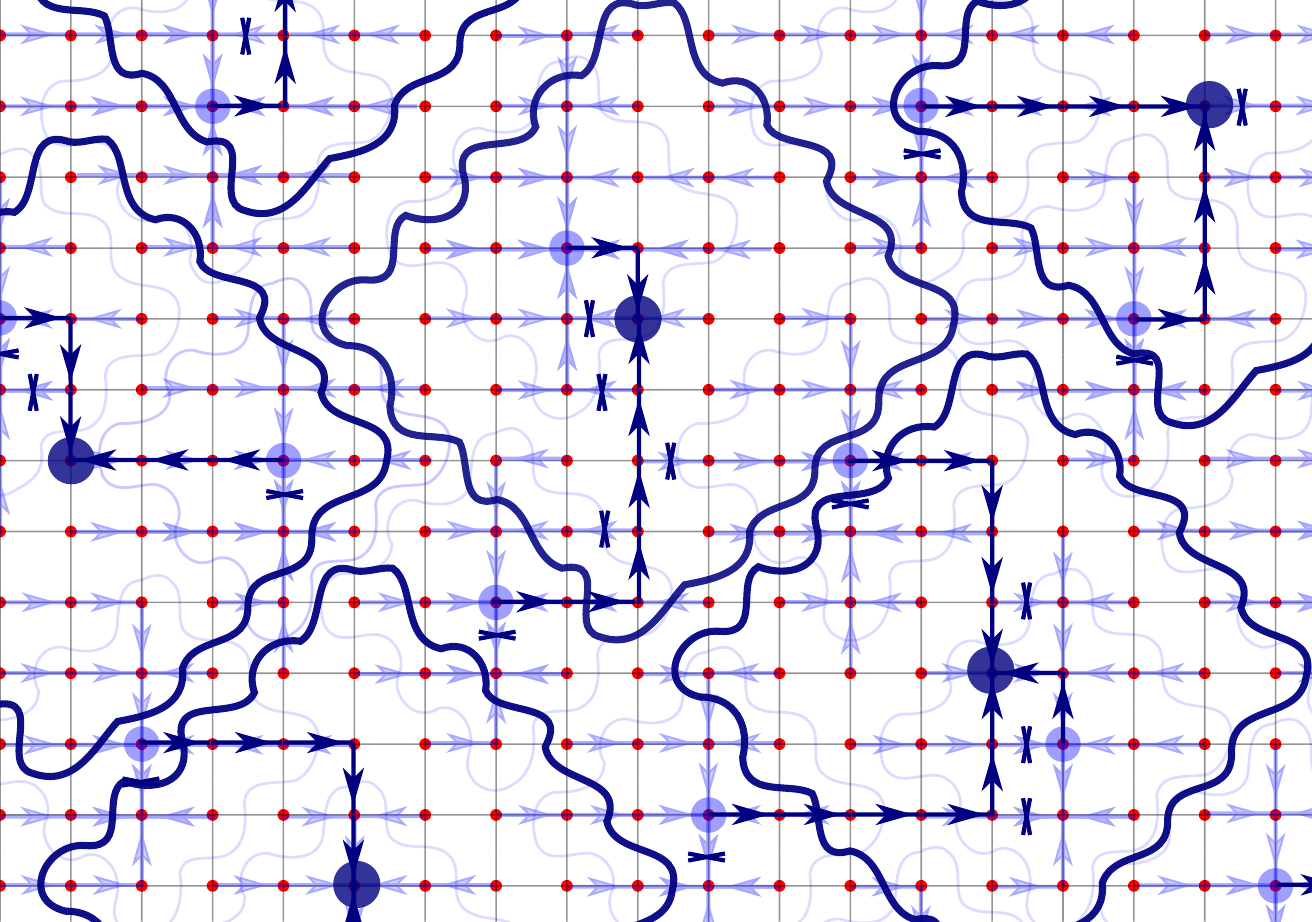}
\caption{$A_0$ (red), $A_1$ (blue) and $f_0$ (blue arrows) -- $A_2$ and $f_1$  (deep blue).
Crosses represent  "change of decision" for some edges.}
\end{figure}

Observe that any measure preserving locally finite Borel graph with positive isoperimetric constant satisfies the hypotheses of Theorem~\ref{thm:quad} (any such graph in fact has exponential growth).

\begin{proof}[Proof of Theorem~\ref{thm:quad}]
  Fix a sequence of natural numbers $r_n$, $n\geq 1$, such that $\sum _{n} 2r_{n+1}/(cr_n^2)<\infty$ (e.g., take $r_n=2^n$).  For each $n\geq 1$ let $A_n$ be a Borel subset of $X$ which is maximal with respect to the property that $B_{r_n}(x)\cap B_{r_n}(y)=\emptyset$ for all distinct $x,y\in A_n$. Then $\mu (A_n)\leq 1/(cr_n^2)$. Let $A_0=X$. Define partial Borel functions $f_n$ from $X$ to $X$ as follows. For
  each $x \in A_n$ choose the lexicographically least minimal length path $x = x_0, x_1,
  \ldots, x_k$ from $x$ to an element of $A_{n+1}$. Let $f_n$ be the function obtained by taking the union of all
  pairs $(x_i, x_{i+1})$ from these paths. See Figure 1.
  Observe that by maximality of $A_{n+1}$, the length of each such path $x_0,x_1,\dots , x_k$ is at most $2r_{n+1}$, and hence $\mu (\dom (f_n)) \leq 2r_{n+1}\mu (A_n)\leq 2r_{n+1}/(cr_n^2)$. Therefore, $\sum_n \mu(\dom(f_n))$ converges. The remaining properties
  from Lemma~\ref{abstract_clinton} are trivial to verify.
\end{proof}

\begin{remark}
We note that a slightly refined argument yields the same result as
Theorem~\ref{thm:quad} for
graphs $\Ggraph$ of \define{superlinear growth}: for every $c > 0$ there exists an $r > 0$
such that for every $x \in X$, $|B_r(x)| \geq cr$. Since we will not need this, we omit the proof.
\end{remark}

\subsection{\texorpdfstring{$\mu$-hyperfinite graphs}{mu-hyperfinite graphs}}\label{sec:hyperfinite}

We now turn to the case of $\mu$-hyperfinite graphs. Suppose ${\Ggraph}$ is
a locally finite Borel graph on a standard probability space $(X,\mu)$, and
$\Ggraph$ is $\mu$-hyperfinite. We do not assume that $\mu$ is preserved.
Then after ignoring a null set we can find
an acyclic Borel subgraph $\Tgraph \subseteq \Ggraph$ on the same vertex set
and having the same connectedness relation. This subgraph will have at
least as many ends as $\Ggraph$ in each connected component. Hence, on the
set where $\Ggraph$ has more than two ends, we can apply
Theorem~\ref{thm:acyclic} to $\Tgraph$ to find a one-ended spanning
subforest.

Thus, it will suffice to handle $\mu$-hyperfinite $\Ggraph$ that are a.e.\ one-ended.
We begin with the following lemma:

\begin{lemma}\label{lem:2in1lem}
Let ${\Ggraph}$ be a locally finite Borel graph on a standard probability
space $(X,\mu)$ in which every connected component has one end. Let
${\Tgraph}\subseteq {\Ggraph}$ be an acyclic Borel subgraph with
$\RR_{{\Tgraph}}=\RR_{{\Ggraph}}$ and assume that every connected component of
${\Tgraph}$ is a 2-regular line. Then there exists a Borel set
$X'\subseteq X$ and an acyclic Borel subgraph ${\Tgraph}'\subseteq {\Ggraph}_{\restriction  X'}$ such that
\begin{enumerate}
\item $X'$ is a complete section for $\RR_{\Ggraph}$ with $\mu (X') < 3/4$,
\item $\RR_{{\Tgraph}'} = \RR_{({\Ggraph}_{\restriction  X'})} = (\RR_{\Ggraph})_{\restriction  X'}$,
\item Every connected component of ${\Ggraph}_{\restriction  X'}$ has one end, and every connected component of ${\Tgraph}'$ is a 2-regular line,
\item For $x,y\in X'$, we have $(x,y)\in \mathcal{\Tgraph}'$ if and only if there is no $z\in X'$ in the interior of the line segment between $x$ and $y$ in $\Tgraph$.
\end{enumerate}
\end{lemma}

\begin{proof}
For $e=(x,y)\in {\Ggraph}$, let $[e]$ denote the set of points in $X$ which lie strictly between $x$ and $y$ on ${\Tgraph}$ (so $x,y\not\in [e]$, and $[e]=\emptyset$ whenever $e\in {\Tgraph}$). For each $e\in {\Ggraph}$ the set $\{ e'=(u,v) \in {\Ggraph}\, : \, \text{either } u\in [e] \text{ or }v\in [e] \}$ is finite; define
\[
n(e) = \max \{ |[e']| \, : \, e'=(u,v)\in {\Ggraph} \text{ and either } u\in [e] \text{ or }v\in [e] \} .
\]
For each $N\geq 1$ let ${\Ggraph}_N = \{ e\in {\Ggraph} \, : \, |[e]|\leq N \text{ and } n(e) \leq N \}$. Then ${\Ggraph}_0\subseteq {\Ggraph}_1\subseteq \cdots$ and $\bigcup _N {\Ggraph}_N = {\Ggraph}$. For a subset ${\Ggraph}'$ of ${\Ggraph}$ let $[{\Ggraph}'] = \bigcup \{ [e] \, : \, e\in {\Ggraph}' \}\subseteq X$. Since ${\Ggraph}$ is one-ended we have that $[{\Ggraph}] = X$. Therefore, by choosing $N$ large enough we may ensure that $\mu ([{\Ggraph}_N] ) >3/4$.

Define a graph $\widetilde{{\Hgraph}}$ whose vertex set is ${\Ggraph}_N$ by
$\widetilde{{\Hgraph}} = \{ (e_0,e_1)\in {\Ggraph}_N\times {\Ggraph}_N \, :
\, [e_0]\cap [e_1] \neq \emptyset \}$. The graph $\widetilde{{\Hgraph}}$
has bounded degree since if $e_0 \in {\Ggraph}_N$ then $|[e_0]|\leq N$ so
there are less than $2N^2$ many (this is not exact) intervals in
${\Tgraph}$ of length at most $N$ which intersect $[e_0]$. By
\cite[Proposition 4.6]{KST99} we may therefore find a Borel coloring of the vertices of $\widetilde{{\Hgraph}}$ using only finitely many colors, say $C_1,\dots , C_k$. Define now ${\Ggraph}_N^0 = {\Ggraph}_N$ and for $1\leq i \leq k$ define
\[
{\Ggraph}^i_N= {\Ggraph}^{i-1}_N\mysetminus \{ e\in C_i \, : \, [e]\subseteq [{\Ggraph}^{i-1}_N\mysetminus \{ e \} ] \} .
\]
Then $[{\Ggraph}^k_N]= [{\Ggraph}_N]$, and $\widetilde{{\Hgraph}}_{\restriction
{\Ggraph}^k_N}$ has degree bounded by $2$, so by \cite[Proposition 4.6]{KST99} there is a Borel 3-coloring of $\widetilde{{\Hgraph}}_{\restriction  {\Ggraph}^k_N}$, say with color sets $D_0$, $D_1$, and $D_2$. Then $[{\Ggraph}_N] = [{\Ggraph}^k_N]=[D_0]\cup [D_1]\cup [D_2]$, so there is some $i\in \{ 0,1,2\}$ with $\mu ([D_i]) \geq \mu ([{\Ggraph}_N])/3 > 1/4$.

We take $X' = X\mysetminus [D_i]$, so that $\mu (X')<3/4$, and we let ${\Tgraph}' = ( {\Tgraph}_{\restriction  X' }) \cup D_{i \restriction X'}$. We have ${\Tgraph}'\subseteq {\Ggraph}_{\restriction  X'}$, and $\RR_{{\Tgraph}'} = (\RR_{\Tgraph})_{\restriction  X'}$ since $D_i$ is an independent set for $\widetilde{{\Hgraph}}$. Property (2) easily follows. It is also clear that every connected component of ${\Tgraph}'$ is a 2-regular line, and that $(x,y)\in \mathcal{\Tgraph}'$ if and only if there is no $z\in X'$ on the interior of the line segment between $x$ and $y$ in $\Tgraph$. This implies that for $e=(x,y)\in {\Ggraph}_{\restriction X'}$ we have $[e]^{\Tgraph '} = [e]\cap X'$, where $[e]^{\Tgraph '}$ denotes the set of points in $X'$ which lie strictly between $x$ and $y$ in ${\Tgraph}'$. For (3), it remains to show that every connected component of ${\Ggraph}_{\restriction  X'}$ has one end. Since ${\Ggraph}_{\restriction X'}$ is locally finite, this is equivalent to showing that for each $x\in X'$ the set $\{ e\in {\Ggraph}_{\restriction X'} \, : \, x\in [e]^{\Tgraph '} \}$ is infinite. Since ${\Ggraph}$ is one-ended and locally finite, for each $x\in X$ the set $\{ e\in {\Ggraph} \, : \, |[e]|> N \text{ and }x\in [e] \}$ is infinite. It therefore suffices to show that the set $\{ e\in {\Ggraph} \, : \, |[e] |> N \}$ is contained in ${\Ggraph}_{\restriction  X'}$ (since for $x\in X'$ and $e\in {\Ggraph}_{\restriction X'}$ we have $x\in [e]^{\Tgraph '}$ if and only if $x\in [e]$). If $|[e]|>N$ then neither endpoint of $e$ lies in $[\Ggraph _N]$, so both endpoints of $e$ lie in $X\mysetminus [\Ggraph _N]\subseteq X\mysetminus [D_i] = X'$, hence $e\in {\Ggraph}_{\restriction X'}$. Finally, $X'$ is a complete section for $\RR_{\Ggraph}$ since the set $\{ e\in {\Ggraph} \, : \, |[e]| >N \}$ meets every connected component of $\mathcal{\Ggraph}$ and, as we just showed, this set is contained in ${\Ggraph}_{\restriction X'}$.
\end{proof}

\begin{lemma}\label{lem:2in1}
Let ${\Ggraph}$ be a locally finite Borel graph on a standard probability
space $(X,\mu)$ in which every
connected component has one end. Let ${\Tgraph_0}\subseteq {\Ggraph}$ be an
acyclic Borel subgraph with $\RR_{{\Tgraph_0}}=\RR_{{\Ggraph}}$ and assume that
every connected component of ${\Tgraph_0}$ has 2 ends. Then there is an a.e.\ one-ended spanning subforest $\Tgraph$ of
$\Ggraph$ such that
$\RR_{\Tgraph} = \RR_\Ggraph$ $\mu$-a.e.
\end{lemma}
\begin{proof}
After moving to a finite-to-1 minor of $\Ggraph$, we may assume without
loss of generality that every component of $\Tgraph_0$ is a 2-regular
line. By Lemma \ref{lem:2in1lem}, after ignoring a $\mu$-null set we may find a vanishing sequence of Borel sets $X= X_0\supseteq X_1\supseteq\cdots$ along with acyclic subgraphs $\Tgraph_n\subseteq \Ggraph_{\restriction X_n}$ such that, for all $n\geq 0$:
\begin{enumerate}
\item $X_{n+1}$ is a complete section for $\RR_{(\Ggraph _{\restriction  X_n})}$,
\item $\RR_{\Tgraph _n } =\RR_{(\Ggraph _{\restriction  X_n})} = (\RR_{\Ggraph})_{\restriction  X_{n}}$,
\item Every connected component of ${\Ggraph}_{\restriction  X_n}$ has one end, and every connected component of ${\Tgraph}_n$ is a 2-regular line.
\item For $x,y\in X_{n+1}$ we have $(x,y)\in {\Tgraph}_{n+1}$ if and only if there is no $z\in X_{n+1}$ in the interior of the line segment between $x$ and $y$ in ${\Tgraph}_n$.
\end{enumerate}
Note that by (4) and induction on $n\geq 0$, for $x,y\in X_{n}$ we have $(x,y)\in {\Tgraph}_{n}$ if and only if there is no $z\in X_{n}$ in the interior of the line segment between $x$ and $y$ in ${\Tgraph}_0$. Now, for each $n\geq 0$, each component of ${{\Tgraph}_n}_{\restriction (X_n\mysetminus X_{n+1})}$ is a finite line segment in ${\Tgraph}_n$. Let $s_n:X_n\mysetminus X_{n+1} \rightarrow X_{n+1}$ be any ${{\Tgraph}_n}_{\restriction (X_n\mysetminus X_{n+1})}$-invariant Borel function such that for each $x\in X_n$ the point $s_n(x)\in X_{n+1}$ which is ${\Tgraph}_n$-adjacent to the ${{\Tgraph}_n}_{\restriction (X_n\mysetminus X_{n+1})}$-component of $x$. Define now $f:X\rightarrow X$ as follows. For $x\in X_n\mysetminus X_{n+1}$ define $f(x)\in X_n$ to be the ${\Tgraph}_n$-neighbor of $x$ which lies along the unique ${\Tgraph}_n$-path from $x$ to $s_n(x)$. Since the sets $X_n$ are vanishing, this defines $f$ on all of $X$. Clearly $(x,f(x))\in \Ggraph$ for all $x\in X$, and $f^i(x)\neq x$ for all $i\geq 1$ (since $f^i(x)$ is eventually in $X_n$ for any fixed $n$ and large $i$). If $x\in X_0\mysetminus X_{n+1}$, then the back-orbit of $x$ under $f$ is contained in the ${{\Tgraph}_0}_{\restriction (X_0\mysetminus X_{n+1})}$-component of $x$, which is finite. Thus, $f$ defines a one-ended subforest of $\Ggraph$. It remains to show that any two $\RR_{\Ggraph}$-related points eventually meet under $f$. By induction on $n\geq 0$, we see that each ${{\Tgraph}_0}_{\restriction (X_0\mysetminus X_{n+1})}$-component contains exactly one ${{\Tgraph}_n}_{\restriction (X_n\mysetminus X_{n+1} )}$-component, and that if two points $x,y\in X_0\mysetminus X_{n+1}$ are in the same ${{\Tgraph}_0}_{\restriction (X_0\mysetminus X_{n+1})}$-component, then there exists $i,j>0$ such that $f^i(x)=s_n\circ\cdots\circ s_0 (x)= s_n\circ\cdots\circ s_0 (y) = f^j(y)$.
\end{proof}

\begin{lemma}\label{lem:hyperfinite}
Let $\Ggraph$ be a locally finite Borel graph on a standard probability
space $(X,\mu )$. Assume that $\Ggraph$ is $\mu$-hyperfinite but $\mu$-nowhere $2$-ended. Then
$\Ggraph$ has a $\mu$-a.e. one-ended spanning subforest $\Tgraph$.
Moreover, if $\Ggraph$ is measure-preserving, then we may choose
$\Tgraph$ so that $\RR_\Tgraph = \RR_\Ggraph$
$\mu$-a.e.
\end{lemma}

\begin{proof}
We may assume that $\mu$ is $\Ggraph$-quasi-invariant.
Since $\Ggraph$ is $\mu$-hyperfinite, after ignoring a
$\Ggraph$-invariant $\mu$-null set we may find an acyclic Borel graph
$\Tgraph_0\subseteq \Ggraph$ with $\RR_{\Tgraph_0} =
\RR_{\Ggraph}$. Let $X_1$ be the ($\RR_{\Ggraph}$-invariant) set of vertices $x$ so that
${\Tgraph_{0} }_{\restriction
[x]_{\RR_{\Tgraph_0}}}$ is one-ended, $X_2$ similarly for the
components on which $\Tgraph_{0}$ is two-ended, and finally let $X_3$ be the
set of components on which $\Tgraph_{0}$ has more than two ends.
${\Tgraph_0}_{\restriction  X_1}$ is a Borel a.e.\ one-ended spanning subforest of $\Ggraph _{\restriction
X_1}$.
The number of ends of $\Ggraph _{\restriction  X_2}$ is one since it is bounded above by the number of ends of ${\Tgraph_{0}} _{\restriction  X_2}$.
By Lemma~\ref{lem:2in1} we can find a Borel a.e.\ one-ended spanning
subforest of $\Ggraph _{\restriction  X_2}$. By Theorem~\ref{thm:acyclic} we can find
a Borel a.e.\ one-ended spanning subforest of  ${\Tgraph_{0}} _{\restriction  X_3}$ and thus of $\Ggraph _{\restriction  X_3}$.

Finally, note that in the case when $\Ggraph$ is measure
preserving, $\Tgraph_0$ must have either one or two ends almost
everywhere, and so $\RR_{\Tgraph} = \RR_{\Ggraph}$ by the
condition that $\RR_\Tgraph = \RR_\Ggraph$ in Lemma~\ref{lem:2in1}.
\end{proof}

The following corollary is not used in the proof of Theorem \ref{thm:pmp}, but we will need it later in the proof of Theorem \ref{thm:planar}.

\begin{corollary}\label{cor:subtree}
Let ${\Ggraph}$ be a locally finite Borel graph on $(X,\mu )$ which is
$\mu$-nowhere two-ended. Assume that there exists an acyclic Borel subgraph ${\Tgraph}\subseteq {\Ggraph}$ with $\RR_{\Tgraph}=\RR_{\Ggraph}$ modulo $\mu$. Then ${\Ggraph}$ has a Borel a.e.\ one-ended spanning subforest.
\end{corollary}

\begin{proof}
Let $X_0$ be the ${\Ggraph}$-invariant set consisting of points whose
${\Tgraph}$-connected component has two ends. We can find a Borel a.e.\
one-ended spanning subforest of $\Ggraph _{\restriction  (X \mysetminus X_0)}$ by Theorem
\ref{thm:acyclic}.
Note that
${\Ggraph}_{\restriction X_0}$ is a.e.\ one-ended because
$\Tgraph$ and $\Ggraph$ have the same connectedness relation, so each
connected component of $\Tgraph$ has at least as many ends as $\Ggraph$,
and $\Ggraph$ is $\mu$-nowhere two-ended. Thus, we can find a Borel a.e.\
one-ended spanning subforest of $\Ggraph _{\restriction  X_0}$ by Lemma~\ref{lem:2in1}.
\end{proof}

\subsection{Proof of Theorem~\ref{thm:pmp}}

We recall one final lemma from \cite{CMT-D}.

\begin{lemma}[{\cite[Proposition 3.1]{CMT-D}}]\label{lem:extend}
Let ${\Ggraph}$ be a locally finite Borel graph on a standard Borel space
$X$. Let $X_0$ be a Borel subset of $X$ which meets every connected
component ${\Ggraph}$, and suppose that ${\Ggraph}_{\restriction X_0}$ has a
Borel one-ended spanning subforest ${\Hgraph}_0$. Then ${\Hgraph}_0$ can be
extended to a Borel one-ended spanning subforest ${\Hgraph}$ of ${\Ggraph}$.
\end{lemma}

We are now ready to prove Theorem~\ref{thm:pmp}.

\begin{proof}[Proof of Theorem~\ref{thm:pmp}]
Supposing ${\Ggraph}$ is a
locally finite Borel graph on a standard probability space $(X,\mu)$, by an
exhaustion argument we can find a Borel partition of $X$ into countably
many Borel ${\Ggraph}$-invariant sets $A, B_0,
B_1, \ldots$ such that ${\Ggraph}_{\restriction A}$ is $\mu$-hyperfinite, and for each $i\geq 0$ there is a
Borel set $C_i\subseteq B_i$ which meets every connected component of
${\Ggraph}_{\restriction B_i}$ such that ${\Ggraph}_{\restriction C_i}$ has positive
isoperimetric constant.

We can find a Borel a.e.\ one-ended spanning subforest of $\Ggraph _{\restriction  A}$ by
Lemma~\ref{lem:hyperfinite}. We can similarly find a Borel a.e.\ one-ended
spanning subforest of each $\Ggraph _{\restriction  C_i}$ by Theorem~\ref{thm:quad}. These
subforests can be extended to Borel a.e.\ one-ended subforests of $\Ggraph _{\restriction
B_i}$ by Lemma~\ref{lem:extend}.
\end{proof}

\subsection{The One Percent Theorem}

Let $\Ggraph$ be a locally finite Borel graph on a standard probability
space $(X,\mu)$, which we assume is $\RR_{\Ggraph}$-quasi-invariant. 
If $\mathcal{S} \subrel \RR_{\Ggraph}$ is an equivalence relation, then we say that
$\mathcal{S}$ is \define{$\Ggraph$-connected} if $\mathcal{S}\cap \Ggraph$ is a graphing of $\mathcal{S}$. Consider the collection
\begin{equation*}
\mathrm{CH}(\Ggraph) = \{ \mathcal{S}\subrel \RR_{\Ggraph} \, :
\, \mathcal{S} \text{ is a } \Ggraph\text{-connected } \mu\text{-hyperfinite equivalence relation} \}
\end{equation*}

We endow $\RR_{\Ggraph}$ with the $\sigma$-finite measure $\nu$ obtained by integrating the counting measure on vertical fibers: $\nu (D)=\int |D_x|\, d\mu (x)$ for each Borel subset $D$ of $\RR_{\Ggraph}$.

Recall the principle of measure exhaustion: there cannot be uncountably many disjoint sets of positive measure in a $\sigma$-finite measure space (see for instance Section 215 of \cite{Fremlin-vol2-2003} ``$\sigma$-finite spaces and the principle of exhaustion'').
Thus any chain $(A_i)_{i\in I}$ of measurable sets in a $\sigma$-finite measure space that is strictly increasing modulo null sets (i.e., $i< j$ implies $A_i\subseteq A_j$ and $A_j\setminus A_i$ has positive measure) must have a countable cofinal sequence.
Since a countable increasing union of $\Ggraph$-connected equivalence relations is $\Ggraph$-connected, and a countable increasing union of $\mu$-hyperfinite equivalence relations is $\mu$-hyperfinite, the union of any chain in $\mathrm{CH}(\Ggraph)$ lies in $\mathrm{CH}(\Ggraph)$. (Here we are working in the $\sigma$-finite measure space $(\RR_\Ggraph ,\nu )$.)
Therefore, by Zorn's Lemma, each $\mathcal{S}_0\in \mathrm{CH}(\Ggraph)$ is contained in a $\mu$-maximal element $\mathcal{S}$ of $\mathrm{CH}(\Ggraph)$. 
That is, if $\mathcal{Q}\in \mathrm{CH}(\Ggraph)$ satisfies $[x]_{\mathcal{S}}\subseteq [x]_{\mathcal{Q}}$ for $\mu$-a.e.\ $x\in X$, then in fact $[x]_{\mathcal{S}}= [x]_{\mathcal{Q}}$ for $\mu$-a.e.\ $x\in X$.

We now have the following corollary.

\begin{corollary}\label{cor:maximal}
Let $\Ggraph$ be a \pmp{} locally finite Borel graph on $(X,\mu )$. Assume that $\Ggraph$ is aperiodic and $\mu$-nowhere two-ended. Then $\Ggraph$ admits a $\mu$-a.e.\ one-ended Borel spanning subforest $\Tgraph\subseteq \Ggraph$ such that $\RR_{\Tgraph}$ is a $\mu$-maximal element of $\mathrm{CH}(\Ggraph)$.
\end{corollary}
Observe that the preservation of the measure is a necessary condition here since, for instance, the free group $\FF_2$ admits $\mu$-hyperfinite non-singular free actions $\FF_2\actson (X,\mu)$. The treeing $\Ggraph$ associated with a free generating set is the single $\mu$-maximal element of $\mathrm{CH}(\Ggraph)$.

\begin{proof}
Since $\Ggraph$ is \pmp{}, aperiodic, and $\mu$-nowhere two-ended,
$\Ggraph$ admits a $\mu$-a.e.\ one-ended Borel spanning subforest
$\Tgraph_0\subseteq \Ggraph$. Let $\mathcal{S}$ be a $\mu$-maximal
element of $\mathrm{CH}(\Ggraph)$ containing $\RR_{\Tgraph_0}$,
and let $\Ggraph_0= \mathcal{S}\cap \Ggraph$, so that $\Ggraph_0$ is a
graphing of $\mathcal{S}$. Since $\Ggraph_0$ is \pmp{} and
$\mu$-hyperfinite, almost every connected component of $\Ggraph_0$ has at
most two ends. Since $\Ggraph_0$ contains $\Tgraph_0$ as a $\mu$-a.e.\
one-ended subforest, $\Ggraph_0$ must be $\mu$-nowhere two-ended (by Lemma~\ref{lem:obstruct}), hence
$\Ggraph_0$ is $\mu$-a.e.\ one-ended. Now (by hyperfiniteness of $\Ggraph_0$)
 after discarding a nullset we can find an acyclic Borel graph
$\Tgraph' \subseteq \Ggraph_0$ such that $\RR_{\Tgraph'} =
\RR_\Ggraph$. To finish, apply Lemma~\ref{lem:2in1} to obtain a
$\mu$-a.e.\ one-ended Borel subtreeing $\Tgraph$ of
$\RR_\Ggraph$ which will satisfy the conclusion of the theorem.
\end{proof}

The combinatorial core of the proof of the forward direction of Elek's Theorem~\ref{thm:Elek} is the fact that a
graph $\Ggraph$ is $\mu$-hyperfinite (i.e., the equivalence relation it generates is $\mu$-hyperfinite) if and only if there are arbitrarily
large sets on which its restriction is finite. There is a dual
statement for $\mu$-nowhere hyperfiniteness, which we find interesting enough to state here, although it will not be used at any point in this article.

\begin{theorem} \label{thm:1percent}
Let $\Ggraph$ be a \pmp{} locally finite Borel graph on a standard
probability space $(X,\mu )$.
\begin{enumerate}

\item \textbf {(Ninety-Nine Percent Theorem)}
  $\Ggraph$ is $\mu$-hyperfinite if and only if for
every $\epsilon > 0$, there exists a Borel complete section $A \subseteq X$ for
$\RR_\Ggraph$ with $\mu(A) > 1 - \epsilon$ so that $\Ggraph_{\restriction A}$ has
finite connected components.
  
    \item  \textbf{(One Percent Theorem)} $\Ggraph$ is $\mu$-nowhere hyperfinite if and only if for every
  $\epsilon >0$ there exists a Borel complete section $A\subseteq X$ for
  $\RR_{\Ggraph}$ with $\mu (A)<\epsilon$ such that $\Ggraph_{\restriction A}$ is $\mu _{\restriction A}$-nowhere hyperfinite.
  \end{enumerate}
\end{theorem}

\begin{proof}

The new content of the theorem is the forward direction of (2). (In our
proof of Theorem~\ref{thm:Elek} we indicated how to prove part (1)).

By Corollary \ref{cor:maximal} (whose assumption  is satisfied when $\Ggraph$ is $\mu$-nowhere hyperfinite)  we may find a Borel $\mu$-a.e.\ one-ended
spanning subforest $\Tgraph\subseteq \Ggraph$ such that
$\RR_{\Tgraph}$ is $\mu$-maximal
among the $\Ggraph$-connected, $\mu$-hyperfinite  equivalence subrelations of $\RR_{\Ggraph}$.

Let $A_1=X$ and for $n\geq 1$ let $A_{n+1} = \{
x\in A_{n} \, : \, \mathrm{deg}_{\Tgraph_{\restriction A_{n}}}(x) \geq 2 \}$. Then $A_1\supseteq A_2\supseteq\cdots$ and $\mu (\bigcap _n A_n ) =0$ since $\Tgraph$ is one-ended, so after ignoring a null set we may assume that $\bigcap _n A_n = \emptyset$. Observe that $\RR_{\Tgraph_{\restriction A_n}}={ (\RR_{\Tgraph})}_{\restriction A_n}$ for all $n$.
Given $\epsilon >0$, let $n$ be so large that $\mu (A_n) < \epsilon /2$. Since $\Ggraph$ is $\mu$-nowhere hyperfinite and $\Tgraph$ is $\mu$-hyperfinite, the set $\Ggraph\mysetminus \RR_{\Tgraph}$ meets almost every connected component of $\Tgraph$. We may therefore find a subset $\Ggraph_0\subseteq \Ggraph\mysetminus \RR_{\Tgraph}$ which is incident with almost every connected component of $\Tgraph$ such that the set $B$, of vertices incident with $\Ggraph_0$, has measure $\mu (B)<\epsilon /2n$. (Finding $\Ggraph_0$ is easy when $\Tgraph$ is ergodic; in general we can simply use the ergodic decomposition of $\Tgraph$.)

\begin{claim}\label{claim:nowhere} $\Tgraph\cup (\Ggraph_{\restriction B})$ is $\mu$-nowhere hyperfinite.\end{claim}

\begin{proof}[Proof of the claim] Suppose otherwise. Then we may find a non-null $\RR_{\Tgraph\cup (\Ggraph_{\restriction B})}$-invariant set $D$ such that $(\RR_{\Tgraph\cup(\Ggraph_{\restriction B})})_{\restriction D}$ is hyperfinite. Then the equivalence relation $\mathcal{Q} := (\RR_{\Tgraph\cup\Ggraph_{\restriction B}})_{\restriction D} \sqcup (\RR_{\Tgraph})_{\restriction (X\mysetminus D )}$ is $\Ggraph$-connected and $\mu$-hyperfinite. By our choice of $B$, each component of $\Tgraph\cup(\Ggraph_{\restriction B})$ contains more than one component of $\Tgraph$, so $\mathcal{Q}$ properly contains $\RR_{\Tgraph}$ since $D$ is non-null. This contradicts the maximality of $\RR_{\Tgraph}$.
\end{proof}

For each $x\in B$ let $\pi (x)\in A_n$ denote the unique vertex in $A_n$ which is closest to $x$ with respect to the graph metric in $\Tgraph$, and let $p_x$ denote the unique shortest path through $\Tgraph$ from $x$ to $\pi (x)$. The length of each $p_x$ is at most $n-1$, so if we let $C$ denote the set of all vertices which lie along $p_x$ for some $x\in B$, then $\mu (C)\leq n\mu (B) < \epsilon /2$. Let $A=A_n\cup C$. Then $\mu (A)<\epsilon$ and
\[
\RR_{\Ggraph_{\restriction A}} \supseteq \RR_{\Tgraph_{\restriction A}} \vee \RR_{\Ggraph_{\restriction B}} =  (\RR_{\Tgraph})_{\restriction A} \vee \RR_{\Ggraph_{\restriction B}} \supseteq (\RR_{\Tgraph \cup (\Ggraph_{\restriction B})})_{\restriction A},
\]
so that $\Ggraph_{\restriction A}$ is $\mu _{\restriction A}$-nowhere hyperfinite since $\Tgraph \cup (\Ggraph_{\restriction B})$ is $\mu$-nowhere hyperfinite.
\end{proof}

\section{Planar graphs are measure treeable}\label{sec:planar}

In this section we show that graphs which are Borel planar (see Definition \ref{def:BorelPlanar}) are measure
treeable. This will imply that these graphs admit Borel a.e.\
one-ended spanning subforests with respect to any Borel probability
measure, provided the graph is $\mu$-nowhere two-ended (Corollary~\ref{cor:planarforest}).

\subsection{Preliminaries}\label{sect:Preliminaries}
A \define{$2$-basis} for a (undirected multi-)graph ${\Ggraph}$
is a collection $\mc{B}$ of simple cycles in ${\Ggraph}$ such that \\
(1) no
edge of ${\Ggraph}$ is contained in more than 2 cycles in $\mc{B}$, and
\\(2)
$\mc{B}$ generates the cycle space of ${\Ggraph}$, i.e., for every cycle
$C$ of ${\Ggraph}$ there exists $B_0,\dots, B_{n-1}\in \mc{B}$ with $1_C =
\sum _{i<n}1_{B_i} \pmod{2}$ (we identify each cycle with its unordered set of edges).

Assume for the moment that $\Ggraph$ is locally finite and $2$-vertex-connected.
Let $\wh{{\Ggraph}}$ be a planar embedding of ${\Ggraph}$,
and assume that $\wh{{\Ggraph}}$ is \define{accumulation-free}, i.e., every
compact subset of the plane meets only finitely many vertices and edges of
$\wh{{\Ggraph}}$. A \define{face} of $\wh{{\Ggraph}}$ is a connected component of the complement of $\wh{{\Ggraph}}$. The boundary of a face of $\wh{{\Ggraph}}$ is either a finite cycle or a bi-infinite 2-regular path.
 In the former case, the boundary cycle is called a \define{facial cycle}.
The following theorem of Thomassen relates $2$-bases to accumulation-free planar embeddings.

\begin{theorem}[Thomassen \cite{Tho80}]\label{thm:Thomassen}
Let ${\Ggraph}$ be a $2$-vertex-connected locally finite graph.
\begin{enumerate}
\item If $\wh{{\Ggraph}}$ is an accumulation-free planar embedding of ${\Ggraph}$, then the set $\mc{B}$ of facial cycles of $\wh{{\Ggraph}}$ forms a $2$-basis for ${\Ggraph}$.

\item If ${\Ggraph}$ admits a $2$-basis $\mc{B}$, then ${\Ggraph}$ has an accumulation-free planar embedding $\wh{{\Ggraph}}$ in which $\mc{B}$ is the set of facial cycles of $\wh{{\Ggraph}}$.
\end{enumerate}
\end{theorem}

In general, if a graph ${\Ggraph}$ is not necessarily $2$-vertex-connected,
then we can consider \define{blocks} of ${\Ggraph}$, i.e., maximal $2$-vertex-connected subsets of ${\Ggraph}$. Then a set $\mc{B}$ of cycles of ${\Ggraph}$ is a $2$-basis for ${\Ggraph}$ if and only if the restriction of $\mc{B}$ to each block of ${\Ggraph}$ is a 2-basis for that block. Furthermore, the incidence relation on blocks gives an acyclic graph structure on the set of blocks.

Let $\mc{B}$ be a $2$-basis for ${\Ggraph}$. The \define{dual} of ${\Ggraph}$ (with respect to $\mc{B}$) is the undirected (multi-)graph ${\Ggraph}^*$, with vertex set $\mc{B}$, and defined as follows. Let ${\Ggraph}_2\subseteq \Ggraph$ be the set of edges of $\Ggraph$ which belong to exactly two cycles from $\mc{B}$, and for each edge $e\in \Ggraph _2$ introduce a corresponding edge $e^*\in {\Ggraph}^*$ which connects $B_0$ with $B_1$, where $B_0,B_1\in\mc{B}$ are the unique cycles for which $e\in B_0\cap B_1$. Thus ${\Ggraph}^* = \{ e^* \, : \, e\in {\Ggraph}_2 \}$, and two distinct cycles $B_0, B_1\in \mc{B}$ are connected by one edge $e^*$ for each edge $e\in B_0\cap B_1$. The connected components of ${\Ggraph}^*$ correspond precisely to blocks of ${\Ggraph}$, i.e., two elements of $\mc{B}$ are in the same connected component of ${\Ggraph}^*$ if and only if they are contained in the same block of ${\Ggraph}$. 

Part (2) of Theorem \ref{thm:Thomassen} easily implies the following.

\begin{proposition}\label{prop:ddual}
Let ${\Ggraph}$ be a locally finite graph on $X$ which is $2$-vertex-connected. Let $\mc{B}$ be a $2$-basis for ${\Ggraph}$ and assume that every edge of ${\Ggraph}$ belongs to two distinct cycles from $\mc{B}$. Let $\Ggraph ^*$ be the dual of ${\Ggraph}$ with respect to $\mc{B}$. For each $x\in X$ let $x^* = \{ e ^* \, : \, e\text{ is incident with }x \}$. 
Then the collection $X^* = \{ x^* \, : \, x\in X \}$ is a $2$-basis for ${\Ggraph}^*$, every edge of ${\Ggraph}^*$ belongs to two distinct elements of $X^*$, and the map $x\mapsto x^*$ provides an isomorphism between ${\Ggraph}$ and ${\Ggraph}^{**}$.
It follows that $\mathcal{G}^*$ is $2$-vertex-connected.
\end{proposition}

Given a $2$-basis $\mc{B}$ of a planar graph $\Ggraph$, together with a subgraph $\Ggraph _0^*$ of the dual $\Ggraph ^*$, we define $\Ggraph \oPerpstar \Ggraph _0^*$ to be the subgraph of $\Ggraph$ obtain by removing all edges which cross an edge of $\Ggraph _0 ^*$, i.e.,
\begin{equation*}
\Ggraph \oPerpstar \Ggraph _0^* = {\Ggraph} \mysetminus \{ e\in {\Ggraph} \, : \, e^* \in {\Ggraph}^*_0 \}
\end{equation*}
(to be pronounced ``o minus star'' or ``cyclope'').
The following proposition explains how various properties of $\Ggraph _0^*$ are reflected in $\Ggraph \oPerpstar \Ggraph _0^*$.

\begin{proposition}\label{prop:duality}
Let ${\Ggraph}$ be an infinite locally finite graph on $X$ which is $2$-vertex-connected. Let $\mc{B}$ be a $2$-basis for ${\Ggraph}$ and assume that every edge of ${\Ggraph}$ belongs to two distinct cycles from $\mc{B}$. Let ${\Ggraph}^*$ be the dual of ${\Ggraph}$ with respect to $\mc{B}$. Let ${\Ggraph}^*_0$ be a spanning subgraph of ${\Ggraph}^*$, and let ${\Hgraph} =\Ggraph \oPerpstar \Ggraph _0^*$. Then:
\begin{enumerate}
\item ${\Hgraph}$ is acyclic if and only if each component of ${\Ggraph}^*_0$ is infinite;
\item Each component of ${\Hgraph}$ is infinite if and only if ${\Ggraph}^*_0$ is acyclic;
\item ${\Hgraph}$ is acyclic with the same connected components as ${\Ggraph}$ if and only if ${\Ggraph}_0^*$ is a one-ended spanning subforest of ${\Ggraph}^*$.
\end{enumerate}
\end{proposition}

\begin{proof}
By Proposition \ref{prop:ddual}, the graph ${\Ggraph}^*$ is $2$-vertex-connected. 
By Theorem \ref{thm:Thomassen}, there is an accumulation-free planar embedding $\wh{{\Ggraph}}$ of ${\Ggraph}$ such that $\mc{B}$ is the set of facial cycles of $\wh{{\Ggraph}}$. Let $F$ be a face of $\wh{{\Ggraph}}$. Then the boundary of $F$ is a finite cycle $B_F\in \mc{B}$ (as opposed to a bi-infinite line) and the closure of $F$ is compact since otherwise the complement of $F$ would be a compact set containing all of $\wh{{\Ggraph}}$, contradicting that ${\Ggraph}$ is infinite and the embedding is accumulation-free.

(1): Suppose that ${\Ggraph}_0 ^*$ is infinite. Let $C$ be a simple cycle in $\wh{{\Ggraph}}$ and let $F_0$ be a face of $\wh{{\Ggraph}}$ in the interior of $C$. Since the interior of $C$ is precompact and the embedding is accumulation-free, there must be some face $F_1$ in the exterior of $C$ such that $B_{F_0}$ and $B_{F_1}$ are connected by a path through ${\Ggraph}_0^*$. This path must contain some edge from $\{ e^*\, : \, e\in C \}$, and thus $C$ is not contained in ${\Hgraph}$. This shows that ${\Hgraph}$ is acyclic. Conversely, suppose that ${\Ggraph}_0^*$ has a finite connected component $\mc{B}_0\subseteq \mc{B}$. Then ${\Kgraph} _0 = \bigcup \mc{B}_0$ is a finite connected subgraph of ${\Ggraph}$, so the image $\wh{{\Kgraph}}_0$ of ${\Kgraph}_0$ has a unique face $F$ which contains infinitely many vertices of $\wh{{\Ggraph}}_0$. The boundary of $F$ is a cycle $C\subseteq {\Kgraph}_0$, and if $e\in C$ then $e$ is contained in only one cycle from $\mc{B}_0$ and hence $e^*\not\in {\Ggraph}_0^*$. This shows that $C$ is a cycle in ${\Hgraph}$.

(2): This follows from (1) and Proposition \ref{prop:ddual}.

(3): Suppose that ${\Ggraph}_0^*$ is a one-ended spanning subforest of ${\Ggraph}^*$. By part (1), ${\Hgraph}$ is acyclic, so it remains to show that ${\Hgraph}$ is connected. Suppose that $x,y\in X$ are connected by an edge $e_0\in {\Ggraph}$, but are not connected by an edge in ${\Hgraph}$. This means that $e_0^*\in {\Ggraph}_0^*$. Since ${\Ggraph}_0^*$ is a one-ended forest, there is a unique finite connected component of ${\Ggraph}_0^* \mysetminus \{ e_0 ^* \}$; let $\mc{B}_0$ be its set of vertices. Then the set $\{ e\in {\Ggraph} \, : \, e\neq e_0 \text{ and }e\text{ is in exactly one } B\in \mc{B}_0 \}$ is a path from $x$ to $y$ through ${\Hgraph}$, hence ${\Hgraph}$ is connected.
For the converse, by Proposition \ref{prop:ddual}, we may assume that ${\Ggraph}_0^*$ is acyclic with the same connected components as ${\Ggraph}^*$, toward the goal of showing that ${\Hgraph}$ is a one-ended spanning subforest of $\Ggraph$. If $\Hgraph$ had more than one end, then ${\Hgraph}$ would contain a bi-infinite line whose image in $\wh{{\Ggraph}}$ divides the plane into two connected components. Then there would be no path through ${\Ggraph}_0^*$ connecting cycles which lie on opposite sides of this bi-infinite line through $\Hgraph$, contrary to ${\Ggraph}_0^*$ being connected.
\end{proof}

\subsection{Borel planar graphs}
Our idea now is to use Lemma \ref{lem:extend}, along with part (3) of Proposition \ref{prop:duality} to obtain acyclic subgraphs of graphs which admit a Borel $2$-basis. To handle some technicalities, we will need the following Proposition.

\begin{proposition}[\cite{CMT-D}]\label{prop:backorbit}
Suppose that ${\Ggraph}$ is a locally finite Borel graph on $X$, and $A\subseteq X$ is Borel. Then there is an acyclic Borel function $f:[A]_{\RR_{\Ggraph}}\mysetminus A \ra [A]_{\RR_{\Ggraph}}$ whose back-orbits are all finite, and whose graph is contained in ${\Ggraph}$.
\end{proposition}

\begin{definition}[Borel planar graphs] \label{def:BorelPlanar}
Let $\Ggraph$ be a locally finite Borel graph on a standard Borel space $X$. A \define{Borel $2$-basis} for $\Ggraph$ is a $2$-basis $\mathcal{B}$ for $\Ggraph$ which is Borel when viewed as a subset of the standard Borel space of all finite subsets of $\Ggraph$. We say that $\Ggraph$ is \define{Borel planar} if it admits a Borel $2$-basis.
\end{definition}

Observe that if a locally finite Borel graph ${\Ggraph}$ admits a Borel $2$-basis $\mc{B}$, then the dual graph ${\Ggraph}^*$ is a Borel graph on $\mc{B}$.
We can now prove that Borel graphs which admit a Borel $2$-basis are measure treeable.

\begin{theorem}\label{thm:planar}
Let ${\Ggraph}$ be a locally finite Borel planar graph on $X$.
Let $\mu$ be a Borel probability measure on $X$. Then there exists a $\Ggraph$-invariant $\mu$-conull Borel set $X_0\subseteq X$ and an acyclic Borel subgraph ${\Ggraph}_0\subseteq {\Ggraph}$ with $(\RR_{\Ggraph_0})_{\restriction X_0}=(\RR_{\Ggraph})_{\restriction X_0}$.
\end{theorem}

\begin{proof}
It suffices to produce an acyclic Borel subgraph of $\Ggraph$ whose restriction to almost every block is connected. Thus, by working block by block we may assume without loss of generality that each component of $\Ggraph$ is $2$-vertex-connected.

Fix a Borel $2$-basis $\mc{B}$ for ${\Ggraph}$ and let ${\Ggraph}^*$ be the
dual graph on $\mc{B}$. For the rest of the proof, elements of $\mc{B}$
will be referred to as \define{facial cycles}.
Let ${\Ggraph}_1$ be the set of edges of ${\Ggraph}$ which are
contained in exactly one facial cycle. Let $\mc{B}_1\subseteq \mc{B}$ be
the set of all facial cycles which are incident with an edge from
${\Ggraph}_1$. Assume first that $\Ggraph _1$ meets every connected component of ${\Ggraph}$. Then, since $\Ggraph$ is $2$-vertex-connected, $\mc{B}_1$ meets every connected component
of ${\Ggraph}^*$. Apply Proposition \ref{prop:backorbit} to the graph
${\Ggraph}^*$ and the set $\mc{B}_1$ to obtain an acyclic subgraph
${\Ggraph}_0^*$ of ${\Ggraph}^*$ in which every connected component either
is one-ended and does not meet $\mc{B}_1$, or is finite and meets
$\mc{B}_1$ in exactly one point. Let $B\mapsto e(B)$ be a Borel function
selecting one edge $e(B)\in B\cap {\Ggraph}_1$ out of each $B\in \mc{B}_1$.
Then arguing as in Proposition \ref{prop:duality}, we see that the graph
${\Ggraph} \oPerpstar  \big( {\Ggraph}_0^* \cup \{ e(B) \, : \, B\in \mc{B}_1 \} \big)$ is acyclic
with the same connected components as $\Ggraph$, so the proof is complete in this case.

We may therefore assume for the rest of the argument that
${\Ggraph}_1=\emptyset$. We may also assume that no connected component of
${\Ggraph}$ is acyclic. It follows that every edge of ${\Ggraph}$ is
contained in two distinct facial cycles. This implies that both ${\Ggraph}$
and ${\Ggraph}^*$ are everywhere one-ended.

If we can show that $\Ggraph ^*$ has a Borel a.e.\ one-ended spanning subforest $\Tgraph ^*$ then we will be done, since then by Proposition \ref{prop:duality}, the subgraph ${\Ggraph} \oPerpstar  \Tgraph ^*$ of $\Ggraph$ would have the desired properties.

Toward this goal, choose, by the discussion preceding Corollary \ref{cor:maximal}, an acyclic aperiodic hyperfinite Borel subgraph ${\Tgraph}$ of ${\Ggraph}$.

Let $Y$ denote the $\RR_{\Ggraph}$-saturation of the set where ${\Tgraph}$ does not have two ends. Theorem \ref{thm:acyclic} and Lemma \ref{lem:extend} imply that ${\Ggraph}_{\restriction Y}$ has a Borel a.e.\ one-ended spanning subforest $\Tgraph _0$. Then by Proposition \ref{prop:duality}, the graph $\mathcal{L} = ({\Ggraph}_{\restriction Y}) ^* \oPerpstar  \Tgraph _0$ is an acyclic subgraph of $({\Ggraph_{\restriction Y}})^*$ with $\RR_{\mathcal{L}} = \RR_{{(\Ggraph_{\restriction Y})}^*}$ modulo $\mu$. While $\mathcal{L}$ may have some components with two ends, every component of ${\Ggraph}^*$ is one-ended, so we can apply Corollary \ref{cor:subtree} to find a Borel a.e.\ one-ended spanning subforest of $(\Ggraph_{\restriction Y} )^*$. By restricting our attention now to the complement of $Y$, we may assume without loss of generality that every connected component of $\Tgraph$ has two ends.

Let ${\Hgraph}^*$ be the subgraph of ${\Ggraph}^*$ given by ${\Hgraph}^* = {\Ggraph}^* \oPerpstar  \Tgraph$. Then ${\Hgraph}^*$ is an acyclic aperiodic subgraph of ${\Ggraph}^*$ by Proposition \ref{prop:duality}. Let $Z^*\subseteq \mathcal{B}$ denote the $\RR_{{\Ggraph}^*}$-saturation of the set where ${\Hgraph}^*$ does not have two ends, and let $Z\subseteq X$ denote the set of vertices incident with some facial cycle in $Z^*$. Applying Theorem \ref{thm:acyclic} and Lemma \ref{lem:extend} again (but this time applied to $\Ggraph ^*$), we see that ${\Ggraph}^*_{\restriction Z^*} = (\Ggraph_{\restriction Z} )^*$ has a Borel a.e.\ one-ended spanning subforest. Thus, by restricting our attention to the complement of $Z$, we may assume without loss of generality that every connected component of ${\Hgraph}^*$ and every connected component of $\Tgraph$ has two ends.

Notice that if $B\in \mc{B}$ is a facial cycle of ${\Ggraph}$, then $B$ cannot intersect more than $2$ connected components of ${\Tgraph}$, since otherwise $B$ would correspond to a "branch vertex" of ${\Hgraph}^*$ and hence ${\Hgraph}^*$ would have a connected component with more than $2$ ends (i.e., if $B$ intersects $3$ components $C_0$, $C_1$, and $C_2$, of ${\Tgraph}$ then there is a path from $B$ to infinity through ${\Hgraph}^*$ which leaves $B$ through an edge between $C_0$ and $C_1$, and likewise for an edge between $C_1$ and $C_2$ and an edge between $C_2$ and $C_0$, so these three paths to infinity correspond to 3 different ends in the ${\Hgraph}^*$-connected component of $B$), contrary to our assumption.

Consider the set $\mc{C}\subseteq \mc{B}$ of facial cycles which intersect exactly 1 connected component of ${\Tgraph}$. Each infinite connected component of ${\Hgraph}^*_{\restriction \mc{C}}$ must have one end, since if it had more than one end then some of these facial cycles would intersect more than 1 component of ${\Tgraph}$ (indeed, each bi-infinite line in $\Hgraph ^*$ separates the plane into two components, and each facial cycle on such a line contains vertices in both components, and hence these vertices belong to distinct components of $\Tgraph$). So the $\RR_{{\Ggraph}^*}$-saturation of the set where ${\Hgraph}^*_{\restriction \mc{C}}$ is infinite admits a Borel a.e.\ one-ended subforest.

We can therefore assume without loss of generality that ${\Hgraph}^*_{\restriction\mc{C}}$ has finite connected components. From here we will now show that $\RR_{\Ggraph}$ is $\mu$-hyperfinite.

Let $\Ggraph _0$ be obtained from $\Ggraph$ by removing every edge $e\in \Ggraph \mysetminus \Tgraph$ that belongs to some facial cycle in $\mc{C}$, and observe that the graphs $\Ggraph _0$ and $\Ggraph$ have the same connected components. This process, of removing edges from $\Ggraph$ to obtain $\Ggraph _0$, merges finite collections of faces (corresponding to components of $\Hgraph ^*_{\restriction \mc{C}}$) with a unique face of $\mc{B} \mysetminus \mc{C}$, and produces a $2$-basis $\mc{B} _0$ for $\Ggraph _0$. Now every $B\in \mc{B}_0$ intersects exactly two connected components of ${\Tgraph}$. The dual graph ${\Ggraph}_0^*$, of ${\Ggraph}_0$ with respect to $\mc{B}_0$, is a minor of ${\Ggraph}^*$ and, if we define ${\Hgraph}_0^* = {\Ggraph}_0^* \oPerpstar  \Tgraph$, then every connected component of ${\Hgraph}_0^*$ is a bi-infinite line. 
Therefore, after applying the analogous argument as well to the dual, we may assume without loss of generality that every connected component of ${\Hgraph}^*$ and of ${\Tgraph}$ is a bi-infinite line.

Now, every line in ${\Tgraph}$ intersects exactly two lines in ${\Hgraph}^*$, and likewise, every line in ${\Hgraph}^*$ intersects exactly two lines in ${\Tgraph}$, so the lines of ${\Tgraph}$ within any fixed $\RR_{\Ggraph}$-class are themselves arranged in the structure of a single bi-infinite $\Z$-line (i.e., take two lines of ${\Tgraph}$ to be adjacent if they both intersect the same line in $\Hgraph ^*$). 
This implies: 
\begin{claim}\label{claim:2-amen}
$\RR_{\Ggraph}$ is 2-amenable in the sense of Jackson--Kechris--Louveau \cite{JKL02}.
\end{claim}

\begin{proof}[Proof of Claim]
For $x\in X$, let $\mathcal{C}_x$ denote the set of all $\RR_{\Tgraph}$-classes contained in $[x]_{\RR_{\Ggraph}}$, and let $\mathcal{C}= \{ (x,C): x\in X, \ C\in \mathcal{C}_x \}$.
Observe that $\mathcal{C}$ is standard Borel, being the quotient space of $\RR_{\Ggraph}$ by the smooth equivalence relation $\tilde{\RR}_{\Tgraph}$ that identifies two pairs in $\RR_{\Ggraph}$ when their left coordinates are equal and their right coordinates are $\RR_{\Tgraph}$-related. 

For each $x\in X$ and $j,k\in \N$ let $a^j_x:[x]_{\RR_{\Tgraph}}\to [0,1]$ be the uniform probability vector on the radius $j$ ball about $x$ in $\Tgraph$, and let $b^k_x :\mathcal{C}_x\to [0,1]$ be the uniform probability vector on the radius $k$ ball about $[x]_{\RR_{\Tgraph}}$ in the $\Z$-line structure on $\mathcal{C}_x$.
Then $(a^j)_{j\in \N}$ is a Borel Reiter sequence witnessing the $1$-amenability of $\RR_{\Tgraph}$, and $(b^k)_{k\in \N}$ is a Borel Reiter sequence witnessing the $1$-coamenability of $\RR_{\Tgraph}$ within $\RR_{\Ggraph}$.
The $1$-coamenability assertion means that: (1) for each $k\in \N$, the function $b^k=(b^k_x)_{x\in X}$ assigns, to each $x\in X$, a probability vector $b^k_x$ on the set $\mathcal{C}_x$; (2) for each $k\in\N$ the map $\RR_{\Ggraph}\to [0,1]$, sending $(x,y)$ to $b^k_x([y]_{\RR_{\Tgraph}})$, is Borel; and (3) $\lim_{k\to\infty} \| b^k_x-b^k_y\|_1  = 0$ for all $(x,y)\in\RR_{\Ggraph}$.

By the Lusin-Novikov Uniformization Theorem there exists a Borel function $s:\mathcal{C}\to X$ such that the map $(x,y)\mapsto (x,s(x,[y]_{\RR_\Tgraph}))$ is a selector for $\tilde{\RR}_{\Tgraph}$.   
For $j,k\in\N$ and $x\in X$ define the probability vector $c^{(k,j)}_x \in \ell ^1([x]_{\RR_{\Ggraph}})$ by
\[
c^{(k,j)}_x = \sum _{C\in\mathcal{C}_x}b^k_x(C)a^j_{s(x,C)}.
\]
For each $(x,y)\in\RR_{\Ggraph}$, a standard computation shows that
\[
\limsup_{k\to\infty}\limsup_{j\to\infty}\| c^{(k,j)}_x - c^{(k,j)}_y\|_1 = 0,
\]
and hence the sequence $(c^{(k,j)})_{(k,j)\in\N\times\N}$ witnesses that $\RR_{\Ggraph}$ is $2$-amenable.
\qedhere[Claim].
\end{proof}
By \cite{JKL02} (see Proposition 4.1 of \cite{EOSS:hyperfiniteness}), since $\RR_\Ggraph$ is $2$-amenable, it is $\mu$-amenable, and hence by \cite{CFW81} $\RR_{\Ggraph}$ is $\mu$-hyperfinite. 
The conclusion of the theorem now follows.
\end{proof}

\begin{corollary}\label{cor:planarforest}
Let ${\Ggraph}$ be a locally finite Borel planar graph on $X$. Let $\mu$ be a Borel probability measure on $X$. Assume that
${\Ggraph}$ is $\mu$-nowhere two-ended. Then $\Ggraph$ has a Borel a.e.\ one-ended spanning subforest.
\end{corollary}

\begin{proof}
This follows from Theorem \ref{thm:planar} and Corollary \ref{cor:subtree}.
\end{proof}

We also note the following variation of Theorem \ref{thm:planar}.

\begin{theorem}\label{thm:varplanar}
Let ${\Ggraph}$ be a locally finite Borel planar graph on $(X,\mu )$, let $\mc{B}$ be a Borel $2$-basis for $\Ggraph$, and let $\Ggraph ^*$ be the associated dual
graph. Let $\Hgraph \subseteq \Ggraph$ denote the set of edges which
belong to at most one $B\in \mc{B}$. Assume that $\Ggraph ^*$ is $\mu$-nowhere two-ended. Then there exists a $\Ggraph$-invariant $\mu$-conull Borel set $X_0\subseteq X$ and an acyclic Borel graph ${\Ggraph}_0$ with $\Hgraph \subseteq \Ggraph _0 \subseteq {\Ggraph}$ with $(\RR_{\Ggraph_0})_{\restriction X_0} = \RR_{{\Ggraph}_{\restriction X_0}}$.
\end{theorem}

\begin{proof}
Let $Y$ denote the set of points $y\in X$ which are not incident to any edge in $\Hgraph$, and for $y\in Y$ let $y^* = \{ e^* \, : \, e\text{ is incident to }y \}$. Then the set $Y^* = \{ y^* \, : \, y\in Y \}$ is a Borel $2$-basis for $\Ggraph ^*$. Therefore, by Corollary \ref{cor:planarforest}, $\Ggraph ^*$ has a Borel a.e.\ one-ended spanning subforest $\Tgraph$. Then the graph $\Ggraph _0 = \Ggraph \oPerpstar \Tgraph$ is the desired subgraph.
\end{proof}

\section{\texorpdfstring{Measure strong treeability of $\Isom(\HH ^2)$}{Measure strong treeability of H2}}

Throughout this section, we let $G$ denote the group $\Isom(\HH ^2)$ of all
isometries of the hyperbolic plane $\HH ^2 = \{ z\in \C \, : \, \mathrm{Im}(z) > 0 \}$, and we let $K$ denote the stabilizer of $i\in \HH ^2$. Then $K$ is a compact subgroup of $G$.
The term cross-section is defined in the Appendix~\ref{sect: Treeability for locally compact groups}.

\begin{theorem}\label{thm: Isom H2 meas tree}
The group $G=\Isom(\HH ^2)$ is measure strongly treeable. That is, given any free Borel action $G\curvearrowright X$ of $G$ on a standard Borel space $X$ and a Borel probability measure $\mu$ on $X$, there exists a $G$-invariant $\mu$-conull Borel set $X_0\subseteq X$ such that $(\RR_{G})_{\restriction Y}$ is Borel treeable for any Borel cross section $Y\subseteq X_0$.
\end{theorem}

\begin{proof}
Let $G\curvearrowright X$ be a free Borel action of $G$ on a standard Borel space $X$, and let $\mu$ be a Borel probability measure on $X$, which we may assume is $\RR_G$-quasi-invariant.
$\RR_G$ descends to a Borel equivalence relation $\wh{R}_G$ on $K\backslash X$. Define $d:\wh{R}_G\ra [0,\infty)$ by $d(K\cdot x , \, Kg\cdot x) = d_{\HH ^2}(g^{-1}(i),i)$, where $d_{\HH ^2}$ denotes the hyperbolic metric on $\HH ^2$. Then $d$ is a Borel function on $\wh{R}_G$ and the restriction of $d$ to each $\wh{R}_G$-class is a metric (since $G$ acts by isometries on $\HH ^2$) which is isometrically isomorphic with $(\HH ^2 , d_{\HH ^2})$.

Let $U\subseteq G$ be an open subset of $G$ whose closure is compact and with $K\subseteq U$. Applying Theorem \ref{thm:cocomp} to the closure of $U$, we may find a Borel cross section $Y\subseteq X$ for the action which is cocompact and satisfies $U\cdot y_0 \cap U\cdot y_1 = \emptyset$ for distinct $y_0,y_1 \in Y$. Let $\nu$ be an $(\RR_G)_{\restriction Y}$-quasi-invariant Borel probability measure on $Y$ given by Proposition \ref{prop:Ymeas}, coming from the canonical $(\RR_G)_{\restriction Y}$-quasi-invariant measure class associated to $Y$ and $\mu$. We will be done once we show that $(\RR_G)_{\restriction Y}$ is treeable on a $\nu$-conull invariant Borel set $Y_0$, since then the $\RR_G$-saturation of $Y_0$ will be the $\mu$-conull Borel set $X_0$ which satisfies the conclusion of the theorem.

Let $\wh{Y}$ and $\wh{\nu}$ denote the images of $Y$ and $\nu$ respectively in $K\backslash X$. 
The intersection of $\wh{Y}$ with each $\wh{R}_G$-class forms a Delone set in that copy of $\HH ^2$. 
For each $\wh{y}\in\wh{Y}$ let $D(\wh{y})\subseteq [\wh{y}]_{\wh{R}_G}$ denote the corresponding Dirichlet--Voronoi region in the metric space $([\wh{y}]_{\wh{R}_G}, d_{\restriction [\wh{y}]_{\wh{R}_G}})$, i.e., $D(\wh{y})$ consists of all points of $[\wh{y}]_{\wh{R}_G}$ whose distance to $\wh{y}$ is strictly less than their distance to any other point of $\wh{Y}\cap [\wh{y}]_{\wh{R}_G}$. Since $G$ acts properly on $\HH ^2$ and $K\subseteq U$, there exists an $r_0 > 0$ such that $d_{\HH ^2}(g^{-1}(i),i ) > r_0$ for all $g\not\in U$. Therefore, for each $\wh{R}_G$-class $C$, the balls of radius $r_0$ about points of $\wh{Y}\cap C$ are pairwise disjoint. Moreover, since $Y$ is cocompact, there exists an $r_1>r_0$ such that the balls of radius $r_1$ about points of $\wh{Y}\cap C$ cover all of $C$.
Therefore, each Dirichlet region $D(\wh{y})$ is the interior of a bounded hyperbolic polygon, and for each $\wh{R}_G$-class $C$, we have $\bigcup _{\wh{y}\in \wh{Y}\cap C} \ol{D(\wh{y})} = C$. Then the complement of all of the Dirichlet regions naturally carves out a Borel graph ${\Hgraph}$ on the Borel set $Z\subseteq K\backslash X$, of all points which are vertices of a boundary polygon of one of the Dirichlet regions, and we have $\RR_{{\Hgraph}}=(\wh{R}_{G})_{\restriction Z}$.
Let $\wh{{\Ggraph}}$ be the Borel graph on $\wh{Y}$ obtained as the dual graph to ${\Hgraph}$, i.e., $\wh{y}_0$ and $\wh{y}_1$ are connected by an edge $e^*$ if and only if the hyperbolic polygons $\ol{D(\wh{y}_0)}$ and $\ol{D(\wh{y}_1)}$ share an edge $e$. Then $\RR_{\wh{{\Ggraph}}}=(\wh{R}_{G})_{\restriction \wh{Y}}$. For each $z\in Z$ let $B_z = \{ e^* \, : \, e\in {\Hgraph}\text{ is incident with }z \}$. Then the collection $\{ B_z \, : \, z\in Z \}$ is a Borel $2$-basis for $\wh{{\Ggraph}}$, so by Theorem \ref{thm:planar} there exists a $\wh{\nu}$-conull $\RR_{\wh{{\Ggraph}}}$-invariant set $\wh{Y}_0\subseteq \wh{Y}$ and an acyclic Borel subgraph $\wh{{\Ggraph}}_0\subseteq \wh{{\Ggraph}}$ with $(\RR_{\wh{{\Ggraph}}_0})_{\restriction \wh{Y}_0} =(\RR_{\wh{{\Ggraph}}})_{\restriction \wh{Y}_0}=(\wh{R}_{G})_{\restriction \wh{Y}_0}$.

Let $Y_0\subseteq Y$ be the preimage of $\wh{Y}_0$ under the projection to $K\backslash X$. By our choice of $Y$, each $\RR_K$-class contains at most one point from $Y$.
Then $\wh{{\Ggraph}}_0$ lifts to a treeing ${\Ggraph}_0$ of $(\RR_{G})_{\restriction Y_0}$.
\end{proof}

\begin{corollary}\label{cor:subgroup}
The groups $\Isom(\HH ^2)$, $\PSL_2(\R )$, and $\SL_2(\R )$ are all measure strongly treeable, as are all of their closed subgroups; thus, for instance, all surface groups $\pi _1 (\Sigma _g )$, $g\geq 2$, are measure strongly treeable. In particular, each of these groups has fixed price.
\end{corollary}

\begin{proof}
This follows since measure strong treeability is closed under taking closed subgroups, and a compact by measure strongly treeable group is itself measure strongly treeable (Theorem \ref{thm:closure}).
The last statement follows from Proposition \ref{prop:stronglytreeablefixedprice}.
\end{proof}

\begin{remark}\label{rem:lcsccost}
To compute the actual cost of these lcsc groups $G$ (with respect to a fixed Haar measure $\lambda$), one may use the formula $\mathrm{cost}(G,\lambda ) -1 = (\mathrm{Cost}(\Gamma )-1)/\mathrm{covol}_{\lambda}(\Gamma )$, where $\Gamma$ is any lattice of $G$.
\end{remark}

\subsection{Groups with planar Cayley graphs}
\begin{theorem}\label{thm:planargroup}
Let $\Gamma$ be a finitely generated group and suppose that $\Gamma$ has a Cayley graph which is planar. Then $\Gamma$ is measure strongly treeable.
\end{theorem}
\begin{proof}
By \cite[Theorem 5.1]{Dr06}, $\Gamma$ is finitely presented.
It follows from Dunwoody's accessibility Theorem \cite{Dunwoody-1985}
that $\Gamma$ is the Bass-Serre fundamental group of a finite graph of groups whose edge groups are all finite and whose vertex groups are finite or one-ended with planar Cayley graphs (by \cite{Babai-1977}).
Therefore, by Theorem \ref{thm:closure}-(8), it suffices to prove the theorem when $\Gamma$ has one end.
In this case, it is $2$-vertex-connected and by Thomassen's Theorem (Theorem~\ref{thm:Thomassen}) it admits an accumulation-free planar embedding.
If $\Gamma$ is amenable then $\Gamma$ is measure strongly treeable by \cite{CFW81}.
Otherwise, by \cite[Theorem 3.1]{Bab97}, $\Gamma$ is isomorphic to a discrete subgroup of $\Isom(\HH ^2)$, so we are done by Corollary \ref{cor:subgroup}.
 \end{proof}

It follows that surface groups are measure strongly treeable.  Hjorth's theorem \cite{Hjo-cost-att} therefore implies that the orbit equivalence relation of any free \pmp{} action of a (infinite) surface group is also generated by a free action of a free group of the same cost.  Does the converse hold?  That is:
\begin{question}\label{question:Hjorthsurface}
 Can any orbit equivalence relation of a free \pmp{} action of a free group also be generated by a free action of a surface group of appropriate cost?
\end{question}

\section{Elementarily free groups and towers}\label{sec:elemfree}

A group is said to be \define{elementarily free} if it satisfies the same first-order sentences as a free group. In \cite{BTW-07},  it is shown that every finitely generated elementarily free group is treeable.
 In this section, we show that such groups are even measure strongly treeable.

\begin{thm}\label{elemfreetreeable}
  Every finitely generated elementarily free group is measure strongly treeable.
\end{thm}

As in \cite{BTW-07}, our argument relies crucially upon the description of the finitely generated elementarily free groups as fundamental groups of certain ``tower'' spaces defined inductively (see Theorem~\ref{th: elem free gps as large towers} below)
uncovered in \cite[Proposition 6]{SeVI} (cf. also \cite{KhM1}) and completed later in \cite{Guirardel-Levitt-Sklinos-2020}.

We do not intend to state a characterization of the elementarily free groups and we will content ourselves with the following construction (Theorem~\ref{th: elem free gps as large towers}).

Let $U$ be a path-connected CW-complex and let $\Sigma$ be a connected compact surface with $k\geq 1$ boundary components and with $\pi_1(\Sigma)\simeq \FF_d$, with $d\geq 2$ (equivalently its Euler characteristic satisfies $\chi(\Sigma)\leq -1$).
Let $V$ be the quotient of $U \disjointunion \Sigma$ by identifying each boundary component $\beta_j$ of $\Sigma$ with a loop $\gamma_j$ in $U$ corresponding to infinite order elements of $\pi_1(U)$.
We say that the space $V$ is an \define{IFL} 
\define{over} $U$.\footnote{IFL may stand for Inductive or Injective - Free or Fuchsian -  Level or Loft.}

 Let $q \colon U \disjointunion \Sigma \to V$ be the  corresponding quotient map and let $q_*$ denote the induced map on fundamental groups, using a common base point $u\in U$.  Then (by standard Bass-Serre theory) $q_*$ realizes an isomorphism $q_*(\pi_1(U))\simeq \pi_1(U)$.

If we iterate inductively this construction, via
a sequence $V_0, V_1, \cdots, V_n$ of CW-complexes where each $V_{h+1}$ is an IFL over $V_h$, we say that the
CW-complex $V_n$ is an \define{IFL tower over} $V_0$.

\begin{theorem}[{\cite{SeVI}, \cite{KhM1},\cite{Guirardel-Levitt-Sklinos-2020}}]
\label{th: elem free gps as large towers}
Every finitely generated elementarily free group can be described as the fundamental group of an
IFL tower over a graph.
\end{theorem}

Both theorems are proved in Section~\ref{sect: back to elementarily free groups}.

\begin{remark}
  The methods of this section establish measure strong treeability for fundamental groups of a larger class of spaces than just elementarily free groups.

In fact, the characterization of elementarily free groups requires moreover the existence of a retraction, and more generally of a non-degenerate
map for extended towers \cite[Definition 4.5, Proposition 4.21]{Guirardel-Levitt-Sklinos-2020},
from the fundamental groups of $V_{h+1}$ to that of $V_h$ satisfying certain conditions in restriction to $\pi_1(\Sigma)$. Our methods do not assume anything like that.
  For us, it is enough that the space $V_{h+1}$ is an IFL over $V_{h}$: the boundary loops $\beta_j$ of $\Sigma$ are glued to elements of infinite order in $\pi_1(V_{h})$ (thus ensuring the injection of the fundamental groups in the natural associated Bass-Serre decomposition) and
  the Euler characteristic of the compact surface $\Sigma$ satisfies $\chi(\Sigma)\leq -1$ (compare \cite{SeVI} where $\chi(\Sigma)\leq -2$ or $\Sigma$ is a punctured torus).
 This allows us some exceptional surfaces $\Sigma$ that are explicitly forbidden for $\omega$-residually free tower spaces:
  the pair of pants (i.e., the sphere minus three disks) and concerning the non-orientable ones, the projective plane minus two disks and the Klein bottle minus one disk.

   Moreover our starting space $V_0$ of height $0$ may include such surfaces as the Klein bottle or the non-orientable genus 3 surface, whose fundamental groups are known to be not even $\omega$-residually free. It may also include  as $V_0$ any space with infinite amenable fundamental group.
 \end{remark}

\subsection{\texorpdfstring{Measure free factors}{Measure free factors of Fn}}
\label{sect:Meas free fact}

Several notions, first introduced in the framework of \pmp{} equivalence relations, admit direct generalizations to the Borel context (possibly with the presence of a measure).
This is the case of the notions of \define{free product decomposition} $\RR=\RR_1*\RR_2$ and \define{free product decomposition with amalgamation} introduced in \cite[D\'ef. IV.9, D\'ef. IV.6]{Gab00a}.
This is also the case for the following: a subgroup $\Lambda\leq \Gamma$ is called a \define{measure free factor} of $\Gamma$ if for some
\pmp{} free action $\Gamma\actson^{a} (X,\mu)$, there exists a subrelation $\mathcal{S}\subrel \RR$ such that  $\RR_a=\RR_{a (\Lambda)} * \mathcal{S}$ ($\mu$-a.e.) \cite[Def. 3.1]{Ga05}. See also \cite{Alonso-2014} for more results on measure free factors.

Similarly, a subgroup $\Lambda\leq \Gamma$ is called a \define{measure strong free factor}
 if for every free Borel action $\Gamma\actson^{a} X$ and for every Borel auxiliary probability measure $\mu$ on $X$,  there exists a subrelation $\mathcal{S}\subrel \RR$ such that
$\RR_a=\RR_{a (\Lambda)} * \mathcal{S}$ on a $\Gamma$-invariant subset $X_0\subseteq X$ of full $\mu$-measure.
If moreover the subrelation $ \mathcal{S}$ can always be chosen to be treeable, then
$\Lambda$ is called a \define{measure strong free factor} of $\Gamma$ \define{with treeable complement}.

The main technical result of this section is the following.
\begin{theorem}\label{th: pi1 V meas str. treeable over U}
If $V$ is an IFL  over $U$, then
$\pi_1(U)=q_*(\pi_1(U))$ is a measure strong free factor of  $\pi_1(V)$ with treeable complement.
\end{theorem}

\begin{corollary}\label{cor: pi1 U some-treeab implies same for pi1 V}
If moreover $\pi_1(U)$ is treeable, strongly treeable, or measure strongly treeable, then the same holds for $\pi_1(V)$.
\end{corollary}
This  corollary follows immediately from Theorem~\ref{th: pi1 V meas str. treeable over U} (using co-induction when $\pi_1(U)$ is only treeable).

\begin{example}\label{ex: meas str treeable non-elemen free}
Take a pair of pants and glue each of its boundary components to a single extra circle with indices $d_1, d_2, d_3\in \Z\setminus \{0\}$ (and similarly for the two other  exceptional allowed surfaces, the projective plane minus two disks or the Klein bottle minus one disk, as soon as $\sum_i |d_i|\geq 3$).
The resulting group is not elementarily free by \cite[Lemma 10.20]{Guirardel-Levitt-Sklinos-2020} but it is measure strongly treeable by Theorem~\ref{th: pi1 V meas str. treeable over U}.
\end{example}
We also obtain an extension of \cite[Th. 3.3]{Ga05} to the non \pmp{} case and for every free action.
\begin{corollary}
Let $\Gamma$ be the free group on $a_1,\dots , a_n, b_1,\dots , b_n$.
The product of commutators $\prod[a_i,b_i]$ is a measure strong free factor in $\Gamma$.
The analogous statement holds in the free group on $a_1,\dots , a_n$ ($n\geq 2$) for the product of the squares: $a_1^2\cdots a_n^2$.
\end{corollary}
It follows that fundamental groups of "branched surfaces" from \cite{Ga05} are measure strongly treeable, and in particular strongly treeable.

Corollaries \ref{cor:gtht^{-1}} and \ref{cor:g0g1} in \S\ref{sec:cyclicmff} below provide further measure strong free factors of free groups.

Similarly to \cite[Question 3.10]{Ga05}, we ask the question:
\begin{question}
 What are all the measure strong free factors of the free group $\FF_r$?
\end{question}

\subsection{Extended IFL towers - Non connected grounds}
\label{sect: extended large towers}

Let $\Sigma$ be a connected compact surface with $k> 0$ boundary components $\{\beta_j \suchthat 1\leq j \leq k\}$ such that $\pi_1(\Sigma)$ is a non-cyclic free group. Consider a collection of disjoint CW-complexes $U_1, U_2, \cdots, U_r$ and a collection  $\{\gamma_j \suchthat 1\leq j \leq k\}$ of loops of infinite order in $\pi_1(U_{f(j)})$ for some surjective map $f:\{1\leq j \leq k\}\to \{1\leq i\leq r\}$.
Let $W$ be the space obtained by gluing $\Sigma$ to the $\sqcup U_i$ by attaching its $j$-th boundary component to $\gamma_j$.
This connected space $W$ is an \define{extended IFL over} the $U_i$'s.
\\
Let $U^*$ be the connected space obtained by attaching for each $i=2, 3, \cdots, r$ an arc $\xi_i$ from (the base point of) $U_1$ to (the base point of)  $U_i$.
\\
Let $V$ be the space obtained from $W$ by adding the arcs $\xi_i$.
\\
On the one hand, Theorem~\ref{th: pi1 V meas str. treeable over U} ensures that $\pi_1(U^*)=\pi_1(U_1)*\pi_1(U_2)*\cdots *\pi_1(U_r)$ is a measure strong free factor of $\pi_1(V)$ with treeable complement and, on the other hand, $\pi_1(V)=\pi_1(W)*\FF_{r-1}$.
Iterating the construction would of course produce \define{extended IFL towers over} the first space.

\begin{corollary}\label{cor: extended etage-treeability}
Let $W$ be an extended IFL over $U_1, U_2, \cdots, U_r$.
If the $\pi_1(U_i)$ are all treeable, strongly treeable, or measure strongly treeable, then the same holds for  $\pi_1(W)$.
\end{corollary}
\begin{remark}\label{rem:no measure strong free factor for extended tower}
Observe that there is not too much to expect in terms of measure strong free factor involving all of the
$\pi_1(U_i)$, as exemplified by Corollary~\ref{cor: a one-rel amalg product}.
\end{remark}

\begin{proof}[Proof of Corollary~\ref{cor: extended etage-treeability}]
When the $\pi_1(U_i)$ all satisfy one of the properties, then their free product also does (Theorem~\ref{thm:closure}-(8)).
By Corollary~\ref{cor: pi1 U some-treeab implies same for pi1 V},
 the same then holds for $\pi_1(V)$.
It follows from Theorem~\ref{thm:closure}-(5) that
treeability,  strong treeability, and  measure strong treeability respectively
pass from $\pi_1(V)$ to its subgroup $\pi_1(W)$.
\end{proof}

An immediate application of Corollary~\ref{cor: extended etage-treeability} gives the following.
\begin{corollary}\label{cor: a one-rel amalg product}
Let $r\geq 3$ and $\Gamma_1, \Gamma_2, \cdots, \Gamma_r$ be countable groups and let $\gamma_i\in \Gamma_i$ be an infinite order element for each $i=1, 2, \cdots, r$.
If the $\Gamma_i$ are treeable, strongly treeable, or measure strongly treeable, then the same holds for
\[\Gamma=\left\langle \Gamma_1, \Gamma_2, \cdots, \Gamma_r \;\middle\vert\;  \prod_{i=1}^r \gamma_i=1\right\rangle=\left(\Gamma_1*\Gamma_2*\cdots *\Gamma_r\right)/\langle\langle \, \prod_{i=1}^r \gamma_i\, \rangle\rangle.\]
\end{corollary}
As a particular case, we recover $\MSTreeable$ for the group $\Gamma=\langle a_1, a_2, s \,\vert\, a_1^2a_2^2s^d=1\rangle$ of Example~\ref{ex: meas str treeable non-elemen free} with the punctured Klein bottle.

\begin{remark}
Observe that the analogous statement fails for $r=2$ as witnessed by the group $\Gamma=\langle t_1, t_2 \,\vert\,  t_1^{d_1}t_2^{d_2}=1\rangle $
which has cost $1$ (as an amalgamated product $\Z*_{\Z} \Z$) and is non-treeable when non-amenable \cite[Corollary VI.22]{Gab00a}, e.g.,
when the indices $|d_i|$ are $\geq 3$.
\end{remark}

\begin{proof}This is exactly the fundamental group of an extended IFL over spaces $U_i$ with $\pi_1(U_i)\simeq \Gamma_i$ when $\Sigma$ is a sphere minus $r$ disks.\end{proof}

\subsection{Proof of Theorem~\ref{th: pi1 V meas str. treeable over U}}\label{subsection:mstproof}

\begin{proof}[Proof of Theorem~\ref{th: pi1 V meas str. treeable over U}]
Fix an enumeration $\{\beta_j \suchthat 1\leq j \leq k\}$ of  the (disjoint) boundary circles of $\Sigma$ based at $u_j\in \beta_j$.
Fix also for all $j\geq 2$ (if there are any) simple paths $t_j$ from $u_j$ to $u_1$ whose interiors are mutually disjoint and disjoint from the $\beta_j$.
For sake of notation, set $t_1$ to be the trivial path at $u_1$.
Let $b_j$ be the corresponding loops based at $u_1$ (and their class in $\pi_1(\Sigma,u_1)$), i.e.,
 $b_j=t_j^{-1} \beta_j t_j$ (for  $j=1, \cdots, k$).
Let $\{a_i : i \leq l$\} enumerate the other generators of a standard presentation of $\pi_1(\Sigma, u_1)$ as a surface with boundary with a single relation $R(a_i,b_j)$. More precisely:

\begin{itemize}
\item
If $\Sigma$ is orientable of genus $g$ then
 $\chi(\Sigma)=-2g-k+2$ and
the standard presentation is given by
$\pi_1(\Sigma)=\langle \{a_i\}_{i=1}^{2g},\{b_j\}_{j=1}^{k} \,\vert\, R((a_i)_i,(b_j)_j) \rangle$
with $R((a_i)_i,(b_j)_j)=\prod_{i=1}^{g}[a_{2i-1},a_{2i}]\prod_{j=1}^{k} b_j$. We set $l=2g$.

\item
If $\Sigma$ is non-orientable of genus $g$ then
$\chi(\Sigma)=-g-k+2$ and the standard presentation is given by
$ \pi_1(\Sigma)= \langle \{a_i\}_{i=1}^{g},\{b_j\}_{j=1}^{k} \,\vert\, R((a_i)_i,(b_j)_j)\rangle$
with
 $R((a_i)_i,(b_j)_j)=  \prod_{i=1}^{g}a_i^2 \prod_{j=1}^{k} b_j$. We set $l=g$.
\end{itemize}
Notice that $\pi_1(\Sigma)$ is isomorphic to a free group on $l+k-1$ generators since $k\geq 1$.

We fix a base point $u\in U$ which we identify with its image in $V$.
\begin{claim} The group
$\pi_1(V,u)$ splits as a graph-of-groups on two vertices, with vertex groups $\pi_1(U,u)$ and $\pi_1(\Sigma, u_1)$, and with $k$ edges joining the two vertices with edge groups isomorphic to $\langle b_j\rangle$.
This leads to the presentation:
  $$\pi_1(V,u)=\langle \pi_1(U,u), a_i,t_j \,\vert\,  R((a_i)_i,(t_j^{-1}\gamma_j t_j)_j)= t_1=1 \rangle.$$
\end{claim}

\begin{remark}
Conversely, a group admitting such a graph-of-groups splitting  with $\gamma_j$ of infinite order in $\pi_1(U,u)$ satisfies the assumptions of Theorem~\ref{th: pi1 V meas str. treeable over U}.
\end{remark}
\begin{proof}[Proof of the claim] The claim is quite clear, but we
take advantage of this part to introduce our notations.
Let $\alpha_i$ be simple loops based at $u_1$ representing $a_i$ which are pairwise disjoint and disjoint from each $t_j$ and $b_j$
 (see Figure 2).
\begin{center}
\begin{figure}[h!]
\resizebox{75mm}{!}{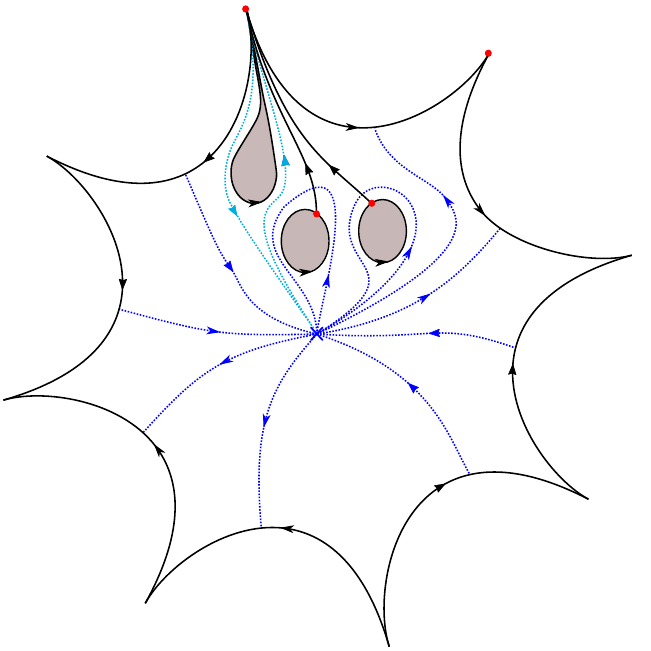} \hskip 2mm
\resizebox{60mm}{!}{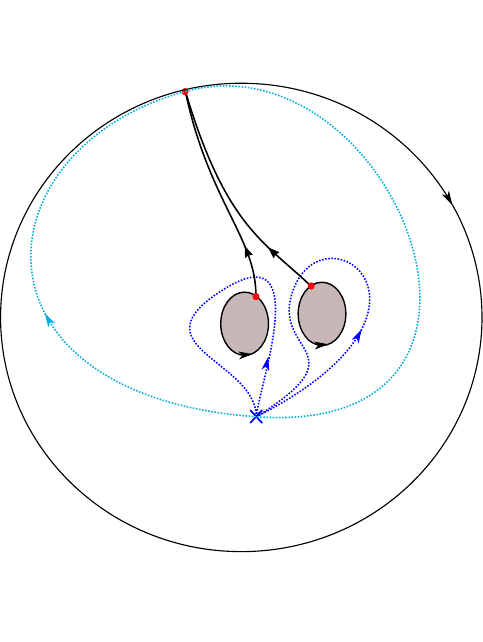}
\label{pict: hyperbolic-octogon-3-holes}
\caption{The graphs $\GGG_1$ and $\GGG^*$ for a genus two surface minus three disks and for a sphere minus three disks}
\end{figure}
\end{center}
The vertices $\{u_j\}$ and the paths $\{\alpha_i, \beta_j, t_j\}$ draw a connected graph $\GGG_1$ on $\Sigma$.
Its complement is a connected $2$-cell $\sigma$, homotopically equivalent to a disk. This defines a cellular decomposition $\CC$ of $\Sigma$.

The polyhedral dual graph has one vertex $\sigma^*$ (in the interior of $\sigma$).  Its edges have two forms: $\tau_j$ is a loop turning around $\beta_j$ and crossing $t_j$, while $\nu_i$  is a loop crossing $\alpha_i$. Observe that they do not cross the $\beta_j$. The loop $\tau_1$ has a particular status since $t_1\equiv u_1$ being trivial, $\tau_1$ contains $u_1$. Removing $\tau_1$, the family
$\nu_1, \nu_2, \cdots, \nu_{l}, \tau_2, \cdots, \tau_k$
 freely  generates the fundamental group $\pi_1(\Sigma, \sigma^*)$.
 We let $\GGG^*$ denote the dual graph with the loop $\tau_1$ removed.

Up to an homotopy equivalence, we may glue each boundary component $\beta_j$ to the loop $\gamma_j$ based at the base point $u$ of $U$ in such a way that $u_j$ is glued to $u$. Then all the $u_j$ become identified with $u$ in $V$ and each $t_j$ (for  $j=2, \cdots, k$) becomes a loop.
 Performing this identification in $\Sigma$ produces the graph $\GGG$ obtained  from $\GGG_1$ by identifying the $u_j$ all together.
 The graph $\GGG^*$ remains unchanged.

Turning to the fundamental group based at $u$ and seeing the quotient space $V=(U\disjointunion\Sigma)/_{\{\beta_j=\gamma_j\}} $ as a graph-of-spaces obtained by first gluing $\beta_1$ (furnishing an amalgam) and then in succession the other $\beta_j$ (iteratively furnishing HNN-extensions with free letter $t_j$), it follows that $\pi_1(V,u)$ is described as a graph-of-groups on two vertices, with vertex groups $\pi_1(U,u)$ and $\pi_1(\Sigma, u_1)$, and with $k$ edges joining the two vertices with edge groups isomorphic to $\langle b_j\rangle$.

By assumption, each $q_*(\gamma_j) = q_*(\beta_j)$ has infinite order in $\pi_1(U,u)$, hence the edge groups inject into the vertex groups and thus both vertex groups inject into $\pi_1(V,u)$ by \cite[Theorem 11, Corollary 1]{Serre}.
Then $\pi_1(V)$ is generated above $ q_*\pi_1(U)$ by adding $\{a_i : i \leq l\}$ and $\{t_j : 2\leq  j \leq k\}$ with a unique relation
$R(a_i,t_j^{-1}\gamma_j t_j)$.
\end{proof}

We now discuss the relevant geometry of the Cayley graph of $\pi_1(V)$ obtained by adding generators $a_i, t_j$ as above to $q_*\pi_1(U)$, with the goal of eventually applying this reasoning measurably across orbits of an action of $\pi_1(V)$ on a standard probability space to prove the proposition.

On each connected component of the pull back  of $U$ and the interior $\Sigma^{\circ}$ of $\Sigma$, under the universal covering
 $p:\widetilde{V}\to V$, the map $p$ coincides with a regular covering $p_{U}:\widehat{U}\to U$ and $p_{\Sigma^{\circ}}:\widehat{\Sigma^{\circ}}\to\Sigma^{\circ}$ respectively.

 Indeed, by injection of the $\pi_1$'s and of the $\langle  t_j\rangle $ ($2\leq j\leq k$)
  in $\pi_1(V)$, the coverings $p_{U}$ and $p_{\Sigma^{\circ}}$ are universal coverings
  and
   $p^{-1}(U)$ and $p^{-1}(\Sigma^{\circ})$ is precisely a disjoint union of copies of $\widehat{U}$ and $\widehat{\Sigma^{\circ}}$ respectively.
These copies form the vertex spaces of the bipartite tree-of-spaces decomposition of $\widetilde{V}$ (mimicking the Bass-Serre tree of the graph-of-groups).
 Observe that on each connected component of $p^{-1}(\Sigma^{\circ})$ the map $p_{\Sigma^{\circ}}$ extends to a copy $p_{\Sigma}:\widehat{\Sigma}\to \Sigma$ of the universal covering.

Again by injectivity, the boundary components $\beta_j$ pull-back to disjoint families of bi-infinite paths in the closure of each component of $p^{-1}(\Sigma^{\circ})\subseteq \tilde{V}$.
Observe that the pull-back graph $\widehat{\GGG}=p_{\Sigma}^{-1}(\GGG)$ in $\widehat{\Sigma}$ connects the various boundary components $p_{\Sigma}^{-1}(\{\beta_1, \cdots, \beta_k\})$, but in a non unique way as witnessed by the relation $R(a_i,t_j^{-1}\gamma_j t_j)$, i.e.,
by the boundary cycle of each of the pull-backs of the $2$-cell $\sigma$ under $p_{\Sigma}$.
These boundary cycles form a $2$-basis of the planar graph
$\widehat{\GGG}=p_{\Sigma}^{-1}(\GGG)$ in $\widehat{\Sigma}$, and the dual graph $\widehat{\GGG}^*=p_{\Sigma}^{-1}(\GGG^*)$ is its dual with respect to that $2$-basis.
As in previous arguments, selecting a one-ended spanning subforest $\TTT^*$ of $\widehat{\GGG}^*$ and removing all the edges of $\widehat{\GGG}$ that cross an edge of $\TTT^*$ resolves this non-uniqueness issue since the resulting subgraph $\TTT$ of $\widehat{\GGG}$ remains connected but becomes acyclic. It contains the boundary lines $p_{\Sigma}^{-1}(\{\beta_1, \cdots, \beta_k\})$.
What we are really after is the complement $$\TTT_0:=\TTT\mysetminus p_{\Sigma}^{-1}(\{\beta_1, \cdots, \beta_k\}),$$
since this is the part of $\TTT$ that lies outside of $p^{-1}(U)$.

Observe that even though some boundary components $\beta_i, \beta_j$ may be identified in $V$ for instance when $\gamma_i=\gamma_j$, their pull-backs in each $\widehat{\Sigma}$ are different (again by injectivity of $\pi_1(\Sigma)$ in $\pi_1(V)$ and $\widehat{\Sigma}=\widetilde{\Sigma}$).

Selecting a one-ended spanning subforest in $p^{-1}(\GGG^*)$ (that is, in each copy of $p_{\Sigma}^{-1}(\GGG^*)$)
and removing the edges it crosses as above in each connected component $\widehat{\Sigma}$ of $p^{-1}(\Sigma)$
leads to a subspace $\widetilde{V}^t$ of $\widetilde{V}$ where each vertex space of type $\widehat{\Sigma}$ is now replaced by a tree connecting its boundary components.

We claim that ``$\widetilde{V}^t$ is a tree above the $p^{-1}({U})$.''
More precisely, contracting each component of type $\widehat{U}$
to a point
produces a graph $\Psi$  which we claim is a tree.   Indeed, if $\theta$ is a cycle without backtracking in the resulting graph $\Psi$, we pull it back to the tree-of-spaces $\widetilde{V}^t$ and select there a backtracking point (entering a vertex space and leaving it by the same edge space). This cannot happen in a vertex space of type $\widehat{U}$ since  $\theta$ has no backtracking when this space is contracted to a point. This cannot happen  in a vertex space of type $\widehat{\Sigma}$ since these are now forests. It follows that $\theta$ is trivial.

We collect some useful observations about the dual graph $\widehat{\GGG}^*=p_{\Sigma}^{-1}(\GGG^*)$:
\begin{enumerate}
\item
  $\widehat{\GGG}^*$ is dual to the pull-back $\widehat{\CC}=p_{\Sigma}^{-1}(\CC)$ of the cellulation $\CC$ of $\Sigma$,
\item
  $\widehat{\GGG}^*$ is planar like $\widehat{\Sigma}$ (in fact it is a tree), \label{G-star planar}
\item
  $\widehat{\GGG}^*$ is connected (since $\GGG^*$ in $\Sigma$ generates the fundamental group of $\Sigma$), and
\item
  $\widehat{\GGG}^*$ is quasi-isometric with the covering group $\pi_1(\Sigma)$ of $\widehat{\Sigma}$.
In particular, it is not 2-ended. \label{G-star superquadratic}
\end{enumerate}

Finally, towards proving Theorem~\ref{th: pi1 V meas str. treeable over U},  choose a  base point $\tilde{u}$ in $\widetilde{V}$ above $u$ and suppose that $\Gamma = \pi_1(V)$ acts in a free, Borel fashion on the standard probability space $(X,\mu)$.  We will find a $\mu$-a.e.~ acyclic subgraph $\TT_0$ of the Borel graph given by the $a_i,t_j$ such that
$\RR_{\pi_1(V)}=\RR_{q_*(\pi_1(U))}*\TT_0$.  Consider the quotient by the diagonal action $\mathcal{V}={\Gamma}\backslash (X\times \widetilde{V})$.
Using the isomorphism $X\simeq  X\times \{\tilde{u}\}$, we obtain on $X$ the Borel graph $\Ggraph= {\Gamma}\backslash (X\times p^{-1}(\GGG))$ which corresponds to the graphing $\Phi$ given by the action of the elements $\{a_i\}_{i\in [1\mathrel{{.}\,{.}}l]}  ,\{\gamma_j\}_{j\in [1\mathrel{{.}\,{.}}k]}, \{t_j\}_{j\in [2\mathrel{{.}\,{.}}k]}$ of $\Gamma$.
The dual graph  $p^{-1}(\GGG^*)\subseteq \widetilde{V}$ delivers a Borel graph $\Ggraph^*$ with the natural Borel probability measure $\mu^*$ on its vertices $X^{*}={\Gamma}\backslash (X\times p^{-1}(\sigma^*))$.  The pull-back from $\tilde{u}$ in $\widetilde{V}$ of a simple path in $\sigma\subseteq \Sigma\subseteq V$ from $u$ to $\sigma^*$ entails a Borel isomorphism $X\overset{\sim}{\to} X^*$ used to push forward the measure $\mu$ to $\mu^*$.
By condition \eqref{G-star superquadratic} above $\Ggraph^*$ is nowhere $2$-ended and it comes with a Borel $2$-basis by \eqref{G-star planar}.
By Corollary~\ref{cor:planarforest} (or Theorem~\ref{thm:acyclic} since $\Ggraph^*$ is in fact acyclic),  $\Ggraph^*$ has a Borel $\mu^*$-a.e.\ one-ended spanning subforest $\TT^*$.

Applying the by now standard procedure above of removing from $\GG$ the edges that cross $\TT^*$  and those associated with the generators $\{\gamma_j\}_{j\in [1\mathrel{{.}\,{.}}k]}$
leads to an acyclic Borel graph $\TT_0$ above the equivalence relation given by the restriction of the action $\Gamma\actson X$ to the subgroup $q_*(\pi_1(U))\leq \Gamma$, giving
$$\RR_{\pi_1(V)}=\RR_{q_*(\pi_1(U))}*\TT_0$$
 $\mu$-almost everywhere.
 This completes the proof of Theorem~\ref{th: pi1 V meas str. treeable over U}.
\end{proof}

\begin{remark}
Observe in the proof of Theorem~\ref{th: pi1 V meas str. treeable over U} that the treeing $\TT_0$ eventually created is obtained by restricting the graphing given by the elements   $a_i,t_j\in \pi_1(V,u)=\langle \pi_1(U,u), a_i,t_j \,\vert\,  R(a_i,t_j^{-1}\gamma_j t_j)= t_1=1 \rangle$.

\end{remark}

\subsection{Back to elementarily free groups}
\label{sect: back to elementarily free groups}

\begin{proof}[Proof of Theorem~\ref{th: elem free gps as large towers}]
This Theorem follows mainly from the description (by hyperbolic $\omega$-residually free tower spaces) in \cite[Proposition 6]{SeVI}, except that in order to include all the elementarily free groups, one also needs to consider extended towers (as explained in \cite[Section 10.3, Definition 4.26, Corollary 7.10.5]{Guirardel-Levitt-Sklinos-2020}), i.e., a generalization of the IFL construction\footnote{Technically, there is another type of level (étage) in \cite{Guirardel-Levitt-Sklinos-2020} consisting in taking free products of $\pi_1(V_j)$ with $\Z$, but this kind of level can be moved down to $V_0$.}
 $V$ where $U$ is not necessarily connected:
The boundary components of $\Sigma$  are attached to (infinite order loops in) the connected components $U_1, U_2, \cdots, U_s$ of $U$
(see also Section~\ref{sect: extended large towers}).

Indeed,
 when $s\geq 2$, we can ``force'' $U$ to be connected by decomposing $\Sigma=\Sigma'\# S$ as a connected sum where $S$ is a sphere minus $s+1$ disks.
The space $V$ has a similar decomposition where we first attach one boundary component of $S$ to each $U_i$ and then attach
$\Sigma'$ via the connected sum to $S$ and along its remaining boundary components to the $U_i$ so as to recover our generalized IFL construction
(see \cite[Proposition 5.1 and its pictures]{Guirardel-Levitt-Sklinos-2020}).
Thus attaching $S$ amounts to connecting the $U_i$'s, with the result of taking the free product of the $\pi_1(U_i)$. The only problem would be if $\Sigma'$ would not satisfy the condition that $\chi(\Sigma')\leq -1$ (compare \cite[Remark 10.3]{Guirardel-Levitt-Sklinos-2020}
 where $\Sigma'$ must be non-exceptional).
This problem is ruled out for us by \cite[Remark 4.25]{Guirardel-Levitt-Sklinos-2020} which claims that (for extended towers producing elementarily free groups one can assume that
$g(\Sigma)\geq n_1$,
where $n_p$ is the number of the $U_i$'s to which $p$ boundary components of $\Sigma$ are attached and where $g$ denotes the genus.
Denoting by $k$ the number of boundary components, we have $k(\Sigma)=\sum_p p n_p$.
Concerning $\Sigma'$, we have
$g(\Sigma')=g(\Sigma)\geq n_1$ and $k(\Sigma')=k(\Sigma)+1-\sum_p n_p=1+\sum_p(p-1)n_p\geq 1$. It follows that $g(\Sigma')+k(\Sigma')\geq 3$ and $\chi(\Sigma')\leq -1$.
\end{proof}

\begin{proof}[Proof of Theorem \ref{elemfreetreeable}]
  Thanks to Theorem~\ref{th: elem free gps as large towers}, simply iterate Theorem~\ref{th: pi1 V meas str. treeable over U} through the levels of an IFL tower with desired fundamental group.
\end{proof}

\begin{remark}
Observe that we could  have used another characterization of the finitely generated elementarily free groups from
\cite[Corollary 7.10.6 or 7.10.7] {Guirardel-Levitt-Sklinos-2020}
as those groups $\Gamma$ for which there exists a free group $\FF$ such that $\Gamma*\FF$ is the fundamental group of a simple (= non-extended) hyperbolic $\omega$-residually free tower space in the sense of Sela \cite{SeVI} with ground spaces $V_0$ given by compact connected graphs with non-zero Euler characteristics.
Then applying Theorem~\ref{th: pi1 V meas str. treeable over U} for such a simple tower
and Theorem~\ref{thm:closure}-(5) to deduce the treeability properties of $\Gamma$ from those of $\Gamma*\FF$ would have delivered an alternative proof of  Theorem \ref{elemfreetreeable}.
However, the approach using Theorem~\ref{th: elem free gps as large towers}  with Theorem~\ref{th: pi1 V meas str. treeable over U} gives a little bit more: it gives the additional structure of successive measure strong free factors with treeable complements.
Compare Section~\ref{sect: extended large towers} and Remark~\ref{rem:no measure strong free factor for extended tower}.
\end{remark}

As in Question \ref{question:Hjorthsurface}, we ask whether a version of Hjorth's theorem \cite{Hjo-cost-att} holds for elementarily free groups:
\begin{question}
  Given a free group and an elementarily free group with the same cost, can every orbit equivalence relation of a free \pmp{} action of the former also be generated by a free action of the latter?
\end{question}

\subsection{Some cyclic measure strong free factors}\label{sec:cyclicmff}

We highlight two further applications of Theorem \ref{th: pi1 V meas str. treeable over U} for obtaining measure strong free factors of free products. One benefit of working in the setting of general Borel probability measures (instead of just the invariant ones) is that it provides access to the induced action construction even for subgroups of infinite index. We use this critically in the proof of the following lemma.

\begin{lemma}\label{lem:ABC}
Let $A$, $B$, and $C$ be countable groups with $A\leq B\leq C$. Suppose that $A$ is a measure strong free factor of $C$. Then $A$ is a measure strong free factor of $B$.
\end{lemma}

\begin{proof}
Let $B\curvearrowright X$ be a free Borel action of $B$ and let $\mu$ be a Borel probability measure on $X$. Consider the induced action of $C$, which is a free Borel action of $C$ on the space $Y=X\times C/B$, and let $\nu$ be a Borel probability measure on $X\times C/B$ in the same measure class as the product of $\mu$ with counting measure on $C/B$. This induced action has the property that the set $X\times \{ 1B \}$ is $B$-invariant and is a complete section for the equivalence relation $\mathcal{R}_{C}$ generated by $C$, and the projection $X\times \{ 1B\} \rightarrow X$ gives an isomorphism of the action $B\curvearrowright X\times \{ 1B \}$ with the original action of $B$ on $X$.

By hypothesis, after discarding a $\nu$-null set from $Y$, the equivalence relation $\mathcal{R}_A$ is a free factor of $\mathcal{R}_C$, and hence $\mathcal{R}_A$ is a free factor of $\mathcal{R}_B$ by \cite[Th{\'e}or{\`e}me~1]{Alvarez} (or \cite[Theorem~3.2]{Alv-Gab}). Since $X\times \{ 1B\}$ is a $B$-invariant subset of $Y$, this restricts to the desired decomposition of $(\mathcal{R}_B)_{\restriction (X\times \{ 1B\} )}$ with $(\mathcal{R}_A)_{\restriction (X\times \{ 1B \} )}$ as a free factor.
\end{proof}

Employing the notation from the previous sections, suppose that we perform a slightly more general gluing than the IFL construction in which we allow for some boundary components of the surface $\Sigma$ to not be glued to $U$. That is, we have pairwise disjoint boundary circles $\beta _1,\dots , \beta _k, \beta _{k+1},\dots , \beta _r$ of $\Sigma$ (where $k\geq 1$, $\pi _1(\Sigma )$ is a free group on at least two generators, and $\beta _j$ is based at $u_j\in \Sigma$) and loops $\gamma _1,\dots ,\gamma _k$ based at $u\in U$ which represent infinite order elements in $\pi _1 (U)$; for $j=1,\dots , k$ we glue $\beta _j$ to $\gamma _j$ so that $u_j$ is identified with $u$, and we do no gluing for $\beta _{k+1},\dots , \beta _r$. Let $V$ be the resulting space obtained after gluing.

Let $b_1,\dots , b_r\in \pi _1(\Sigma , u_1)$ be the loops based at $u_1$ corresponding to the composition of paths $b_j=t_j\beta _jt_j^{-1}$, where $t_j$ is a path from $u_1$ to $u_j$ and $t_1$ is trivial.

\begin{proposition}\label{prop:msff} For each $k+1\leq j\leq r$, the subgroup $\langle b_j\rangle$ is a measure strong free factor of $\pi _1(V,u)$.
\end{proposition}

\begin{proof} Consider the space $U'$ obtained as the wedge product of $(U,u)$ with $(r-k)$-many auxiliary circles $\gamma_{k+1},\dots, \gamma _r$. Let $V'$ be the space obtained by identifying $\beta _j$ with $\gamma _j$ for all $1\leq j\leq r$. Then by Theorem \ref{th: pi1 V meas str. treeable over U}, $\pi _1(U',u)$ is a measure strong free factor of $\pi _1 (V',u )$ with treeable complement. For each $k+1\leq j\leq r$, the subgroup $\langle \gamma _j \rangle$ is a free factor of $\pi _1(U',u)$, and hence $\langle \gamma _j\rangle$ is a measure strong free factor of $\pi _1 (V',u)$. Since $\gamma _j$ is conjugate to $b_j$ in $\pi _1 (V',u)$, it follows that $\langle b_j\rangle$ is a measure strong free factor of $\pi _1 (V',u)$.

The inclusion map $V\hookrightarrow V'$ is injective on $\pi _1$ (the group $\pi _1 (V')$ is obtained as an iterated HNN-extension of $\pi _1(V)\ast \FF _{r-k}$). Thus, for each $k+1\leq j \leq r$ we have that $\langle b_j\rangle \leq \pi _1(V,u)\leq \pi (V',u)$, and since $\langle b_j\rangle$ is a measure strong free factor of $\pi _1 (V',u)$, it is a measure strong free factor of $\pi _1(V,u)$ by Lemma \ref{lem:ABC}.
\end{proof}

\begin{corollary}\label{cor:gtht^{-1}}
Let $\Gamma \ast \langle t \rangle$ be the free product of a countable group $\Gamma$ and an infinite cyclic group $\langle t\rangle$. Let $\gamma$ and $\delta$ be elements of $\Gamma$ of infinite order. Then $\gamma t\delta t^{-1}$ generates a measure strong free factor of $\Gamma \ast\langle t\rangle$.
\end{corollary}

\begin{proof}
Let $(U,u)$ be a pointed CW complex with $\pi _1(U,u)=\Gamma$, and identify $\gamma ,\delta \in \Gamma$ with loops based at $u\in U$. Let $\Sigma$ be a sphere with three disjoint boundary circles $\beta _1,\beta _2,\beta _3$, with $\beta _j$ based at $u_j$. Let $t_j$ and $b_j$ be as above, so that $b_3=b_1b_2$ in $\pi _1(\Sigma ,u_1)$. Glue $\beta _1$ to $\gamma$, glue $\beta _2$ to $\delta$, and leave $\beta _3$ unglued; denote the resulting space by $V$. Let $t$ be the image of $t_2$ in $V$ so that $t$ is a loop based at $u$. Then $\pi _1 (V,u)=\Gamma \ast \langle t\rangle$, and in $\pi _1(V,u)$ we have the identities $b_1=\gamma$, $b_2=t\delta t^{-1}$, and $b_3=b_1b_2=\gamma t\delta t^{-1}$. Proposition \ref{prop:msff} then shows that $\langle b_3\rangle =\langle \gamma t\delta t^{-1}\rangle$ is a measure strong free factor of $\pi _1(V,u)=\Gamma \ast \langle t\rangle$.
\end{proof}

\begin{corollary}\label{cor:g0g1}
Let $\Gamma _0$ and $\Gamma _1$ be countable groups, and let $\gamma _0 \in \Gamma _0$ and $\gamma _1\in \Gamma _1$ be elements of infinite order. Then $\gamma _0\gamma _1$ generates a measure strong free factor of $\Gamma _0\ast \Gamma _1$.
\end{corollary}

\begin{proof}
Let $\langle t\rangle $ be an infinite cyclic group, and let $\varphi$ be the automorphism of the free product $\Gamma _0\ast \Gamma _1\ast \langle t\rangle$ with $\varphi (\delta )=t^{-1}\delta t$ for all $\delta \in \Gamma _1$ and which is the identity on both $\Gamma _0$ and $\langle t\rangle$. Corollary \ref{cor:gtht^{-1}} implies that $\gamma _0t\gamma _1t^{-1}$ generates a measure strong free factor of $\Gamma _0\ast \Gamma _1 \ast \langle t\rangle$, and hence $\varphi (\gamma _0t\gamma _1t^{-1}) = \gamma _0\gamma _1$ generates a measure strong free factor of $\Gamma _0\ast \Gamma _1\ast \langle t \rangle$. Therefore, $\gamma _0\gamma _1$ generates a measure strong free factor of $\Gamma _0\ast \Gamma _1$ by Lemma \ref{lem:ABC}.
\end{proof}

Specializing to free groups, Corollaries \ref{cor:gtht^{-1}} and \ref{cor:g0g1} provide strengthenings of the measure free factor results obtained in \cite{Alonso-2014}.

\begin{corollary}\label{cor:freegroupmsff}
If $\F_2=\langle a,b\rangle$ is a free group, then for $p\not=0, n\not=0$ any element of the form $a^p b^n$ or of the form $a^p b a^n b^{-1}$ generates a measure strong free factor of $\F _2$. More generally, in $\F_r=\F_{r-1}*\langle t\rangle$, every element of the form $vtwt^{-1}$, with $v,w \in \F _{r-1}$ nontrivial, generates a measure strong free factor of $\F _r$, and in $\F _r = \F _{r-k}*\F _k$, every element of the form $vw$ with $v\in \F _{r-k}$ and $w\in \F _k$ both nontrivial, generates a measure strong free factor of $\F _r$.
\end{corollary}

\section{\texorpdfstring{Ergodic dimension of aspherical $n$-manifolds}{Ergodic dimension of aspherical n-manifolds}}

The \define{ergodic dimension of a group $\Gamma$} \cite[D\'ef. 6.4]{Gab02} is the smallest geometric dimension of the orbit-equivalence relations $\RR_{\alpha}$ among all of its free \pmp{} actions  $\Gamma\actson ^{\alpha}(X,\mu)$ (see \cite{Gab-erg-dim} for more on this notion).
The goal of this section is to prove the following theorems:
\begin{thm}\label{erg dim d-mfld}
  Suppose $\Gamma $ is the fundamental group of a compact aspherical connected manifold
  $M$ (possibly with boundary) of
  dimension at least $2$. Then all free \pmp{} actions of $\Gamma$ have
  ergodic dimension at most $\dim(M) - 1$.
\end{thm}

Recall that $M$ is \define{aspherical} if its universal cover $\tilde{M}$ is contractible.
A manifold is \define{closed} if it is compact without boundary.

\begin{thm}\label{erg dim 3-dim mfld}
  Suppose $\Gamma$ is the fundamental group of a closed aspherical connected manifold of
  dimension $3$. Then either
  \begin{enumerate}
  \item $\Gamma$ is amenable, or
  \item $\Gamma$ has strong ergodic dimension $2$.
  \end{enumerate}
\end{thm}

\subsection{Removing a dual one-ended subforest}
\label{subsect:remove dual of F}
We start by considering the effect of removing from a contractible complex the cells associated with a one-ended subforest of its dual complex.

Let $M$ be a smooth compact connected manifold of dimension $d$, possibly with boundary. Let $\bar\Sigma$ be a triangulation of $M$. It induces a triangulation of its boundary $\partial M$.
Pick a base point $u$ among the vertices of $\bar\Sigma$.
The induced  triangulation ${\Sigma}$ of the universal cover $\tilde{M}$ is invariant under the covering group $\pi_1(M,u)$, once a pull-back $\tilde{u}$ of $u$ is chosen.

A $(d-1)$-dimensional cell is called \define{singular} if it belongs to the boundary of $\tilde{M}$, i.e., if it is a face of a single $d$-dimensional cell.
A $(d-1)$-dimensional cell is called \define{regular} otherwise, i.e., if it is a face of exactly two $d$-dimensional cells.

We define the \define{dual graph} ${{\GGG}^*}$ of ${\Sigma}$.
Its set of vertices $\VVV({{\GGG}^*})$ is in bijection with the set ${\Sigma}^{(d)}$ of $d$-dimensional cells of ${\Sigma}$.
We denote the natural bijection by $$\Dual_d:\VVV({{\GGG}^*})\to {\Sigma}^{(d)}.$$
Its set of edges $\EEE({{\GGG}^*})$ is in bijection with the set ${\Sigma} ^{(d-1)}_r$ of regular $(d-1)$-dimensional cells of ${\Sigma}$ via
$$\Dual_{d-1}:  \EEE({{\GGG}^*})\to {\Sigma} ^{(d-1)}_r$$ where an edge $e$  joins two vertices $v_1, v_2$
 if and only if $\Dual_{d-1}(e)$  is a common face of the $d$-dimensional cells $\Dual_{d}(v_1), \Dual_{d}(v_2)$.

Let ${\mathtt{F}^{*}}$ be a one-ended spanning subforest of the dual graph ${{\GGG}^*}$ of ${\Sigma}$. Using the same notation as in Section~\ref{sect:Preliminaries}, we define the simplicial subcomplex ${\Sigma} \oPerpstar {{\mathtt{F}^{*}}}$ of ${\Sigma}$
 by removing all the $d$-dimensional cells from ${\Sigma}$ and all the $(d-1)$-dimensional cells
 corresponding to an edge in ${\mathtt{F}^{*}}$, i.e., \[{\Sigma} \oPerpstar {\mathtt{F}^{*}}:={\Sigma}\mysetminus \left(\Dual_{d}(\VVV({\mathtt{F}^{*}}))\cup \Dual_{d-1}(\EEE({\mathtt{F}^{*}}))\right).\]
Observe that ${\Sigma}$ and ${\Sigma} \oPerpstar {\mathtt{F}^{*}}$ have the same $k$-skeleton for $k\leq d-2$.

\begin{proposition}\label{Sigma-F is contractible}
If ${\Sigma} $ is contractible then the $(d-1)$-dimensional complex ${\Sigma} \oPerpstar {\mathtt{F}^{*}}$ is contractible.
\end{proposition}

\begin{proof}[Proof of Proposition \ref{Sigma-F is contractible}]
By Whitehead's theorem \cite{Whitehead-1949-1}, it is enough to show that for all integers $k\geq 0$, the $\pi_k({\Sigma}\oPerpstar {{\mathtt{F}^{*}}}, x)$ are trivial.

Let $I^k:=[0,1]^k$ be the $k$-dimensional cube and $\partial I^k$ its boundary.

Consider a continuous map $f:(I^k,\partial I^k)\to ({\Sigma}\oPerpstar {{\mathtt{F}^{*}}},x)$ (sending $\partial I^k$ to $x$) and consider its composition $j\circ f$ with the inclusion $j: {\Sigma}\oPerpstar {{\mathtt{F}^{*}}}\hookrightarrow {\Sigma} $.
Since $ {\Sigma} $ has trivial $\pi_k$, there is a homotopy $\tilde{H}: (I^k\times I, (\partial I^k)\times I) \to ({\Sigma} ,x)$
between $j\circ f(\cdot )=\tilde{H}(\cdot ,1)$ and the constant map $\tilde{H}(\cdot ,0): I^k\times\{0\}\to \{x\}$.

We define the compact subcomplex $K_0$ as the smallest subcomplex of ${\Sigma} $ that contains the image $\tilde{H}(I^k\times I)$.
We will show that $K_0$ is contained in a compact subcomplex $K\subseteq {\Sigma} $ which admits a retraction $U: K\to K\cap  ({\Sigma}\oPerpstar {{\mathtt{F}^{*}}})$, i.e., a continuous map whose restriction to $K\cap  ({\Sigma}\oPerpstar {{\mathtt{F}^{*}}})$ is the identity.
The composition $U\circ \tilde{H}: (I^k\times I, (\partial I^k)\times I) \to ({\Sigma}\oPerpstar {{\mathtt{F}^{*}}},x)$ will provide a homotopy in $({\Sigma}\oPerpstar {{\mathtt{F}^{*}}},x)$ connecting $f$ to the trivial map $I^k\to \{x\}$, thus proving that $\pi_k({\Sigma}\oPerpstar {{\mathtt{F}^{*}}},x)$ is trivial.

We think of the one-ended forest ${\mathtt{F}^{*}}$ as an oriented graph with the orientation pointing toward the single infinite end in each connected component.
Every vertex $v\in \VVV({\mathtt{F}^{*}})=\VVV({{\GGG}^*})$ has thus exactly one outgoing edge; let's call it $o(v)$.
The ${\mathtt{F}^{*}}$-back-orbit of $v$ is the unique finite component of ${\mathtt{F}^{*}}\mysetminus o(v)$. It contains $v$.
An edge $e\in \EEE({\mathtt{F}^{*}})$ is the out-going edge of a single vertex $\iota(e)$; the ${\mathtt{F}^{*}}$-back-orbit of $e$ is the ${\mathtt{F}^{*}}$-back-orbit of $\iota(e)$.
If $e$ is an edge in $\EEE({{\GGG}^*})\mysetminus \EEE({\mathtt{F}^{*}})$, its ${\mathtt{F}^{*}}$-back-orbit is set to be empty.
The \define{${\mathtt{F}^{*}}$-back-orbit saturation} of a finite set $A \subseteq \VVV({\mathtt{F}^{*}})\cup \EEE({\mathtt{F}^{*}})$ is the union of $A$ with the ${\mathtt{F}^{*}}$-back-orbit of all its elements.
It is also finite. A finite set is \define{${\mathtt{F}^{*}}$-back-orbit saturated} if it coincides with its ${\mathtt{F}^{*}}$-back-orbit saturation

We transfer these notions to the subcomplexes of ${\Sigma}$ via the maps $\Dual_{d}, \Dual_{d-1}$.
The ${\mathtt{F}^{*}}$-back-orbit of a $(d-1)$-cell $\tau$ of ${\Sigma}$ is the image under $\Dual_{d}$ and $\Dual_{d-1}$ of the ${\mathtt{F}^{*}}$-back-orbit saturation of $\Dual_{d-1}^{-1}(\tau)$.
The \define{${\mathtt{F}^{*}}$-back-orbit saturation} of a finite subcomplex $K_0\subseteq {\Sigma}$ is the union of $K_0$ with the ${\mathtt{F}^{*}}$-back-orbit of all its $(d-1)$-cells.
A finite subcomplex $K\subseteq {\Sigma}$ is \define{${\mathtt{F}^{*}}$-back-orbit saturated} if it coincides with its ${\mathtt{F}^{*}}$-back-orbit saturation.

\begin{lemma}\label{lem:back-orbit_retraction}
If $L$ is a compact ${\mathtt{F}^{*}}$-back-orbit saturated subcomplex of ${\Sigma} $, then it admits a retraction to $L\cap ({\Sigma}\oPerpstar {{\mathtt{F}^{*}}})$.
\end{lemma}

\begin{proof}
The key point is that any simplex admits a retraction to its boundary minus any one of its faces.
We then argue by a decreasing induction on the number of $d$-dimensional cells of $L$:
\\
Observe if $L$ has no $d$-dimensional cells then it has to be contained in ${\Sigma}\oPerpstar {{\mathtt{F}^{*}}}$.
\\
Otherwise, pick $\epsilon$ one of the $n\geq 1$ cells of dimension $d$ of $L$.
Consider in ${\mathtt{F}^{*}}$ the unique path $p$ from $\Dual_{d}^{-1}(\epsilon)$ to the infinite end.
By finiteness of $L$ there is a least vertex $v\in p\cap \Dual_{d}^{-1}(L)$.
By definition $\sigma:=\Dual_{d}(v)$ and $\tau:=\Dual_{d-1}(o(v))$ are cells of $L$.

Claim:
Removing $\sigma$ and $\tau$ from $L$ delivers a ${\mathtt{F}^{*}}$-back-orbit saturated subcomplex $L_{\sigma}\subseteq L$ of ${\Sigma} $ with one $d$-dimensional cell less, together with a retraction $U_{\sigma}: L\to L_{\sigma}$. 
This is because, by construction, the cell $\sigma$ doesn't belong to the back-orbit of any other $(d-1)$-dimensional cell of $L$ than $\tau$.
The cell $\sigma$ admits a retraction $V_{\sigma}$ to $(\partial \sigma)\mysetminus \tau$ (its boundary minus $\tau$), which extends (by the identity) on the rest of $L$.

Now, a decreasing induction gives successive retractions $U_{\sigma_1}, U_{\sigma_2}, \cdots, U_{\sigma_n}$ whose composition is a retraction from $L$ to a subcomplex contained in ${\Sigma}\oPerpstar {{\mathtt{F}^{*}}}$. This completes the proof of the lemma.\end{proof}

Applying Lemma \ref{lem:back-orbit_retraction} to the ${\mathtt{F}^{*}}$-back-orbit saturation $K$ of $K_0$ completes the proof of Proposition~\ref{Sigma-F is contractible}.
\end{proof}

\begin{lemma}\label{lem:growth cond for dual graph}
If $M$ is a compact connected smooth $d$-dimensional manifold (possibly with boundary) and if ${\Sigma} $ is a triangulation of its universal cover $\tilde{M}$ that is the pull-back of a (finite)  triangulation of $M$, then the dual graph ${{\GGG}^*}$ of ${\Sigma} $ is quasi-isometric to the fundamental group of $M$.
\end{lemma}

\begin{proof}
Indeed, ${{\GGG}^*}$ (as well as ${\Sigma} $) is equipped with a co-compact free action of $\pi_1(M)$.
It is enough to check that ${{\GGG}^*}$ is connected.
Given any two vertices of ${{\GGG}^*}$, they correspond to two $d$-dimensional cells $\sigma$ and $\tau$ of ${\Sigma} $.
Any path in $\tilde{M}$ joining two points $z_{\sigma}$ and $z_{\tau}$ of their interior (the interior of a connected manifold with boundary is path-connected) can be deformed (with fixed extremities) to a path in $\tilde{M}\mysetminus \partial\tilde{M}$  that avoids the cells of dimension $\leq d-2$
 and that  crosses the $(d-1)$-dimensional regular cells in finitely many points). Such a path delivers a path in ${{\GGG}^*}$ joining the two vertices.
\end{proof}

\subsection{Proof of Theorems~\ref{erg dim 3-dim mfld} and \ref{erg dim d-mfld}}

Let $\Gamma$ be the fundamental group of a compact aspherical smooth manifold $M$ of dimension $d\geq 2$, possibly with boundary, and let $\Gamma \actson^a (X,\mu)$ be a free \pmp{} action.

Consider  a triangulation ${\Sigma} $ of its universal cover $\tilde{M}$ which is the pull-back of a (finite) triangulation of $M$ and let ${{\GGG}^*}$ be its dual graph.
When equipped with the diagonal action, $\widetilde{{\Sigma} }:=X\times {\Sigma} $
gives a contractible field of complexes on which the equivalence
relation $\RR_a$ acts smoothly. Similarly, $\Gamma$ and $\RR_a$ act smoothly on the dual field of graphs $X\times {{\GGG}^*}$.
The quotient $\GG:=\Gamma\backslash (X\times {{\GGG}^*})$ gives a measure preserving Borel graph on the finite measure space \[\Gamma\backslash((X,\mu)\times \VVV({{\GGG}^*}))\simeq\Gamma\backslash\left((X,\mu)\times {\Sigma} ^{(d)}\right)\simeq (X,\mu)\times \Gamma\backslash{\Sigma} ^{(d)}.\]
Since the actions are free, almost every connected component of $\GG$ is isomorphic to ${{\GGG}^*}$.

If ${{\GGG}^*}$ is finite, so is $\Gamma$ and it has strong ergodic dimension $=0$.
If ${{\GGG}^*}$ is two-ended, then $\Gamma$ is  two-ended as well (by Lemma~\ref{lem:growth cond for dual graph}) and thus amenable, and we are done by Ornstein-Weiss Theorem \cite{OW80}: $\Gamma$ has strong ergodic dimension $=1$.

Otherwise, Theorem~\ref{thm:pmp} ensures the existence of an a.e.\  one-ended subforest ${\mathcal{F}}$ of $\GG$.

Pulling back ${\mathcal{F}}$ under the quotient map $X\times {{\GGG}^*} \longrightarrow \GG$ delivers a smooth $\RR_a$-invariant field ${\mathcal{F}^{*}}$ of one-ended subforests of ${\GG^{*}}:=X\times {{\GGG}^*}$.

Then consider the smooth $\RR_a$-invariant field of simplicial subcomplexes $\widetilde{{\Sigma} }\oPerpstar {{\mathcal{F}^{*}}}$ obtained by removing from $\widetilde{{\Sigma} }$ all the $d$-dimensional cells and all the $(d-1)$-dimensional cells associated with an edge from the forest ${\mathcal{F}^{*}}$. This construction is made fiber-wise and the bijections $\Dual_{d}$ and $\Dual_{d-1}$ extend naturally to the field of complexes framework. The resulting complex $(\widetilde{{\Sigma} }\oPerpstar {{\mathcal{F}^{*}}})_x$ above a.e.\  $x\in X$ is precisely the complex $\widetilde{{\Sigma} }_{x}\oPerpstar{{\mathcal{F}^{*}}_{x}}$, of the kind (${\Sigma} $ minus cells dual to a one-ended subforest) that has been considered in Section~\ref{subsect:remove dual of F}. Applying Proposition \ref{Sigma-F is contractible} shows that a.e.\  complex $(\widetilde{{\Sigma} }\oPerpstar{{\mathcal{F}^{*}}})_x$ is contractible (and of course $(d-1)$-dimensional).
We have proved Theorem~\ref{erg dim d-mfld}. \hfill$\square$

\bigskip
As for Theorem~\ref{erg dim 3-dim mfld} where $M$ is $3$-dimensional, we already know by Theorem~\ref{erg dim d-mfld} that $\Gamma$ has ergodic dimension $\leq 2$.\\
In case $\Gamma$ has ergodic dimension $=1$, then (by \cite[Cor. 3.17]{Gab02}) the $2$nd $\ell^2$-Betti number $\beta_{2}^{(2)}(\Gamma)=0$.
By Poincar\'e duality \cite[Sect. 5]{CG86} we deduce $\beta_1(\Gamma)=0$.
In case $M$ is non-orientable, the orientation cover $\bar{M}\to M$ has fundamental group $\bar{\Gamma}$, an index $2$ subgroup of $\Gamma$, and we get
 $2 \beta_{1}^{(2)}(\Gamma)= \beta_{1}^{(2)}(\bar{\Gamma})=\beta_{2}^{(2)}(\bar{\Gamma})=2 \beta_{2}^{(2)}(\Gamma)=0$.
By \cite[Prop. 6.10]{Gab02}, a group of ergodic dimension $=1$ with $\beta_1^{(2)}=0$ is amenable.
In case $\Gamma$ has ergodic dimension $=0$, then $\Gamma$ is finite, thus amenable.
Thus, if $\Gamma$ admits a \pmp{} free action of geometric dimension $<2$, then it is amenable.
Theorem~\ref{erg dim 3-dim mfld} is proved.
\hfill $\square$

\newpage
\begin{appendices}

\section{Treeability for locally compact groups}\label{sec:lcsc}
\label{sect: Treeability for locally compact groups}

\begin{definition}\label{def:BorelTreeable}
Let $E$ be a Borel equivalence relation on a standard Borel space $X$. We say that $E$ is \define{Borel treeable} if there is an acyclic Borel graph $\Tgraph$ on $X$ with $E_{\Tgraph} = E$. 
Given a Borel probability measure $\mu$ on $X$, we say that $E$ is \define{$\mu$-treeable} if there is a $\mu$-conull Borel subset $X_0$ of $X$ such that the restriction $E_{\restriction X_0}$ is Borel treeable. We say that $E$ is \define{measure treeable} if $E$ is $\mu$-treeable for every Borel probability measure $\mu$ on $X$.
\end{definition}

Let $G$ be a locally compact second countable (lcsc) group, and let
$G\curvearrowright X$ be a free Borel action of $G$ on a standard Borel
space $X$. A subset $Y\subseteq X$ is called a \define{cross section} of the action $G\curvearrowright X$ if $Y$ is a complete section for the action and if there exists a neighborhood $U$ of $1_G$ such that the map $U\times Y \rightarrow X$, $(g,y)\mapsto g\cdot y$, is injective.

\begin{theorem}[\cite{Ke92}]\label{thm:cross}
Let $G\curvearrowright X$ be a Borel action of an lcsc group $G$ on a standard Borel space $X$. Then there exists a Borel cross section for the action.
\end{theorem}

A cross section $Y\subseteq X$ is called \define{cocompact} if there exists a compact subset $K\subseteq G$ such that $K\cdot Y = X$.

\begin{theorem}[See {\cite[\S 2]{Sl17}}]
\label{thm:cocomp}
Let $G\curvearrowright X$ be a free Borel action of an lcsc group $G$ on a standard Borel space $X$. Then there exists a Borel cross section for the action which is cocompact. Moreover, given any compact set $L\subseteq G$, we may find a cocompact Borel cross section $Y\subseteq X$ such that $L\cdot y_0 \cap L\cdot y_1 = \emptyset$ for distinct $y_0,y_1\in Y$.
\end{theorem}

The analogous result was known before in the measure-theoretic setting \cite{Forrest-1974}.

\begin{proposition}\label{prop:qinv}
Let $G\curvearrowright X$ be a free Borel action of an lcsc group $G$ on a standard Borel space $X$ and let $\mu$ be a Borel probability measure on $X$. Then there exists a quasi-invariant Borel probability measure $\mu '$ such that $\mu$ and $\mu '$ have the same $G$-invariant null sets.
\end{proposition}

\begin{proof}
Let $m$ be a probability measure on $G$ which is equivalent to Haar measure. Define $\mu ' = m\ast \mu$, i.e., $\int _X f\, d\mu ' = \int _G \int _X g\cdot f \, d\mu \, dm$.
\end{proof}

\begin{definition}\label{def: B, mu, measure treeable}
Let $G\curvearrowright X$ be a free Borel action of an lcsc group $G$ on a standard Borel space $X$.
\begin{enumerate}
\item The action is called \define{Borel treeable} if the equivalence relation $\RR_G$ is Borel treeable.
\item Let $\mu$ be a Borel probability measure on $X$. The action is called
\define{$\mu$-treeable} if there exists a $G$-invariant $\mu$-conull Borel set $X_0\subseteq X$ such that $G\curvearrowright X_0$ is Borel treeable.
\item\label{def: measure treeable action} The action is called \define{measure treeable} if it is $\mu$-treeable for every Borel probability measure $\mu$.
\end{enumerate}
\end{definition}

\begin{proposition}
Let $G\curvearrowright X$ be a Borel action of an lcsc group $G$ on a standard Borel space $X$. Then the following are equivalent:
\begin{enumerate}
\item The equivalence relation $\RR_G$ is Borel treeable.
\item There exists a Borel cross section $Y$ such that $(\RR_{G})_{\restriction Y}$ is Borel treeable.
\item For every Borel cross section $Y$ the restriction $(\RR_{G})_{\restriction Y}$ is Borel treeable.
\end{enumerate}
\end{proposition}

\begin{proof}
(1)$\Rightarrow$(3): Let $\Tgraph$ be a Borel treeing of $\RR_G$ and let $Y$ be a Borel cross section for the action. For each $x\in X$ the set $[x]_{\RR_G} \cap Y$ is countable, so the set of points along a path from $x$ to $Y$ of minimal length is also countable. Therefore, the Borel set $\{ (x,z)\in \Tgraph \, : \, d_{\Tgraph}(z,Y)< d_{\Tgraph}(x,Y) \}$ has countable sections, so we may find a Borel function $f:X\mysetminus Y \ra X$ with $d_{\Tgraph}(f(x),Y) < d_{\Tgraph}(x,Y)$ for all $x\in X\mysetminus Y$. For each $x\in X$ let $n$ be least with $f^n(x)\in Y$, and let $\pi (x) = f^n(x)$. The set $A = \{ (x,z)\in \Tgraph \, : \, \pi (x)\neq \pi (z) \}$ is Borel, and the map $A\ra (\RR_{G})_{\restriction Y}$, $(x,z) \mapsto (\pi (x),\pi (z))$ is injective since $\Tgraph$ is a tree. Thus, the graph $\Tgraph _Y = \{ (\pi (x), \pi (z)) \, : \, (x,z)\in A \}$ is Borel, and $\Tgraph _Y$ is acyclic with $\RR_{\Tgraph _Y}=(\RR_{G})_{\restriction Y}$ since $\Tgraph$ is acyclic with $\RR_{\Tgraph} = \RR_G$.

(3)$\Rightarrow$(2) is immediate from Theorem \ref{thm:cross}. (2)$\Rightarrow$(1): Suppose that (2) holds and let $\Tgraph _Y$ be a Borel treeing of $(\RR_{G})_{\restriction Y}$. Since the set $\{ (x,y)\in \RR_G \, : \, x\in X\mysetminus Y  , \ y\in Y \}$ has countable sections we may find a Borel map $f: X\mysetminus Y \ra Y$. Then the graph $\Tgraph = \Tgraph _Y \cup \{ (x,f(x)) \, : \, x\in X\mysetminus Y \} \cup \{ (f(x),x) \, : \, x\in X\mysetminus Y \}$ is a Borel treeing of $\RR_G$.
\end{proof}

\begin{proposition}\label{prop:Ymeas}
Let $G\curvearrowright X$ be a Borel action of an lcsc group $G$ on a standard Borel space $X$ and let $\mu$ be a $G$-quasi-invariant Borel probability measure on $X$. Let $Y$ be a Borel cross section for the action. Then there exists a Borel probability measure $\nu$ on $Y$ which is $(\RR_{G})_{\restriction Y}$-quasi-invariant, and satisfies $\nu (A) = 0$ if and only if $\mu ([A]_{\RR_G}) = 0$, for all Borel $A\subseteq Y$. Moreover, if $\nu '$ is any other $(\RR_{G})_{\restriction Y}$-quasi-invariant Borel probability measure on $Y$ with this property then $\nu \sim \nu '$.
\end{proposition}

\begin{proof}
Fix an open precompact neighborhood $U$ of $1_G$ such that the map $U\times Y \rightarrow X$, $(g,y)\mapsto g\cdot y$, is injective. Since $G$ is lcsc we may find a sequence $(g_n)_{n\in \N}$ in $G$ such that $G=\bigcup _n g_nU$. Then $1= \mu (X) = \mu (\bigcup _n g_n U\cdot Y )$, so $\mu (g_nU \cdot Y)>0$ for some $n\in \N$ and hence $\mu (U\cdot Y)>0$ since $\mu$ is quasi-invariant. Let $\mu _0$ denote the normalized restriction of $\mu$ to $U\cdot Y$, and let $\nu$ be the Borel probability measure on $Y$ given by $\nu (A) = \mu _0(U\cdot A)$. Suppose that $\nu (A) =0$, so that $\mu (U\cdot A) =0$. Since $\mu$ is $G$-quasi-invariant we have $\mu (g_nU\cdot A)= 0$ for all $n\in \N$ and hence $\mu ([A]_{\RR_G} )\leq \sum _n \mu (g_nU\cdot A ) =0$. Conversely, if $A\subseteq Y$ is a Borel set with $\mu ([A]_G ) =0$, then $\nu (A) = \mu (U \cdot A)/ \mu (U\cdot Y ) \leq \mu ([A]_{\RR_G})/\mu (U\cdot Y ) = 0$. It follows that $\nu$ is $(\RR_{G})_{\restriction Y}$-quasi-invariant. The last statement is clear.
\end{proof}

In the case of a free \pmp{} action of a unimodular lcsc $G$, the following is well known (see \cite{Kyed-Petersen-Vaes-L2} for a detailed treatment).

\begin{proposition}\label{prop:covol}
Let $G\curvearrowright X$ be a free Borel action of a unimodular lcsc group $G$ on a standard Borel space $X$ and let $\mu$ be a $G$-invariant Borel probability measure on $X$. Fix a Haar measure $\lambda$ on $G$. Let $Y$ be a Borel cross section for the action. Then there is a unique $(\RR_{G})_{\restriction Y}$-invariant probability measure $\nu _Y$ on $Y$, and a unique value $0<\mathrm{covol}(Y)<\infty$ such that, for any neighborhood $U$ of $1_G$ as in the definition of cross section, the pushforward of $(\lambda_{\restriction U} )\times \nu _Y$ under the map $(g,y)\mapsto gy$ is equal to $\mathrm{covol}(Y)(\mu_{\restriction UY} )$.

Moreover, if $Y'$ is any other cross section for the action, then there exist cross sections $Y_0\subseteq Y$ and $Y_0'\subseteq Y'$, together with a measure preserving bijection $\varphi : (Y_0, \nu _{Y_0})\rightarrow (Y_0', \nu _{Y_0'})$ taking $\RR_{\restriction Y_0}$ to $\RR_{\restriction Y_0'}$, such that
\[
\frac{\nu _Y (Y_0)}{\nu _{Y'}(Y_0')} = \frac{\mathrm{covol}(Y)}{\mathrm{covol}(Y')}.
\]
\end{proposition}

Recall the \define{cost} of a \pmp{} countable Borel equivalence relation $\RR$ on $(Y,\nu )$ is defined to be the infimum of the $\nu _{\RR}$-measures of generating sets for $\RR$, where $\nu _{\RR}$ is the natural Borel measure on $\RR$ \cite{Gab00a}. The following definition also appears in \cite[Definition 3.1]{carderi2018asymptotic}.

\begin{definition}\label{def:lcsccost}
Let $G$ be a unimodular lcsc group, and fix a Haar measure $\lambda$ on $G$. Let $G\curvearrowright X$ be a free Borel action of $G$ on a standard Borel space $X$ and let $\mu$ be a $G$-invariant Borel probability measure on $X$. The \define{cost} of $\RR_G$ is defined by
\[
\mathrm{cost}(\RR_G) - 1 = \frac{\mathrm{cost}(({\RR_G})_{\restriction Y}) - 1}{\mathrm{covol}(Y)},
\]
where $Y$ is any cross section for the action.  Proposition~\ref{prop:covol} together with the induction formula for cost  \cite[Proposition II.6]{Gab00a} ensure that this value does not depend on the choice of cross section; it does depend on the choice of Haar measure, namely 
$\mathrm{cost}(\RR_G,a\lambda ) -1 = \frac{1}{a}(\mathrm{cost}(\RR_G,\lambda )-1)$,
for $a>0$.

The \define{cost} of $(G,\lambda )$ is defined to be the infimum of costs of orbit equivalence relations generated by free \pmp{} actions of $G$. 
We say $G$ has \define{fixed price} if all such orbit equivalence relations have the same cost (this is independent of the choice of Haar measure).
\end{definition}

\begin{example}
Let $G=\mathrm{Aut}(T_n)$ be the automorphism group of the $n$-regular tree $T_n$. Fix a vertex of $T_n$, let $K$ denote its stabilizer in $G$, and let $\lambda$ be the Haar measure on $G$ giving $K$ measure $1$. For any free Borel action $G\curvearrowright X$, modding out by $K$ delivers a treeable countable Borel equivalence relation $\RR_0$. A Borel transversal $Y$ for the action of $K$ on $X$ is then a cross section for the action of $G$, and $(\RR_{G})_{\restriction Y}$ is isomorphic to $\RR_0$, and thus treeable. Moreover, if $\mu$ is any $G$-invariant Borel probability measure on $X$, then the relations $\RR_0$ and $(\RR_{G})_{\restriction Y}$ are \pmp{} treeable of cost $n/2$, and $\mathrm{covol}(Y)=1$.  We conclude that $\mathrm{Aut}(T_n)$ is Borel strongly treeable, and that $(\mathrm{Aut}(T_n),\lambda )$ has fixed price $n/2$.
\end{example}

\begin{proposition}\label{prop:stronglytreeablefixedprice}
Let $G$ be a unimodular lcsc group that is strongly treeable. Then $G$ has fixed price.
\end{proposition}

\begin{proof}
Using cross sections (together with the fact that the pull back of a cross section of a free action under a factor map is a cross section) the proof follows as in \cite[Proposition VI.21]{Gab00a}.
\end{proof}

The cost of some lcsc groups is also considered in \cite{AM-21}.

\section{Treeability of groups and permanence properties}\label{sec:permanence}

\begin{definition}
\label{def:various treeab, groups}
Let $G$ be a locally compact second countable group.
\begin{enumerate}
\item We say that $G$ is \define{Borel strongly treeable} if every free Borel action of $G$ is Borel treeable.
\item We say that $G$ is \define{measure strongly treeable} if every free Borel action of $G$ is measure treeable.
\item We say that $G$ is \define{strongly treeable} if every free \pmp{} action $G\curvearrowright (X,\mu )$ of $G$ is $\mu$-treeable.
\item We say that $G$ is \define{treeable} if there exists some free \pmp{} action $G\curvearrowright (X,\mu )$ of $G$ which is $\mu$-treeable.
\end{enumerate}
\end{definition}
We renew our warning that "arborable" and "anti-arborable" in \cite{Gab00a} stands for what we call strongly treeable and non-treeable respectively here.

Let $\BSTreeable$, ${\MSTreeable}$, $\STreeable$, and $\Treeable$ denote the classes of Borel strongly treeable, measure strongly treeable, strongly treeable, and treeable lcsc groups respectively. Then clearly
\begin{equation*}\label{eqn:contain}
\BSTreeable\subseteq {\MSTreeable}\subseteq {\STreeable}\subseteq \Treeable .
\end{equation*}

\begin{question}\label{q:contain}
Which of these containments, if any, are strict?
\end{question}

\begin{theorem}\label{thm:closure} Let $\BSTreeable$, ${\MSTreeable}$, $\STreeable$, and $\Treeable$ be as above. Then:
\begin{enumerate}
\item $\BSTreeable$ contains all countable locally nilpotent groups and all locally compact compactly generated groups of polynomial growth.

\item ${\MSTreeable}$ contains all lcsc amenable groups.

\item ${\MSTreeable}$ contains $\Isom(\HH ^2)$, $\SL_2(\R )$, and $\PSL_2(\R )$.

\item Let $\ClassTreeable$ be one of $\BSTreeable$, ${\MSTreeable}$, or ${\STreeable}$. Suppose that
\[
1\ra K\ra G \ra L \ra 1
\]
is a short exact sequence where $K$ is compact and $L\in \ClassTreeable$. Then $G\in \ClassTreeable$.

\item $\BSTreeable$, ${\MSTreeable}$, and $\Treeable$ are closed under taking closed subgroups.
\\
If $G=G_1*G_2$ is countable and belongs to $\STreeable$ then so does $G_1$.

\item ${\STreeable}$ is closed under taking closed co-finite subgroups. In particular, if $G\in {\STreeable}$ then every lattice of $G$ is in ${\STreeable}$.

\item If $G$ has a closed co-finite subgroup in $\Treeable$, then $G\in \Treeable$. In particular, if $G$ has a lattice in $\Treeable$, then $G\in \Treeable$.

\item $\BSTreeable$, ${\MSTreeable}$, ${\STreeable}$ and $\Treeable$ are each closed under taking free products (of countable groups) with amalgamation over a finite subgroup.

\item $\BSTreeable$, ${\MSTreeable}$, and ${\STreeable}$ are each closed under taking HNN extensions (of countable groups) with respect to an isomorphism between two finite subgroups.

\item Let $\ClassTreeable$ be one of $\BSTreeable$, ${\MSTreeable}$, or ${\STreeable}$. Suppose that $G$ is the Bass-Serre fundamental group of a graph of countable groups in which each vertex group is in $\ClassTreeable$, and each edge group is finite. Then $G\in \ClassTreeable$.
\end{enumerate}
\end{theorem}

(1) follows from \cite{SS13} and \cite{JKL02}. Part (2) follows from \cite{CFW81}. Part (3) is Corollary \ref{cor:subgroup}. The proofs of the remaining facts are routine generalizations of results appearing in the literature: (4) follows immediately from the fact that Borel actions of compact groups are smooth, (5) generalizes \cite[Th. 5]
{Gab00a} and follows also from \cite[Proposition 3.3]{JKL02} (treeability for subrelations,
together with a standard induction $ H\backslash (X\times G)\curvearrowleft G$ -- choosing of a finite equivalent measure when necessary -- and restriction to the identity slice for $\BSTreeable$ and $\MSTreeable$);
 the "$\STreeable$ part" requires the use of \cite{Tor06} (any two free \pmp{} actions of $G_1$ and $G_2$ on $(X,\mu)$ can be realized by the restrictions of a free \pmp{} action of $G_1*G_2$ up to a conjugation of the $G_2$-action by an element of $\Aut(X,\mu)$).
(6) generalizes \cite[Th. VI.19 (i)]{Gab00a}, (7) follows from a standard induction argument, and (8), (9) and (10) generalize \cite[Prop. VI.10]{Gab00a} from the context of countable groups and \pmp{} actions (the proofs extend immediately to the Borel context; for $\Treeable$ in (8) use the diagonal action of the free product obtained from co-induction of treeable actions of the factors).

\end{appendices}

\bibliographystyle{alpha}

\begin{thebibliography}{CMTD}

\bibitem[Ada88]{Ada88}
{}S.~Adams.
\newblock Indecomposability of treed equivalence relations.
\newblock {\em Israel J. Math.}, 64(3):362--380, 1988.

\bibitem[Ada90]{Ada90}
{}S.~Adams.
\newblock Trees and amenable equivalence relations.
\newblock {\em Ergodic Theory Dynam. Systems}, 10(1):1--14, 1990.

\bibitem[AGT26]{Abert-Gab}
M.~Ab{\'e}rt, D.~Gaboriau and S. Tanushevski.
\newblock Higher dimensional cost and profinite actions.
\newblock In preparation, 2026+.

\bibitem[Alo14]{Alonso-2014}
J.~Alonso.
\newblock Measure free factors of free groups.
\newblock {\em Groups Geom. Dyn.}, 8(1):1--21, 2014.

\bibitem[Alv10]{Alvarez}
A.~Alvarez.
\newblock Th{\'e}or{\`e}me de Kurosh pour les relations d’{\'e}quivalence bor{\'e}liennes.
\newblock {\em Annales de l'Institut Fourier}, 60(4):1161-1200, 2010.

\bibitem[AlG12]{Alv-Gab}
A.~Alvarez and D.~Gaboriau.
\newblock Free products, orbit equivalence and measure equivalence rigidity.
\newblock {\em Groups Geom. Dyn.}, 6(1):53-82, 2012.

\bibitem[AM21]{AM-21}
M.~Ab{\'e}rt and S.~Mellick.
\newblock Point processes, cost, and the growth of rank in locally compact
  groups.
\newblock {\em Israel J. Math.}, 251(1):47--154, 2022.


\bibitem[AS90]{AS90}
S.R.~Adams and R.J.~Spatzier.
\newblock Kazhdan groups, cocycles and trees.
\newblock {\em Amer. J. Math.}, 112(2):271--287, 1990.

\bibitem[Bab97]{Bab97}
L.~Babai.
\newblock The growth rate of vertex-transitive planar graphs.
\newblock In {\em Proceedings of the Eighth Annual ACM-SIAM Symposium on Discrete
  Algorithms ({N}ew {O}rleans, {LA}, 1997)}, pages 564--573. ACM, New York, 1997.

 \bibitem[Bab77]{Babai-1977}
L. Babai.
\newblock Some applications of graph contractions.
\newblock {\em J. Graph Theory}, 1(2):125--130, 1977.
\newblock Special issue dedicated to Paul Tur{\'a}n.


\bibitem[BLPS01]{BLPS01}
I.~Benjamini, R.~Lyons, Y.~Peres, and O.~Schramm.
\newblock Uniform spanning forests.
\newblock {\em Ann. Probab.}, 29(1):1--65, 2001.

\bibitem[BS96]{Benjamini-Schramm-96}
I.~Benjamini and O.~Schramm.
\newblock Percolation beyond $\mathbf{Z}\sp d$, many questions and a few
  answers.
\newblock {\em Electron. Comm. Probab.}, 1:no.\ 8, 71--82 (electronic), 1996.

\bibitem[BTW07]{BTW-07}
M.R.~Bridson, M.~Tweedale, and H.~Wilton.
\newblock Limit groups, positive-genus towers and measure-equivalence.
\newblock {\em Ergodic Theory Dynam. Systems}, 27(3):703--712, 2007.

\bibitem[Car18]{carderi2018asymptotic}
A.~Carderi.
\newblock Asymptotic invariants of lattices in locally compact groups.
\newblock {\em C. R. Math. Acad. Sci. Paris}, 361:375--415, 2023.

\bibitem[CFW81]{CFW81}
A.~Connes, J.~Feldman, and B.~Weiss.
\newblock An amenable equivalence relation is generated by a single
  transformation.
\newblock {\em Ergodic Theory Dynam. Systems}, 1(4):431--450, 1981.

\bibitem[CG86]{CG86}
J.~Cheeger and M.~Gromov.
\newblock ${L}\sb 2$-cohomology and group cohomology.
\newblock {\em Topology}, 25(2):189--215, 1986.

\bibitem[CK77]{Connes-Krieger-1977}
A.~Connes and W.~Krieger.
\newblock Measure space automorphisms, the normalizers of their full groups,
  and approximate finiteness.
\newblock {\em J. Functional Analysis}, 24(4):336--352, 1977.

\bibitem[CM16]{CM16}
C.T.~Conley and B.D.~Miller.
\newblock A bound on measurable chromatic numbers of locally finite {B}orel
  graph.
\newblock {\em Math. Res. Lett.}, 23(6): 1633--1644, 2016.

\bibitem[CMTD16]{CMT-D}
C.T.~Conley, A.S.~Marks, and R.D.~Tucker-Drob.
\newblock Brooks' theorem for measurable colorings.
\newblock {\em Forum Math. Sigma}, 4, 2016.

\bibitem[Con76]{Connes-1976}
A.~Connes.
\newblock Classification of injective factors. {C}ases {II$_{1}$}, {II$_{\infty
  }$}, {III$_{\lambda }$}, {$\lambda \not=1$}.
\newblock {\em Ann. of Math. (2)}, 104(1):73--115, 1976.

\bibitem[Dro06]{Dr06}
C.~Droms.
\newblock Infinite-ended groups with planar {C}ayley graphs.
\newblock {\em J. Group Theory}, 9(4):487--496, 2006.

\bibitem[Dun85]{Dunwoody-1985}
M.J. Dunwoody.
\newblock The accessibility of finitely presented groups.
\newblock {\em Invent. Math.}, 81(3):449--457, 1985.


\bibitem[Dye59]{Dye59}
H.~Dye.
\newblock On groups of measure preserving transformations. {I}.
\newblock {\em Amer. J. Math.}, 81:119--159, 1959.

\bibitem[Dye63]{Dye63}
H.~Dye.
\newblock On groups of measure preserving transformations. {I}{I}.
\newblock {\em Amer. J. Math.}, 85:551--576, 1963.

\bibitem[EOSS24]{EOSS:hyperfiniteness}
S.~Kunnawalkam Elayavalli, K.~Oyakawa, F.~Shinko, and P.~Spaas.
\newblock Hyperfiniteness for group actions on trees.
\newblock {\em Proc. Amer. Math. Soc.}, 152(9):3657--3664, 2024.

\bibitem[Ele12]{El12}
G.~Elek.
\newblock Finite graphs and amenability.
\newblock {\em J. Funct. Anal.}, 263(9):2593--2614, 2012.

\bibitem[FM77]{FM77}
J.~Feldman and C.C.~Moore.
\newblock Ergodic equivalence relations, cohomology, and von {N}eumann
  algebras. {I}.
\newblock {\em Trans. Amer. Math. Soc}, 234(2):289--324, 1977.

\bibitem[For74]{Forrest-1974}
P.~Forrest.
\newblock On the virtual groups defined by ergodic actions of {$R^{n}$} and
  {${\bf Z}^{n}$}.
\newblock  {\em Advances in Math.},14:271--308, 1974.

\bibitem[Fre03]{Fremlin-vol2-2003}
D.~H.~Fremlin.
\newblock Measure theory, {V}ol.\ 2: {B}road foundations.
\newblock Torres Fremlin, Colchester, 2003.

\bibitem[Fur99]{Furman1999}
A.~Furman.
\newblock Orbit equivalence rigidity.
\newblock {\em Ann. of Math. (2)}, 150(3):1083--1108, 1999.

\bibitem[Fur11]{Furman2011}
A.~Furman.
\newblock A survey of measured group theory, geometry, rigidity, and group
  actions, 296--374.
\newblock {\em Chicago Lectures in Math., Univ. Chicago Press, Chicago, IL},
  2011.

\bibitem[Gab00]{Gab00a}
D.~Gaboriau.
\newblock Co\^ut des relations d'\'equivalence et des groupes.
\newblock {\em Invent. Math.}, 139(1):41--98, 2000.

\bibitem[Gab02a]{Gab02}
D.~Gaboriau.
\newblock Invariants {$L\sp 2$} de relations d'\'equivalence et de groupes.
\newblock {\em Publ. Math. Inst. Hautes \'Etudes Sci.}, 95:93--150, 2002.

\bibitem[Gab02b]{Gab02b}
D.~Gaboriau.
\newblock On orbit equivalence of measure preserving actions.
\newblock In {\em Rigidity in dynamics and geometry (Cambridge, 2000)}, pages
  167--186. Springer, Berlin, 2002.

\bibitem[Gab05]{Ga05}
D.~Gaboriau.
\newblock Examples of groups that are measure equivalent to the free group.
\newblock {\em Ergodic Theory Dynam. Systems}, 25(6):1809--1827, 2005.

\bibitem[Gab26]{Gab-erg-dim}
D.~Gaboriau.
\newblock On ergodic dimension.
\newblock In preparation, 2026+.

\bibitem[GL09]{GL}
D.~Gaboriau and R.~Lyons.
\newblock A measurable-group-theoretic solution to von {N}eumann's problem.
\newblock {\em Invent. Math.}, 177(3):533--540, 2009.

\bibitem[GLS20]{Guirardel-Levitt-Sklinos-2020}
V. Guirardel, G. Levitt, and R. Sklinos.
\newblock Towers and the first-order theories of hyperbolic groups.
\newblock {\em Mem. Amer. Math. Soc.}, 296(1477):v+111, 2024.

\bibitem[GN21]{Gab-top-dim}
D.~Gaboriau and C.~No{\^u}s.
\newblock On the top-dimensional $\ell^2$-{B}etti numbers.
\newblock {\em Annales de la Facult{\'e} des Sciences de Toulouse},
  30:1121--1137, 2021.

\bibitem[Gro93]{Gro93}
M.~Gromov.
\newblock Asymptotic invariants of infinite groups.
\newblock In {\em Geometric group theory, {V}ol.\ 2 ({S}ussex, 1991)}, volume
  182 of {\em London Math. Soc. Lecture Note Ser.}, pages 1--295. Cambridge
  Univ. Press, Cambridge, 1993.
  
  \bibitem[Hjo06]{Hjo-cost-att}
G.~Hjorth.
\newblock A lemma for cost attained.
\newblock {\em Ann. Pure Appl. Logic}, 143(1-3):87--102, 2006.

\bibitem[Hjo08]{Hj08b}
G.~Hjorth.
\newblock Non-treeability for product group actions.
\newblock {\em Israel J. Math.}, 163(1):383--409, 2008.

\bibitem[JKL02]{JKL02}
S.~Jackson, A.S.~Kechris, and A.~Louveau.
\newblock Countable {B}orel equivalence relations.
\newblock {\em J. Math. Log.}, 2(1):1--80, 2002.

\bibitem[Kai97]{Kai97}
V.A.~Kaimanovich.
\newblock Amenability, hyperfiniteness, and isoperimetric inequalities.
\newblock {\em C. R. Acad. Sci. Paris S\'er. I Math.}, 325(9):999--1004, 1997.

\bibitem[Kec92]{Ke92}
A.S.~Kechris.
\newblock Countable sections for locally compact group actions.
\newblock {\em Ergodic Theory Dynam. Systems}, 12(2):283--295, 1992.

\bibitem[KKR17]{KKR-2017}
J.~Koivisto, D.~Kyed, and S.~Raum.
\newblock Measure equivalence and coarse equivalence for unimodular locally
  compact groups.
\newblock {\em Groups Geom. Dyn.}, 15(1):223--267, 2021.

\bibitem[KM98]{KhM1}
O.~Kharlampovich and A.~Myasnikov.
\newblock Irreducible affine varieties over a free group. {I}. {I}rreducibility
  of quadratic equations and {N}ullstellensatz.
\newblock {\em J. Algebra}, 200(2):472--516, 1998.

\bibitem[KM04]{KM04}
A.S.~Kechris and B.D.~Miller.
\newblock {\em Topics in orbit equivalence}, Lecture Notes in
  Mathematics, 1852.
\newblock Springer-Verlag, Berlin, 2004.

\bibitem[Kne29]{Kneser-1929}
H.~Kneser.
\newblock Geschlossene {F}l{\"a}chen in dreidimensionalen {M}annigfaltigkeiten.
\newblock {\em Jahresbericht der Deutschen Mathematiker-Vereinigung},
  38:248--259, 1929.

\bibitem[KPV15]{Kyed-Petersen-Vaes-L2}
D.~Kyed, H.D.~Petersen, and S.~Vaes.
\newblock {$L^2$}-{B}etti numbers of locally compact groups and their cross
  section equivalence relations.
\newblock {\em Trans. Amer. Math. Soc.}, 367(7):4917--4956, 2015.

\bibitem[KST99]{KST99}
A.S.~Kechris, S.~Solecki, and S.~Todorcevic.
\newblock Borel chromatic numbers.
\newblock {\em Adv. Math.}, 141(1):1--44, 1999.

\bibitem[Lev95]{Le95}
G.~Levitt.
\newblock On the cost of generating an equivalence relation.
\newblock {\em Ergodic Theory Dynam. Systems}, 15(6):1173--1182, 1995.

\bibitem[LPS06]{LPS06}
R.~Lyons, Y.~Peres, and O.~Schramm.
\newblock Minimal spanning forests.
\newblock {\em Ann. Probab.}, 34(5):1665--1692, 2006.


\bibitem[Lyo09]{Lyons-2009-random-complexes}
R.~Lyons.
\newblock Random complexes and {$l^2$}-{B}etti numbers.
\newblock {\em J. Topol. Anal.}, 1(2):153--175, 2009.

\bibitem[Mil09]{Mil09}
B.D.~Miller.
\newblock Ends of graphed equivalence relations, {I}.
\newblock {\em Israel J. Math.}, 169:375--392, 2009.


\bibitem[Mil62]{Milnor-1962}
J.~Milnor.
\newblock A unique decomposition theorem for {$3$}-manifolds.
\newblock {\em Amer. J. Math.}, 84:1--7, 1962.

\bibitem[MvN36]{MvN36}
F.~Murray and J.~von~Neumann.
\newblock {On rings of operators.}
\newblock {\em Ann. of Math. (2)}, 37(1):116--229, 1936.

\bibitem[OW80]{OW80}
D.~Ornstein and B.~Weiss.
\newblock Ergodic theory of amenable group actions. i. {T}he {R}ohlin lemma.
\newblock {\em Bull. Amer. Math. Soc.(NS)}, 2(1):161--164, 1980.

\bibitem[Pem91]{Pemantle-1991-WSF}
R.~Pemantle.
\newblock Choosing a spanning tree for the integer lattice uniformly.
\newblock {\em Ann. Probab.}, 19(4):1559--1574, 1991.

\bibitem[Pop06]{Pop06a}
S.~Popa.
\newblock On a class of type {${\rm II}\sb 1$} factors with {B}etti numbers
  invariants.
\newblock {\em Ann. of Math. (2)}, 163(3):809--899, 2006.

\bibitem[Pop18]{Popa2018-vanishing}
S.~Popa.
\newblock On the vanishing cohomology problem for cocycle actions of groups on
  {$\rm II_1$} factors.
\newblock {\em Ann. Sci. \'{E}c. Norm. Sup\'{e}r. (4)}, 54(2):407--443, 2021.


\bibitem[PP00]{PP00}
R.~Pemantle and Y.~Peres.
\newblock Nonamenable products are not treeable.
\newblock {\em Israel J. Math.}, 118:147--155, 2000.

\bibitem[PV14a]{Popa-Vaes-2014-unique-Cartan-Fn}
S.~Popa and S.~Vaes.
\newblock Unique {C}artan decomposition for {$\rm II_1$} factors arising from
  arbitrary actions of free groups.
\newblock {\em Acta Math.}, 212(1):141--198, 2014.

\bibitem[PV14b]{Popa-Vaes-2014-unique-Cartan-hyperbolic-gp}
S.~Popa and S.~Vaes.
\newblock Unique {C}artan decomposition for {$\rm II_{1}$} factors arising from
  arbitrary actions of hyperbolic groups.
\newblock {\em J. Reine Angew. Math.}, 694:215--239, 2014.

\bibitem[Sel06]{SeVI}
Z.~Sela.
\newblock Diophantine geometry over groups. {VI}. {T}he elementary theory of a
  free group.
\newblock {\em Geom. Funct. Anal.}, 16(3):707--730, 2006.

\bibitem[Ser77]{Serre}
J.P.~Serre.
\newblock {\em Arbres, amalgames, {SL}{$\sb{2}$}}, volume~46 of {\em
  Ast\'erisque}.
\newblock S.M.F., Paris, 1977.

\bibitem[Slu17]{Sl17}
K.~Slutsky.
\newblock Lebesgue orbit equivalence of multidimensional {B}orel flows: A
  picturebook of tilings.
\newblock {\em Ergodic Theory Dynam. Systems}, 37(6):1966--1996, 2017.

\bibitem[SS13]{SS13}
S.~Schneider and B.~Seward.
\newblock Locally nilpotent groups and hyperfinite equivalence relations.
\newblock {\em Math. Res. Lett.}, 31(2):511--578, 2024.

\bibitem[Tho80]{Tho80}
C.~Thomassen.
\newblock Planarity and duality of finite and infinite graphs.
\newblock {\em J. Combin. Theory Ser. B}, 29(2):244--271,
  1980.

\bibitem[T{\"o}r06]{Tor06}
A.~T{\"o}rnquist.
\newblock Orbit equivalence and actions of {$\FF_n$}.
\newblock {\em J. Symb. Log.}, 71(1):265--282, 2006.

\bibitem[Tim19]{Ti19}
A.~Tim{\'a}r.
\newblock Unimodular random planar graphs are sofic.
\newblock {\em Combin. Probab. Comput.}, 32(6):851--858, 2023.

\bibitem[Whi49]{Whitehead-1949-1}
J.H.C.~Whitehead.
\newblock Combinatorial homotopy. {I}.
\newblock {\em Bull. Amer. Math. Soc.}, 55:213--245, 1949.

\end{thebibliography}

\noindent Clinton T.~Conley, Department of Mathematical Sciences, Carnegie Mellon University, 5000 Forbes Ave., Pittsburgh, PA 15213-3890.  \email{clintonc@andrew.cmu.edu}
\medskip

\noindent Damien Gaboriau, Université de Lyon, CNRS, UMPA ENS de Lyon, 46,allée d’Italie 69364 Lyon Cedex 07, FRANCE \email{damien.gaboriau@ens-lyon.fr}
\medskip

\noindent Andrew S.~Marks, Department of Mathematics, 970 Evans Hall MC 3840,
Berkeley CA 94720. \email{marks@math.berkeley.edu}
\medskip

\noindent Robin D.~Tucker-Drob, Department of Mathematics, University of Florida, Gainesville, FL 32611 \email{r.tuckerdrob@ufl.edu}

\end{document}